\documentclass[11pt, A4paper]{article}

\usepackage[english]{babel}
\usepackage[T1]{fontenc}
\usepackage[utf8]{inputenc}
\usepackage[
  margin=2.5cm,
  includefoot,
  footskip=30pt,
]{geometry}
\usepackage{lmodern}
\usepackage{amsmath}
\usepackage{amssymb}
\usepackage{amsthm}
\usepackage{booktabs}
\usepackage{array}
\usepackage{graphicx}
\usepackage{cite}
\usepackage[dvipsnames]{xcolor}
\usepackage{tikz-cd}
\usetikzlibrary{ decorations.markings,angles}
\usepackage{stmaryrd}
\usepackage{mathrsfs}
\usepackage{extarrows}
\usepackage{mathtools}
\usepackage{enumerate}
\usepackage{adjustbox}
\usepackage[mathcal]{eucal}
\usepackage{hyperref}
\hypersetup{
  colorlinks,
        breaklinks=true,
        plainpages=false,%
        citecolor=blue,
        linkcolor=blue,
        urlcolor=blue,
        bookmarksopen=true,%
        bookmarksnumbered=false,%
        bookmarksdepth=5%
}
\usepackage{tabularx}
\usepackage[strict]{changepage}

\mathchardef\mhyphen="2D
\def\on{\operatorname}

\makeatletter
\providecommand{\leftsquigarrow}{%
  \mathrel{\mathpalette\reflect@squig\relax}%
}
\newcommand{\reflect@squig}[2]{%
  \reflectbox{$\m@th#1\rightsquigarrow$}%
}
\makeatother

\newcommand\noloc{%\colon for maps going right to left
  \nobreak
  \mspace{6mu plus 1mu}
  {:}
  \nonscript\mkern-\thinmuskip
  \mathpunct{}
  \mspace{2mu}
}

\definecolor{ao}{rgb}{0.0, 0.5, 0.0}

\usepackage{aliascnt}
\usepackage{cleveref}

\newtheorem{theorem}{Theorem}[section]

\newaliascnt{lemma}{theorem}
\newtheorem{lemma}[lemma]{Lemma}
\aliascntresetthe{lemma}

\newaliascnt{conjecture}{theorem}

\aliascntresetthe{conjecture}

\newaliascnt{proposition}{theorem}
\newtheorem{proposition}[proposition]{Proposition}
\aliascntresetthe{proposition}

\newaliascnt{corollary}{theorem}
\newtheorem{corollary}[corollary]{Corollary}
\aliascntresetthe{corollary}

\newtheorem{introthm}{Theorem}

\theoremstyle{definition}

\newaliascnt{construction}{theorem}
\newtheorem{construction}[construction]{Construction}
\aliascntresetthe{construction}

\newaliascnt{definition}{theorem}
\newtheorem{definition}[definition]{Definition}
\aliascntresetthe{definition}

\newaliascnt{notation}{theorem}
\newtheorem{notation}[notation]{Notation}
\aliascntresetthe{notation}

\newaliascnt{remark}{theorem}
\newtheorem{remark}[remark]{Remark}
\aliascntresetthe{remark}

\newaliascnt{example}{theorem}
\newtheorem{example}[example]{Example}
\aliascntresetthe{example}

\newtheorem{introex}{Example}

\newaliascnt{question}{theorem}

\aliascntresetthe{question}

\newaliascnt{problem}{theorem}

\aliascntresetthe{problem}

\title{Geometric models for the derived categories of Ginzburg algebras of $n$-angulated surfaces via local-to-global principles}
\author{Merlin Christ}
\date{\today}

\begin{document}
\maketitle

\abstract{We consider a class of relative $n$-Calabi--Yau dg algebras, referred to as relative Ginzburg algebras, associated with marked surfaces equipped with a decomposition into $n$-gons ($n$-angulation). We relate their derived categories to the geometry of the surface. Results include the description of a subset of the objects in the derived categories in terms of curves in the surfaces and their Homs in terms of intersections. The description of these derived categories as the global sections of perverse schobers greatly facilitates the construction of these geometric models, as the construction reduces to gluing local data. This approach may be considered as a generalized, algebraic analogue of matching sphere constructions appearing in the symplectic geometry of Lefschetz fibrations. Most results also hold for the perverse schobers defined over any commutative ring spectrum. 

As an application of the geometric model in the case $n=3$, we match certain $\on{Ext}$-groups in the derived categories of these relative Ginzburg algebras and the extended mutation matrices of a class of cluster algebras with coefficients, associated to multi-laminated marked surfaces by Fomin-Thurston.}

\tableofcontents

\section{Introduction}

We consider a class of relative Ginzburg algebras arising from marked surfaces equipped with an ideal triangulation, introduced in \cite{Chr22}. These dg algebras describe relative Calabi--Yau versions of $3$-Calabi--Yau Ginzburg dg algebras associated with quivers with potentials associated with triangulated marked surfaces by Labardini-Fragoso \cite{Lab09}. We continue the study of the derived categories of these relative Ginzburg algebras, using their description as the global sections of perverse schobers, as started by the author in \cite{Chr22}. More generally, we further consider the derived categories of relative higher Ginzburg algebras associated to marked surfaces equipped with an ideal $n$-angulation (meaning a decomposition into $n$-gons). 

The aim of this paper is to relate these derived categories directly to the geometry of the underlying surfaces. We describe a subset of the objects in terms of curves and their Homs in terms of intersections of the curves, which forms a partial geometric model for these categories. The main innovation is that we arrive at these geometric models by local-to-global arguments, by using the sheaf properties of perverse schobers. Our results substantially improve on previous results in the literature on partial geometric models for non-relative Ginzburg algebras associated with marked surfaces \cite{Qiu16,Qiu18,QZ19,IQZ20}. While we focus on one class of examples of derived categories, these methods may find further applications in the study of global sections of other classes of perverse schobers.

Perverse schobers refer to categorifications of perverse sheaves, introduced by Kapranov-Schechtman in \cite{KS14}. While the general theory of perverse schobers is still conjectural, there is a working definition of perverse schobers on surfaces with nonempty boundary in terms of constructible sheaves of stable $\infty$-categories defined on ribbon graphs, see \cite{Chr22}. Such perverse schobers are called parametrized by the ribbon graph. The derived category of the relative Ginzburg algebra associated with an $n$-angulated marked surface is equivalent to the category of global sections of a perverse schober on the surfaces, which is the limit of a certain finite diagram of stable $\infty$-categories. The perverse schober is parametrized by the dual ribbon graph of the given ideal $n$-angulations. Similar to the gluing of sections of sheaves, we are able to glue objects of the categories (global sections of the perverse schobers) from local data (local sections of the perverse schobers). This greatly reduces the complexity of constructing the partial geometric model. 

After a brief introduction to relative Ginzburg algebras of marked surfaces in \Cref{sec:1.1}, we describe their partial geometric model in \Cref{sec:1.2}. We comment on the relation of the geometric model to constructions in Fukaya categories of Lefschetz fibrations in \Cref{sec:1.3}. In \Cref{sec:1.4}, we describe an application of the geometric model to the categorification of the extended mutation matrices of a class of cluster algebras with coefficients associated to marked surfaces.

\subsection{Relative Ginzburg algebras of \texorpdfstring{$n$}{n}-angulated surfaces}\label{sec:1.1}

Ginzburg algebras are a class of $3$-Calabi--Yau dg algebras, associated to quivers with potential, first considered by Ginzburg in \cite{Gin06}. Their derived categories have been used for, among other things, the categorification of cluster algebras, see \cite{Kel12} for a survey, and the algebraic description of Fukaya categories \cite{Smi15,Smi21,SW23}. Particularly relevant for this work are a class of Ginzburg algebras obtained from a quiver with potential constructed from a marked surface with an ideal triangulation. We also allow marked surface with marked points in the interior, called punctures. 

We are further concerned with a generalization of this class of Ginzburg algebras which are called relative Ginzburg algebras, introduced in \cite{Wu21,Chr22}. A relative Ginzburg algebra can be associated to an ice quiver with potential, meaning a quiver with the extra datum of a frozen subquiver and a potential. In the case of Ginzburg algebras arising from triangulated marked surfaces, the difference to the non-relative versions is that these dg algebras incorporate additional information about the boundary of the surface; the frozen vertices of the ice quiver are given by the boundary edges of the triangulation. As demonstrated in \cite{Chr22}, such relative Ginzburg algebras can be glued from smaller relative Ginzburg algebras in the same way that the underlying marked surfaces can be glued from ideal triangles. These gluing properties are realized in loc.~cit.~in terms of the description of the derived $\infty$-categories of these relative Ginzburg algebras as the $\infty$-categories of global sections of perverse schobers. 

The first result of this paper is the following theorem, describing a generalization of the perverse schober description to include the derived $\infty$-category of the relative (higher) Ginzburg algebra $\mathscr{G}_\mathcal{T}$ associated with a marked surface equipped with an $n$-valent spanning graph $\mathcal{T}$, for some $n\geq 2$. The graph $\mathcal{T}$ is dual to an ideal $n$-angulation $\mathcal{T}$ of ${\bf S}$, meaning a decomposition of the surface into $n$-gons with vertices at the marked points.  

\begin{introthm}[\Cref{gluethm}]
Let ${\bf S}$ be a marked surface with an ideal $n$-angulation with dual $n$-valent spanning graph $\mathcal{T}$. There exists an equivalence of $\infty$-categories
\[
 \mathcal{H}(\mathcal{T},\mathcal{F}_\mathcal{T}(k))\simeq \mathcal{D}(\mathscr{G}_\mathcal{T})
\]
between the $\infty$-category of global sections of the perverse schober $\mathcal{F}_{\mathcal{T}}(k)$ and the derived $\infty$-category of the relative (higher) Ginzburg algebra $\mathscr{G}_\mathcal{T}$.
\end{introthm}

The integer $n$ describes the relative Calabi--Yau dimension of the relative Ginzburg algebra. The reason we consider $n$-angulated surfaces, rather than arbitrary polygonal surfaces with mixed-angulations, is that there exists a canonical, and particularly simple, choice of grading of the quiver underlying the relative higher Ginzburg algebra. The more general setting of surfaces with mixed-angulations is explored in subsequent joint work with Haiden and Qiu.

The $\mathcal{T}$-parametrized perverse schober $\mathcal{F}_\mathcal{T}(k)$ used for the description of $\mathcal{D}(\mathscr{G}_\mathcal{T})$ depends on the choice of commutative ground ring $k$. We will define a version of this perverse schober also over any base $\mathbb{E}_\infty$-ring spectrum $R$, denoted $\mathcal{F}_\mathcal{T}(R)$.

\subsection{The geometric model}\label{sec:1.2}

Fix a marked surface ${\bf S}$ with an ideal $n$-angulation with dual $n$-valent spanning graph $\mathcal{T}$ with $n\geq 2$. We develop a geometric model for a full subcategory of the $\infty$-category of global sections of the perverse schober $\mathcal{F}_\mathcal{T}(R)$. The case $R=k$ then gives the geometric model for a full subcategory of the derived $\infty$-category of the relative Ginzburg algebra $\mathscr{G}_\mathcal{T}$. The derived $\infty$-category of the non-relative Ginzburg algebra associated with the $n$-angulation arises as the full subcategory of $\mathcal{D}(\mathscr{G}_\mathcal{T})$ consisting of the global sections of $\mathcal{F}_\mathcal{T}(k)$ whose evaluation at the external edges of $\mathcal{T}$ vanish. Our results for $\mathcal{D}(\mathscr{G}_\mathcal{T})$ thus restrict to a geometric model for (a full subcategory of) the derived category of the non-relative Ginzburg algebra, and this overlaps considerably with results from the series of papers \cite{Qiu16,Qiu18,QZ19,IQZ20}. We however substantially extend these previous results in generality. For instance, the previous geometric models were restricted to the $k$-linear case, to marked surfaces without interior marked points and, for $n\not= 3$, to (a subset of) the finite modules.

For the geometric model, we consider certain homotopy classes of curves in ${\bf S}$, referred to as matching curves. Such matching curves have endpoints on the boundary of ${\bf S}$ or at the vertices of $\mathcal{T}$ and may also be closed. We associate a global section $M_{\gamma}^L$ of $\mathcal{F}_{\mathcal{T}}(R)$ to each matching datum $(\gamma,L)$ in the surface, which consists of a matching curve, and a 'local value' $L$, which in the case $R=k$ corresponds to a module over the graded polynomial algebra $k[t_{n-2}]$ with generator in degree $|t_{n-2}|=n-2$.  By choosing different values of $L$, we can realize in the geometric model classes of finite, perfect and even non-perfect $\mathscr{G}_\mathcal{T}$-modules. For example, we show that the direct summands (i.e.~projective modules associated to vertices of the quiver) of the relative Ginzburg algebra $\mathscr{G}_\mathcal{T}$ are equivalent to modules $M_{\gamma}^{k[t_{n-2}]}$ associated to matching data with local value the polynomial algebra itself. Two other distinguished choices of local value $L$ are $L=k$ the trivial module and $L=k[t_{n-2}^\pm]$ the module of Laurent polynomials.

For a certain subclass of the matching curves, called pure matching curves, we further describe the $R$-linear morphism objects (i.e.~derived Homs) between the $M_{\gamma}^L$'s in terms of multiple types of intersections of the curves, see \Cref{homthm,homthm2}. The simplest examples of the description of the endomorphisms in \Cref{homthm2} are the following:

\begin{introex}
Let $\gamma$ be a pure matching curve in ${\bf S}$ without self-intersections. Let $L=k\in \mathcal{D}(k[t_{n-2}])$ be the trivial module.
\begin{enumerate}
\item[(1)] Suppose that $\gamma$ has both its endpoints at the vertices of $\mathcal{T}$. Then $M_{\gamma}^k$ is an $n$-spherical object, meaning that its derived endomorphisms are given by
\[ \on{REnd}(M_{\gamma}^k)\simeq k\oplus k[-n]\,.\]
\item[(2)] Suppose that $\gamma$ has one endpoint at a vertex of $\mathcal{T}$ and one endpoint on the boundary of ${\bf S}$. Then $M_{\gamma}^k$ is an exceptional object, meaning that
\[ \on{REnd}(M_{\gamma}^k)\simeq k\,.\]
\item[(3)] Suppose that $\gamma$ has both its endpoints on the boundary of ${\bf S}$. Then $M_{\gamma}^k$ is an $(n-1)$-spherical object, meaning that
\[ \on{REnd}(M_{\gamma}^k)\simeq k\oplus k[1-n]\,.\]
\end{enumerate}
See also \Cref{homex} for a precise statement.

The above $n$-spherical objects arise from modules over the non-relative Ginzburg algebras. The exceptional and $(n-1)$-spherical objects do not arise from modules over the non-relative Ginzburg algebra. Their natural interpretation is as algebraic analogues of non-compact Lagrangians appearing in partially wrapped Fukaya categories of Lefschetz fibrations, see also \Cref{sec:1.3} and \Cref{homex}. From this perspective, the exceptional objects correspond to Lefschetz thimbles.
\end{introex}

The results of \Cref{homthm,homthm2} apply to coefficients in an arbitrary $\mathbb{E}_\infty$-ring spectrum $R$, with the exception of the description of the endomorphisms of objects arising from closed matching curves. We do not investigate the question of the existence of exotic $S^1$-local systems for general $R$ and their endomorphisms. See also \Cref{table:main_results}, where we collect the main results of this paper and specify for which coefficients they are demonstrated.

{\renewcommand{\arraystretch}{2.1}
\begin{table}[h!]
\begin{center}
\begin{tabular}{c|c|c|c}
Type of result & Coefficients & $n$-angulations & Section\\
\hline
Definition of the perverse schober $\mathcal{F}_\mathcal{T}$ & arbitrary $R$ & $n\geq 2$ & \ref{sec:relGinzburg} \\
\begin{minipage}[t]{17em} Global sections via the relative Ginzburg algebra: $\mathcal{H}(\mathcal{T},\mathcal{F}_\mathcal{T})\simeq \mathcal{D}(\mathscr{G}_\mathcal{T})$\end{minipage} & comm.~ring $k$ & $n\geq 2$ & \ref{sec:relGinzburg} \\
\begin{minipage}[t]{17em} Assignment: $\text{matching data }(\gamma,L)\mapsto \text{global section}~M_{\gamma}^L$ \end{minipage}& arbitrary $R$ & $n\geq 2$ & \ref{sec:objectsfromcurves}\\
Flip equivalences on global sections & arbitrary $R$ & $n\geq 2$ & \ref{subsec:flips}\\
\begin{minipage}[t]{17em}\raggedright The derived Hom $\on{Mor}(M_{\gamma}^L,M_{\gamma'}^{L'})$ via intersections for $\gamma\not= \gamma'$ \end{minipage}& arbitrary $R$ & $n\geq 3$ & \ref{sec:homs}\\
\begin{minipage}[t]{17em}\raggedright $\on{Mor}(M_{\gamma}^L,M_{\gamma}^{L'})$ via intersections for $\gamma$ open \end{minipage}& arbitrary $R$ & $n\geq 3$ & \ref{sec:homs}\\
\begin{minipage}[t]{17em}\raggedright $\on{Mor}(M_{\gamma}^L,M_{\gamma}^{L'})$ via intersections for $\gamma$ closed \end{minipage}& comm.~ring $k$ & $n\geq 3$ & \ref{sec:homs}\\
\begin{minipage}[t]{17em}\raggedright $\on{Mor}(M_{\gamma}^{\phi^*(R)},M_{\gamma'}^{\phi^*(R)})$ via intersections \end{minipage}& arbitrary $R$ & $n=2$ & \ref{subsec:n=2}\\
$M_{\gamma}^L$ is indecomposable & field $k$ & $n\geq 3$ & \ref{sec:indecomposability}\\
The Jacobian algebra $H_0(\mathscr{G}_{\mathcal{T}})$ is gentle & comm.~ring $k$ & \begin{minipage}[t]{5em}$n\geq 3$, no\\ punctures\end{minipage} & \ref{subsec:Jacobian}\\
\begin{minipage}[t]{17em} Categorification of the extended mutation matrix \end{minipage} & comm.~ring $k$ & $n=3$ & \ref{subsec:categorical_description_extended_mut_matrix}
\end{tabular}
\caption{An overview over the main results and their respective setting. The coefficients are either in an arbitrary $\mathbb{E}_\infty$-ring spectrum $R$ or in a commutative ring or field $k$.}\label{table:main_results}
\end{center}
\end{table}
\renewcommand{\arraystretch}{1}}

For non-pure matching curves and $L$ arbitrary, we show in \Cref{ex:counter} that the morphism objects do not simply count intersections. An apparent exception is given by the extremal case $L=k$. We nevertheless exclude non-pure matching data with local value $L=k$ in our treatment, because a systematic description of the morphism objects would require dealing with gradings of the surface and the curves (which is not necessary when dealing with pure matching curves). 

Finally, we also express derived equivalences between relative Ginzburg algebras associated to flips of the $n$-angulations in the geometric model via rotations of parts of the surface, see \Cref{mutthm}.

We give two direct applications of the geometric model in this paper, further applications related to the categorification of cluster algebras are given in \cite{Chr22b}. The first is the description of the extended mutation matrices of a class of cluster algebras with coefficients associated to marked surfaces in terms of $\on{Ext}$-groups in $\mathcal{D}(\mathscr{G}_\mathcal{T})$, see below in \Cref{sec:1.4}. The second is the description of the homology algebra $H_\ast(\mathscr{G}_\mathcal{T})$ in the case that the surface has no punctures. We show in \Cref{jacprop4}, that $H_*(\mathscr{G}_\mathcal{T})$ is equivalent to the tensor product $\mathscr{J}_\mathcal{T}\otimes_k k[t_{n-2}]$, where $\mathscr{J}_\mathcal{T}=H_0(\mathscr{G}_\mathcal{T})$ is the Jacobian algebra. We will also see that $\mathscr{J}_\mathcal{T}$ is always a gentle algebra, generalizing an observation from \cite{ABCP10}. However, if ${\bf S}$ contains punctures, the Jacobian algebra may be infinite dimensional.
The homology algebras of non-relative Ginzburg algebras are generally more complicated; we thus propose to consider relative Ginzburg algebras of $n$-angulated surfaces as better behaved versions of the corresponding non-relative Ginzburg algebras, which encode the same or more information in their derived categories.

\subsection{Remarks on the relation to Fukaya categories of Lefschetz fibrations}\label{sec:1.3}

The handling of Fukaya categories comes with many analytical and geometric difficulties. For an efficient application of the rich intuition which Fukaya categories provide to typical representation theoretic questions, such as classifications problems, an algebraic approach to the construction of Fukaya categories is desirable. The emerging theory of perverse schobers seeks to give a higher categorical and algebraic approach to the construction of (some classes of) Fukaya categories and other Fukaya-type categories. 

For a typical instance where perverse schobers might be applied, consider a Lefschetz fibration $\pi\colon X\rightarrow {\bf S}$ from a suitable exact symplectic manifold $X$ to a surface equipped with a set of marked points $M\subset \partial {\bf S}$. There should exist a perverse schober on ${\bf S}$, parametrized by a spanning graph of ${\bf S}$, which can be described in terms of the typical fiber of the fibration and the vanishing cycles at the singular fibers. The category of global sections of the parametrized perverse schober should describe the partially wrapped Fukaya category of $X$ with stops which lie above $M$. In case that ${\bf S}$ is a disc, the resulting theory is supposed be similar to the approach to Fukaya categories of Lefschetz fibrations by Fukaya-Seidel categories, see \cite{Sei08}.

These expectations formed the main motivation for the perverse schober description of the derived category of a relative Ginzburg algebra of a triangulated surface in terms of a perverse schober in \cite{Chr22}. As shown by Ivan Smith \cite{Smi15}, the finite derived category of the corresponding non-relative Ginzburg algebra admits an embedding into the (non-wrapped) derived Fukaya category of a Calabi--Yau $3$-fold $Y$, equipped with a Lefschetz fibration to the surface. The typical fiber of the fibration is the cotangent bundle $T^*S^2$ of the $2$-sphere; the generic stalk of the perverse schober is hence given by its Ind-complete wrapped Fukaya category $\on{Ind}\mathcal{W}(T^*S^2)\simeq \mathcal{D}(k[t_1])$. Even though the derived categories of higher Ginzburg algebras of $n$-angulated surfaces have for $n>3$ so far not been related to Fukaya categories, we can nevertheless exhibit algebraic versions of many of the usual features possessed by partially wrapped Fukaya categories of Calabi--Yau $n$-folds equipped with Lefschetz fibrations to surfaces with typical fiber $T^*S^{n-1}$. 

Before we return to relative Ginzburg algebras, we briefly recall the geometric model for the derived categories of (graded) gentle algebras, proven in \cite{OPS18,BS19}. These derived categories are by \cite{HKK17,LP20} equivalent to the partially wrapped Fukaya categories of surfaces. These categories are further known as topological Fukaya categories and can be described as the global sections of constructible (co)sheaves on ribbon graphs embedded in the surfaces, see \cite{DK18,DK15}, which fit into the framework of parametrized perverse schobers. For this class of categories, it is clear how the geometric model, describing objects in terms of decorated curves, relates to the symplectic geometry. These curves themselves simply describe Lagrangians (half-dimensional submanifolds on which the symplectic form vanishes) inside the surface and thus objects in the Fukaya category. 

The geometric model for the derived category $\mathcal{D}(\mathscr{G}_\mathcal{T})$ of the relative Ginzburg algebra $\mathscr{G}_\mathcal{T}$ does not seek to describe the objects in terms of some half-dimensional subspaces of some speculative Calabi--Yau $n$-fold $Y$. Instead, given a Lagrangian $U\subset Y$ whose image under the Lefschetz fibration is a curve $\gamma$ in the surface, the corresponding object should be given by $M_{\gamma}^L\in \mathcal{D}(\mathscr{G}_\mathcal{T})$ (see \Cref{sec:1.2} for the notation), where $L\subset T^*S^{n-1}$ is the Lagrangian given by the typical fiber of the map $U\rightarrow \gamma$. Indeed, in our geometric model, $L$ is an object of $\mathcal{D}(k[t_{n-2}])$, which is by \cite{Abo11} equivalent to the $\on{Ind}$-complete wrapped Fukaya category of $T^*S^{n-1}$. Two particularly interesting choices of $L$ are the trivial $k[t_{n-2}]$-modules $k$, which corresponds to the Lagrangian zero-section of $T^*S^{n-1}$, and also $k[t_{n-2}]$, which corresponds to the Lagrangian fiber of the projection $T^*S^{n-1}\rightarrow S^{n-1}$. The other interesting choice $L=k[t_{n-2}^\pm]$, the module of Laurent polynomials, is a non-compact object and thus does not lie in the wrapped Fukaya category, only in its $\on{Ind}$-completion.

Let us further highlight the special case that a Lagrangian $U$ in $Y$ intersects the singular fibers of the Lefschetz fibration. In that case, if $U$ maps to a curve $\gamma$, this curve will end in a vertex of the parametrizing ribbon graph; which describe in the case $n=3$ the singular values of the Lefschetz fibration. If both endpoints of $\gamma$ are at singular values, then $U$ is a compact Lagrangian sphere and the corresponding object $M_{\gamma}^k\in \mathcal{D}(\mathscr{G}_\mathcal{T})$ is a spherical object, see \Cref{homex}. In this case, we thus algebraically recover the well-known construction of Lagrangian matching cycles, see \cite{Sei08}. 

\subsection{Categorification of the extended mutation matrix}\label{sec:1.4}

Ginzburg algebras can be used for the categorification of cluster algebras. Central in this regard is Amiot's quotient construction \cite{Ami09}, which defines the $2$-CY generalized cluster category as the cosingularity category of the Ginzburg algebra. There are also more direct links between Ginzburg algebras and the combinatorics of cluster algebras. For example, the mutation matrix of a cluster algebra can be recovered via the Euler-characteristics of the $\on{Ext}$-complexes of the simple $3$-spherical modules over the Ginzburg algebra associated to the vertices of the underlying quiver. This observation is made, formulated in the more general setting of cluster collections, in \cite[Section 8.1]{KS08}. As an application of the geometric model for relative Ginzburg algebras, we extend the relation between the mutation matrices and Ginzburg algebras to extended mutation matrices and relative Ginzburg algebras of triangulated surfaces. The extended mutation matrix consists of the mutation matrix and the $c$-matrix, the latter encodes the coefficients of the cluster algebra. 

We consider the class of cluster algebras with coefficients introduced in \cite{FT12}. The input is a marked surface with an ideal triangulation and a multi-lamination, that is a collection of laminations which are each a collection of certain disjoint curves in the surface, which we call lamination curves. The cluster variables of the cluster algebra are the lambda lengths, i.e.~certain coordinates on the decorated Teichm\"uller space of the surface. Fomin and Thurston define the $c$-matrix in terms of the shear coordinates of the surface, which describe a signed count of intersections of the laminations and the edges of the triangulation.

Each lamination curve $\gamma$ gives rise to a matching datum $(\gamma,k)$ with local value $L=k$ chosen to be the trivial representation $k\in \mathcal{D}(k[t_{1}])$, and thus to a global section $M_{\gamma}^k\in \mathcal{D}(\mathscr{G}_\mathcal{T})$. We can hence associate a $\mathscr{G}_\mathcal{T}$-module $M_{\lambda}^k$ to a lamination $\lambda$ as the direct sum of the modules associated to the lamination curves. The categorical model for the $c$-matrix consists of the Euler characteristics of the $\on{Ext}$-groups between the $3$-spherical modules associated to vertices of the quiver and the $M_{\lambda}^k$'s. Using our description of the morphism objects in terms of intersections, we can match the shear coordinates and the Euler characteristic of the $\on{Ext}$-groups, which gives us the following.

\begin{introthm}[\Cref{bmatthm}]\label{introthm3}
Consider a marked surface ${\bf S}$ with an ideal triangulation and a multi-lamination $\Lambda$. The extended mutation matrix of \cite{FT12} agrees with the categorical extended mutation matrix.
\end{introthm}

We wish to mention a related result for non-relative higher Ginzburg algebras of acyclic quivers. In \cite{KQ15}, it is shown that the Euler characteristics of the $\on{Ext}$-groups between the simple modules of the higher Ginzburg algebras match the combinatorics of colored quivers of \cite{BT09}.

\subsection*{Structure of the paper} 

In \Cref{sec:highercats}, we introduce some preliminaries from higher category theory, most importantly the Grothendieck construction. In \Cref{sec:schobers}, we recall the notion of a parametrized perverse schober from \cite{Chr22} and discuss the perverse schobers describing relative higher Ginzburg algebras. The main results of this paper appear in \Cref{sec:objsfromcurves,sec:homs}: we construct the global sections associated to curves and describe the morphisms objects in terms of intersections. In \Cref{subsec:Jacobian}, we observe that the Jacobian algebra of such a relative Ginzburg algebra $\mathscr{G}_\mathcal{T}$ is a gentle algebra and compute the homology algebra of $\mathscr{G}_\mathcal{T}$. In \Cref{subsec:flips}, we give the geometric description of the derived equivalences associated to flips of the $n$-angulations. Finally, in \Cref{sec:mutation} we apply the geometric model for relative Ginzburg algebras of triangulated surfaces to the categorification of the extended mutation matrices of cluster algebras of the surfaces.

\subsection*{Notation and conventions}

This paper is formulated in the language of stable $\infty$-categories, as developed in the extensive works \cite{HTT,HA}. We generally follow the terminology and notation used in \cite{HTT,HA}. We use the homological grading convention.

\subsection*{Acknowledgements}
I wish to thank my Ph.D.~advisor Tobias Dyckerhoff for his guidance and many helpful conversations. I further wish to thank Fabian Haiden for helpful discussions. Finally, I wish to thank Yu Qiu for giving me the opportunity to present the approach of this paper in a lecture course in spring 2022 at YMSC/BIMSA, Beijing. The author acknowledges support by the Deutsche Forschungsgemeinschaft under Germany’s Excellence Strategy – EXC 2121 “Quantum Universe” – 390833306. The author is a member of the Hausdorff Center for Mathematics at the University of Bonn (DFG GZ 2047/1, project ID
390685813).

\section{Background from higher category theory}\label{sec:highercats} 

In \Cref{sec:limits}, we recall the description of limits in the $\infty$-category of $\infty$-categories in terms of coCartesian sections of the Grothendieck construction, which will be used heavily throughout the paper. We also discuss the computation of limits in such a limit $\infty$-category. In \Cref{sec:lincats}, we discuss morphisms objects in linear $\infty$-categories. In \Cref{sec:locsystemsSn}, we discuss a spherical adjunction arising from $\infty$-categorical local systems on spheres.

\subsection{Limits of \texorpdfstring{$\infty$}{infinity}-categories}\label{sec:limits}

We denote by $\on{Cat}_\infty$ the $\infty$-category of $\infty$-categories and by $\on{St}$ the subcategory of stable $\infty$-categories and exact functors. 

An $\infty$-category $\mathcal{C}$ is called presentable if it is accessible\footnote{This means $\mathcal{C}$ is equivalent to the $\infty$-category of $\kappa$-$\on{Ind}$-objects in a small $\infty$-category, with $\kappa$ some regular cardinal.} and admits small colimits, see \cite[5.5.0.1]{HTT}. Limits and colimits will in the following always be assumed to be small. Examples of presentable $\infty$-categories include unbounded derived $\infty$-categories of dg algebras and $\infty$-categories of modules over $\mathbb{E}_\infty$-ring spectra. A remarkable property of presentable $\infty$-categories is the adjoint functor theorem, see \cite[5.5.2.9]{HTT}, and the duality for limits and colimits it entails. Namely, a functor $F\colon \mathcal{C}\rightarrow \mathcal{D}$ between presentable stable $\infty$-categories admits 
\begin{itemize}
\item a right adjoint if and only if $F$ preserves colimits.
\item a left adjoint if and only if $F$ preserves limits and $\kappa$-filtered colimits for some regular cardinal $\kappa$ (this latter condition is also called being an accessible functor). 
\end{itemize}

We denote by $\mathcal{P}r^L$ and $\mathcal{P}r^R$ the subcategories with objects the presentable $\infty$-categories and morphisms the functors which preserve colimits, respectively, preserve limits and are accessible.

\begin{theorem}[$\!\!${\cite[5.5.3.4]{HTT}}]\label{dualthm}
There exists a canonical equivalence of $\infty$-categories $\mathcal{P}r^L\simeq (\mathcal{P}r^R)^{\on{op}}$, restricting to the identity on objects and mapping each functor in $\mathcal{P}r^L$ to its right adjoint.
\end{theorem}

\Cref{dualthm} implies that the colimit of a diagram $D\colon Z\to \mathcal{P}r^L$ is equivalent to the limit of the right adjoint diagram $D^{\on{radj}}\colon Z^{\on{op}}\to\mathcal{P}r^R$. 

The inclusions $\mathcal{P}r^L,\mathcal{P}r^R,\on{St}\subset \on{Cat}_\infty$ preserve all limits. Limits in $\on{Cat}_\infty$ can be computed as follows. Consider a $1$-category $C$ and a diagram $D\colon C\rightarrow \on{Set}_\Delta$ of simplicial sets taking values in $\infty$-categories. We obtain the diagram of $\infty$-categories $D'\colon \on{N}(C)\rightarrow \on{Cat}_\infty$, given by considering $D$ as a strictly commuting diagram of $\infty$-categories. Let $p\colon \Gamma(D)\rightarrow \on{N}(C)$ be the relative nerve construction of \cite[3.2.5.2]{HTT}, where $\on{N}(C)$ denotes the nerve of $C$. We call $p$ or $\Gamma(D)$ the (covariant) \textit{Grothendieck construction} of $D$. The functor $p$ is a coCartesian fibration classified by the functor $D'$. The objects and morphisms in the Grothendieck construction $\Gamma(D)$ can be described as follows. 

The fiber of $p$ over $x\in N(C)$, i.e.~the pullback $\infty$-category $\Gamma(D)\times_{N(C)} \{x\}$, is given by $D(x)$. The set of objects of $\Gamma(D)$ is thus the disjoint union of the sets of objects of the $\infty$-categories $D(x)$ with $x\in C$. Given $x,y\in C$ and two objects $X\in D(x)$ and $Y\in D(y)$, a morphism $\alpha\colon X\rightarrow Y$ in $\Gamma(D)$ consists of 
\begin{itemize}
\item a morphism $f\colon x\rightarrow y$ in $C$ and 
\item a morphism $D(f)(X)\rightarrow Y$ in $D(y)$.
\end{itemize}
If $D(f)(X)\rightarrow Y$ is an equivalence, we call the morphism $\alpha$ a  \textit{coCartesian} morphism and write $\alpha\colon X\xrightarrow{!}Y$. If $D(f)$ admits a right adjoint, we call the morphism $\alpha$ a \textit{Cartesian} morphism if $D(f)(X)\rightarrow Y$ is a counit morphism of the adjunction $D(f)\dashv \on{radj}(D(f))$ and write $\alpha\colon X\xrightarrow{\ast}Y$. We note that these notions coincide with the notions of $p$-coCartesian and $p$-Cartesian edges of \cite[2.4.1.1]{HTT}.

By \cite[7.4.1.9]{Ker}, a choice of limit of the functor $D'$ is given by the $\infty$-category of coCartesian sections of $p\colon \Gamma(D)\rightarrow \on{N}(C)$, that is the full subcategory of the $\infty$-category 
\[ \on{Fun}_{\on{N}(C)}(\on{N}(C),\Gamma(D))\coloneqq \on{Fun}(\on{N}(C),\Gamma(D))\times_{\on{Fun}(\on{N}(C),\on{N}(C))} \{\on{id}_{\on{N}(C)}\}\,,\]
spanned by coCartesian sections, i.e.~those sections $s\colon N(C)\rightarrow \Gamma(D)$ satisfying that $s(e)$ is a coCartesian morphism for each morphism $e$ in $N(C)$. We can thus describe objects in the limit of $D$ as coCartesian sections and morphisms as natural transformations between coCartesian sections. 

In the cases of interest for us, the diagram $D'$ takes values in stable and presentable $\infty$-categories and limit and colimit preserving functors, which we thus assume in the following. Limits and colimits in the limit of $D'$ can in this case be computed in the respective fibers of $p$. This can be seen as follows: the computation of limits and colimits in the $\infty$-category of sections of $p\colon \Gamma(D')\rightarrow N(C)$ is equivalent to the computation of $p$-relative limits and colimits in the $\infty$-category $\on{Fun}(N(C),\Gamma(D'))$. By \cite[5.1.2.3, 4.3.1.10]{HTT} limits and colimits are computed pointwise in the respective fibers of $p$. It thus suffices to note that the property of a section of $p$ being coCartesian is preserved under limits and colimits in the $\infty$-category of sections of $p$, which follows from $D'(\alpha)$ being a limit and colimit preserving functor for each $1$-simplex $\alpha$ of $\on{N}(C)$. 

Finally, we also introduce the following terminology used later on.
\begin{definition}\label{def:supportofmorph}
Let $p\colon \Gamma(D')\rightarrow N(C)$ be as above, with $D'$ taking values in stable $\infty$-categories. We define the support of a morphism $\beta\colon s\rightarrow s'$ between (not necessarily coCartesian) sections of $p$ to be the subset of vertices $x$ of $N(C)$ such that $\beta(x)\colon s(x)\rightarrow s'(x)$ is not zero. 
\end{definition}

\subsection{Linear \texorpdfstring{$\infty$}{infinity}-categories and morphism objects}\label{sec:lincats}

Given an $\mathbb{E}_\infty$-ring spectrum $R$, the $\infty$-category $\on{RMod}_R$ of right $R$-module spectra admits the structure of a symmetric monoidal $\infty$-category such that the monoidal product preserves colimits separately in both variables. The $\infty$-category $\on{RMod}_R$ thus defines a commutative algebra object in the symmetric monoidal $\infty$-category $\mathcal{P}r^L$, see \cite[Section 4.8]{HA}, and we can consider its left modules. Note that if $R=k$ is a commutative ring, then $\on{RMod}_k$ is equivalent as a symmetric monoidal $\infty$-category to the (unbounded) derived $\infty$-category $\mathcal{D}(k)$.

\begin{definition}
Let $R$ be an $\mathbb{E}_\infty$-ring spectrum. The $\infty$-category $\on{LinCat}_R=\on{LMod}_{\on{RMod}_R}(\mathcal{P}r^L)$ of $R$-linear $\infty$-categories is defined as the $\infty$-category of left modules over $\on{RMod}_R$ in $\mathcal{P}r^L$.
\end{definition}

\begin{definition}[$\!\!${\cite[4.2.1.28]{HA}}]
Let $R$ be an $\mathbb{E}_\infty$-ring spectrum. Let $\mathcal{C}$ be an $R$-linear $\infty$-category and let $X,Y\in \mathcal{C}$. A morphism object is an $R$-module $\on{Mor}_{\mathcal{C}}(X,Y)\in \on{RMod}_R$ equipped with a map $\alpha\colon \on{Mor}_{\mathcal{C}}(X,Y)\otimes_R X\rightarrow Y$ in $\mathcal{C}$ such that for every object $C\in \on{RMod}_R$ composition with $\alpha$ induces an equivalence of spaces
\[ \on{Map}_{\on{RMod}_R}(C,\on{Mor}_{\mathcal{C}}(X,Y))\rightarrow \on{Map}_{\mathcal{C}}(C\otimes X,Y)\,.\]
\end{definition}

We thus have $\pi_i\on{Mor}_{\mathcal{C}}(X,Y)\simeq \pi_0\on{Map}_{\mathcal{C}}(X[i],Y)$. 

We will also denote $\on{Mor}_\mathcal{C}(X,X)$ by $\on{End}(X)$.

\begin{remark}
Morphism objects always exist and the formation of morphism objects forms a functor 
\[ \on{Mor}_\mathcal{C}(\mhyphen,\mhyphen)\colon \mathcal{C}^{\on{op}}\times \mathcal{C}\longrightarrow \on{RMod}_R\]
which preserves limits in both entries.
\end{remark}

Given a stable $\infty$-category $\mathcal{C}$ and two objects $A,B\in \mathcal{C}$, the $n$-th $\on{Ext}$-group is defined as 
\[ \on{Ext}^n_{\mathcal{C}}(A,B)\coloneqq\pi_0\on{Map}_{\mathcal{C}}(A,B[n])\simeq \pi_{-n} \on{Mor}_{\mathcal{C}}(A,B) \,.\] 
If $\mathcal{C}$ is $k$-linear, $\on{Mor}_{\mathcal{C}}(A,B)\in \on{RMod}_k\simeq \mathcal{D}(k)$ describes a chain complex and $\on{Ext}^n_{\mathcal{C}}(A,B)$ is its $(-n)$-th homology group. The abelian group $\on{Ext}^n_{\mathcal{C}}(A,B)$ thus inherits the structure of a (discrete) $k$-module. If all involved $\on{Ext}$-groups are free $k$-modules, we denote by 
\[ \chi \on{Ext}^\ast_{\mathcal{C}}(A,B)=\sum_{i\in\mathbb{Z}} (-1)^i \on{rk}_{k}\on{Ext}^{i}_{\mathcal{C}}(A,B)\] 
the Euler characteristic.

\subsection{\texorpdfstring{$\infty$}{infinity}-categories of local systems on spheres}\label{sec:locsystemsSn}

Let $R$ be an $\mathbb{E}_\infty$-ring spectrum and $n\geq 2$. We call the $R$-linear $\infty$-category $\on{Fun}(S^{n-1},\on{RMod}_R)$ the $\infty$-category of $\on{RMod}_R$-valued local systems on the $(n-1)$-sphere $S^{n-1}$. The pullback functor 
\[ f^*\colon \on{RMod}_R\rightarrow \on{Fun}(S^{n-1},\on{RMod}_R)\] along $f\colon S^{n-1}\rightarrow \ast$ admits a right adjoint $f_*$, given by the limit functor and the adjunction $f^*\dashv f_*$ is spherical, see \cite[Cor. 3.9]{Chr20}. In this section, we describe some of the properties of the adjunction $f^*\dashv f_*$ which are needed later on. 

\begin{lemma}[{$\!\!$\cite[Lem.~7.13, Prop.~7.15]{Chr22}}]\label{model1}
For $n\geq 2$ let $R[t_{n-2}]$ be the free algebra object in the symmetric monoidal $\infty$-category $\on{RMod}_R$ generated by $R[n-2]$ and consider the morphism $\phi\colon R[t_{n-2}]\xrightarrow{t_{n-2}\mapsto 0}R$ of algebra objects in $\on{RMod}_R$.
\begin{enumerate}[(1)]
\item Suppose that $n\geq 3$. There exists an equivalence of $R$-linear $\infty$-categories $\on{Fun}(S^{n-1},\on{RMod}_R)\simeq \on{RMod}_{R[t_{n-2}]}$ making the following diagram commute.
\[
\begin{tikzcd}
{\on{Fun}(S^{n-1},\on{RMod}_R)} \arrow[rr, "\simeq"] &                                                     & {\on{RMod}_{R[t_{n-2}]}} \\
                                                     & \on{RMod}_R \arrow[ru, "\phi^*"'] \arrow[lu, "f^*"] &                         
\end{tikzcd}
\] 
\item Suppose that $n=2$. Let $L$ be the simplicial set consisting of a single $0$-simplex and a single non-degenerate $1$-simplex and $\tilde{f}\colon L\to \ast$ is the unique morphism of simplicial sets. There exists an equivalence of $R$-linear $\infty$-categories $\on{Fun}(L,\on{RMod}_R)\simeq \on{RMod}_{R[t_{0}]}$ making the following diagram commute.
\[
\begin{tikzcd}
{\on{Fun}(L,\on{RMod}_R)} \arrow[rr, "\simeq"] &                                                     & {\on{RMod}_{R[t_{0}]}} \\
                                                     & \on{RMod}_R \arrow[ru, "\phi^*"'] \arrow[lu, "\tilde{f}^*"] &                         
\end{tikzcd}
\]
\end{enumerate}
\end{lemma}

\begin{remark}\label{rem:n=2_spherical_adjunction}
We note that $\on{Fun}(S^1,\on{RMod}_R)$ describes the full subcategory of $\on{Fun}(L,\on{RMod}_R)$ of objects where the $1$-simplex in $L$ acts as an equivalence. The image of the functor $\tilde{f}^*$ is contained in $\on{Fun}(S^1,\on{RMod}_R)\subset \on{Fun}(L,\on{RMod}_R)$. The sphericalness of the adjunction $f^*\dashv f_*$ implies the sphericalness of the adjunction $\tilde{f}^*\dashv \tilde{f}_*$, using that $\tilde{f}^*$ is monadic and \cite[Thm.~4.1]{Chr20}.
\end{remark}

\begin{lemma}\label{endomorlem}
There exists an equivalence of $R$-modules 
\[ \on{End}(f^*(R))=\on{Mor}_{\on{Fun}(S^{n-1},\on{RMod}_R)}(f^*(R),f^*(R))\simeq R\oplus R[1-n]\,.\]
\end{lemma}

\begin{proof}
The $R$-linear functor $f^*\colon \on{RMod}_R\rightarrow \on{Fun}(S^{n-1},\on{RMod}_R)$ is fully determined by the image of $R$, see \cite[Section 4.8.4]{HA}, and thus equivalent to the $R$-linear functor 
\[ \mhyphen\otimes_R f^*(R)\colon \on{RMod}_R\rightarrow \on{Fun}(S^{n-1},\on{RMod}_R)\,.\] 
Its right adjoint $f_*$ is equivalent to $\on{Mor}_{\on{Fun}(S^{n-1},\on{RMod}_R)}(f^*(R),\mhyphen)$. The equivalence $f_*f^*(R)\simeq R\oplus R[1-n]$ is shown in \cite[Section 3.1]{Chr20}.
\end{proof}

For $n\geq 3$, let $i\colon \ast\rightarrow S^{n-1}$ be the inclusion of any point with pullback $i^*\colon \on{Fun}(S^{n-1},\on{RMod}_R)\rightarrow \on{RMod}_R$ and right adjoint $i_*$. In the case $n=2$, we instead denote by $i$ the inclusion $\ast\to L$. Under the equivalence from \Cref{model1}, the object $i_*(R)$ is mapped to $R[t_{n-2}]\in \on{RMod}_{R[t_{n-2}]}$. 

For $R$ an arbitrary $\mathbb{E}_\infty$-ring spectrum, the description of the cotwist functor $T_{\on{Fun}(S^{n-1},\on{RMod}_R)}$ of the adjunction $f^*\dashv f_*$, as defined in \Cref{def:spherical_adjunction}, is an open problem, see \cite[Conj. 7.16]{Chr22}. We can however describe the values $T_{\on{Fun}(S^{n-1},\on{RMod}_R)}(i^*(R))$ and $T_{\on{Fun}(S^{n-1},\on{RMod}_R)}(f^*(R))$. 

\begin{lemma}\label{lem:twist=n}
\begin{enumerate}[(1)]
\item Assume that $n\geq 3$ and let $T_{\on{Fun}(S^{n-1},\on{RMod}_R)}$ be the cotwist functor of the spherical adjunction $f^*\dashv f_*$ and $X\in \on{RMod}_R$. There exist equivalences in $\on{Fun}(S^{n-1},\on{RMod}_R)$
\begin{align*}
T_{\on{Fun}(S^{n-1},\on{RMod}_R)}(f^*(X))&\simeq f^*(X)[1-n]\,,\\
T_{\on{Fun}(S^{n-1},\on{RMod}_R)}(i_*(X))&\simeq i_*(X)[1-n]\,.
\end{align*}
\item Let $T_{\on{Fun}(L,\on{RMod}_R)}$ be the cotwist functor of the spherical adjunction $\tilde{f}^*\dashv \tilde{f}_*$ and $X\in \on{RMod}_R$. There exist equivalences in $\on{Fun}(L,\on{RMod}_R)$
\begin{align*}
T_{\on{Fun}(L,\on{RMod}_R)}(f^*(X))&\simeq f^*(X)[-1]\,,\\
T_{\on{Fun}(L,\on{RMod}_R)}(i_*(X))&\simeq i_*(X)[-1]\,.
\end{align*}
\end{enumerate}
\end{lemma}

\begin{proof}
We prove part (1) below. Part (2) follows from a minor adaptation of the argument for part (1).

Let $T_{\on{RMod}_R}$ be the twist functor of the adjunction $f^*\dashv f_*$. The description of image under the cotwist $T_{\on{Fun}(S^{n-1},\on{RMod}_R)}(f^*(X))$ directly follows from the equivalences of functors $f^*T_{\on{RMod}_R}\simeq T_{\on{Fun}(S^{n-1},\on{RMod}_R)}f^*$, see \cite[Lemma 2.2]{Chr20}, and $T_{\on{RMod}_R}\simeq [1-n]$, see \cite[Section 3.1]{Chr20}.

For the description of $T_{\on{Fun}(S^{n-1},\on{RMod}_R)}(i_*(X))$, we begin by recalling some notation from \cite[Section 3.1]{Chr20}, where the sphericalness of the adjunction $f^*\dashv f_*$ is proven. Given a simplicial set $Z$, we denote by $Z^\triangleright=Z\ast \Delta^0$ the simplicial join. Consider the recursively defined simplicial sets $P_0=S^0=\Delta^0 \amalg \Delta^0$ and
\[ P_n\coloneqq P_{n-1}^\triangleright \coprod_{P_{n-1}} P_{n-1}^\triangleright\,.\]
We denote any vertex in $P_n\backslash P_{n-1}$ by $n$. Let $g\colon P_n\rightarrow \ast$ and $g^*\colon \mathcal{D}\rightarrow \on{Fun}(P_n,\mathcal{D})$ be the pullback functor with right adjoint $g_*$ given by the limit functor. The $\infty$-category $\on{Fun}(S^{n},\on{RMod}_R)$ embeds fully faithfully into $\on{Fun}(P_n,\on{RMod}_R)$, with image the functors mapping all edges in $P_n$ to equivalences in $\on{RMod}_R$, and the adjunction $g^*\dashv g_*$ restricts to $f^*\dashv f_*$ along this inclusion. Consider the simplicial sets 
\begin{align*} 
Z_1&=\{n\} \times \Delta^{1} \amalg_{\{n\}\times \Delta^{\{1\}}} P_{n}\times \Delta^{\{1\}}\\
Z_2&=\left(P_{n}\backslash \{n\}\right) \times \Delta^{1} \amalg_{\left(P_{n}\backslash \{n\}\right) \times \Delta^{\{0\}}} P_{n}\times \Delta^{\{0\}}\\
Z_3&=P_{n}\times \Delta^1
\end{align*}
We define the following $\infty$-categories via Kan extensions.
\begin{itemize}
\item Let $\mathcal{D}_1\subset \on{Fun}(Z_1,\on{RMod}_R)$ be the full subcategory spanned by right Kan extensions along $P_n\times \Delta^{\{1\}}\hookrightarrow Z_1$.
\item Let $\mathcal{D}_2\subset \on{Fun}(Z_2,\on{RMod}_R)$ be the full subcategory spanned by left Kan extensions along $P_n\times \Delta^{\{0\}}\hookrightarrow Z_2$.
\item Let $\mathcal{D}_1'\subset \on{Fun}(Z_3,\on{RMod}_R)$ be the full subcategory spanned by left Kan extensions along $Z_1\rightarrow Z_3$ of functors in $\mathcal{D}_1$.
\item Let $\mathcal{D}_2'\subset \on{Fun}(Z_3,\on{RMod}_R)$ be the full subcategory spanned by right Kan extensions along $Z_2\rightarrow Z_3$ of functors in $\mathcal{D}_2$.
\end{itemize}
 The diagrams in $\mathcal{D}_1'$ have the property that their restriction to $P_n\times \Delta^{\{0\}}$ is everywhere zero, except at $(n,0)$, where their value is identical to their value at $(n,1)$. Similarly, the diagrams in $\mathcal{D}_2'$ have the property that their restriction to $P_n \times \Delta^{\{1\}}$ is equivalent to the their restriction to $P_n\times \Delta^{\{0\}}$, except for having value $0$ at the vertex $(n,1)$. For $i=1,2$, the functors 
\[ R_i\colon \mathcal{D}_i'\rightarrow \mathcal{D}_i\rightarrow \on{Fun}(P_n\times \Delta^{\{1\}},\on{RMod}_R)\] 
are trivial fibrations by \cite[4.3.2.15]{HTT}. We can thus choose an essentially unique section of $R_i$, denoted $R^{-1}_i$. Consider the functor 
\begin{align*}
\on{colim}_{\Delta^1}\colon \on{Fun}(Z_3,\on{RMod}_R)&\simeq \on{Fun}(\Delta^1,\on{Fun}(P_n,\on{RMod}_R))\\
&\xlongrightarrow{\on{Fun}(\Delta^1,\on{colim})} \on{Fun}(\Delta^1,\on{RMod}_R)\end{align*}
The composite of $R^{-1}_1$ and the restriction of $\on{colim}_{\Delta^1}$ to $\mathcal{D}_1'$ describes a natural transformation $\eta\colon i^*\rightarrow f_!$. The cofiber of this natural transformation in $\on{Fun}(\on{Fun}(S^{n},\on{RMod}_R),\on{RMod}_R)$ can be described as the composite $\nu$ of $R^{-1}_2$, $\on{lim}_{\Delta^1}$ restricted to $\mathcal{D}_2'$ and the evaluation functor at $1\in \Delta^1$. A small computation reveals that $\nu$ is equivalent to $i^*[n]$.

Passing to right adjoints, we obtain a fiber and cofiber sequence $i_*[-n]\rightarrow f^*\rightarrow i_*$ in $\on{Fun}(\on{RMod}_R,\on{Fun}(S^{n},\on{RMod}_R))$, which evaluated at $X\in \on{RMod}_R$ yields a fiber and cofiber sequence $i_*(X)[-n]\rightarrow   f^*(X)\xrightarrow{\alpha} i_*(X)$. The composing of $\alpha$ with the equivalence $f^*(X)\simeq f^*f_*i_*(X)$ is a counit map of the adjunction $f^*\dashv f_*$. It thus follows $T_{\on{Fun}(S^{n},\on{RMod}_R)}(i_*(X))\simeq i_*(X)[-n]$, as desired.
\end{proof}

\section{Parametrized perverse schobers}\label{sec:schobers}

After discussing some background material on ribbon graphs and surfaces in \Cref{sec:graphssurfaces}, we give in \Cref{sec:defschober} an overview over the notion of a parametrized perverse schober from \cite{Chr22}. We then proceed in \Cref{sec:relGinzburg} by describing how the relation between perverse schobers and relative Ginzburg algebras of triangulated surfaces extends to $n$-angulated surfaces.

\subsection{Ribbon graphs and \texorpdfstring{$n$}{n}-angulated surfaces}\label{sec:graphssurfaces}

\begin{definition}~
\begin{itemize}
\item A graph $\mathcal{T}$ consists of two finite sets $\mathcal{T}_0$ of vertices and $\on{H}$ of halfedges together with an involution $\tau\colon \on{H}\rightarrow \on{H}$ and a map $\sigma\colon \on{H}\rightarrow \mathcal{T}_0$.
\item Let $\mathcal{T}$ be a graph. We call the orbits of $\tau$ the edges of $\mathcal{T}$. An edge is called internal if the orbit contains two elements and called external if the orbit contains a single element. 
\item A ribbon graph consists of a graph $\mathcal{T}$ together with a cyclic order on the set $\on{H}(v)$ of halfedges incident to $v$ for all $v\in\mathcal{T}_0$.
\end{itemize}
\end{definition}

\begin{definition}
Let $\mathcal{T}$ be a graph. We define the category $\on{Exit}(\mathcal{T})$ with
\begin{itemize}
\item objects given by the vertices and edges of $\mathcal{T}$ and 
\item non-identity morphisms from vertices to their incident edges. If an edge $e$ both begins and ends at a vertex $v$, there are two morphisms $v\rightarrow e$.
\end{itemize}
The geometric realization $|\mathcal{T}|$ of $\mathcal{T}$ is defined as the geometric realization of $\on{Exit}(\mathcal{T})$ (as a simplicial set). 
\end{definition}

\begin{definition}\label{spgrdef}~
\begin{itemize}
\item By a surface ${\bf S}$, we mean a connected, oriented, $2$-dimensional, smooth manifold with (possibly empty) boundary $\partial {\bf S}$.
\item By a marked surface $({\bf S},M)$ (or simply ${\bf S}$), we mean a compact surface ${\bf S}$ together with a non-empty finite subset of marked points $M$. Note that all boundary components of ${\bf S}$ are circles. We further require that each boundary component contains a marked point. Internal marked points are called punctures.
\item Let ${\bf S}$ be a marked surface and let $\mathcal{T}$ be a graph with an embedding $f\colon |\mathcal{T}|\rightarrow {\bf S}\backslash M$. We call $f$, or by abuse of notation $\mathcal{T}$, a spanning graph for ${\bf S}$ if 
\begin{enumerate}[(1)]
\item the embedding $f$ is a homotopy equivalence,
\item the restriction $|\mathcal{T}|\cap f^{-1}(\partial {\bf S})\rightarrow \partial {\bf S}\backslash M$ of $f$ is a homotopy equivalence.
\end{enumerate}
\end{itemize}
\end{definition}

\begin{remark}\label{embrem}
A spanning graph inherits a canonical structure of a ribbon graph, where the cyclic order at each vertex is the counterclockwise order induced by the orientation of the surface. We also call such ribbon graphs spanning. 
\end{remark}

\begin{definition}
A graph $\mathcal{T}$ is called $n$-valent if each vertex of $\mathcal{T}$ has valency $n$, i.e.~$n$ incident halfedges. 
\end{definition}

\begin{remark}\label{rem:nang}
Every $n$-valent spanning graph $\mathcal{T}$ of a marked surface ${\bf S}$ is the dual graph of an ideal $n$-angulation, meaning a decomposition of ${\bf S}$ into $n$-gons whose vertices lie at the marked points. Loops of the $n$-valent spanning graph correspond to self-folded edges of the $n$-gons. The case $n=3$ corresponds to ideal triangulations, see for example \cite[Def.~2.6]{FST08} for a definition.
\end{remark}

We note that not all marked surfaces admit a trivalent spanning graph, and much less an  $n$-valent spanning graph. For example, spheres with less than $3$ punctures, the unpunctured $1$-gon and the unpunctured $2$-gon do not admit a trivalent spanning graph.

\begin{remark}\label{ssurfrem}
Let $\mathcal{T}$ be a spanning ribbon graph of a marked surface ${\bf S}$. To each vertex $v$ of $\mathcal{T}$ of valency $n$ we associate a (non-compact) surface $\Sigma_v$ with an embedding of $v$ and its $n$ incident halfedges. We depict $\Sigma_v$ as follows (in green). The dotted lines correspond to open ends, whereas the solid lines indicate the boundary. 
\begin{center}
\begin{tikzpicture}
 \draw[color=ao][very thick][densely dotted] plot [smooth] coordinates {(0.5,2) (0.7,0.7) (2,0.5)};
 \draw[color=ao][very thick][densely dotted] plot [smooth] coordinates {(0.5,-2) (0.7,-0.7) (2,-0.5)};
 \draw[color=ao][very thick][densely dotted] plot [smooth] coordinates {(-0.5,-2) (-0.7,-0.7) (-1,-0.6)};
 \draw[color=ao][very thick][densely dotted] plot [smooth] coordinates {(-0.5,2) (-0.7,0.7) (-1,0.6)};
 \draw[color=ao][very thick] plot [smooth] coordinates {(2,-0.5) (2,0.5)};
 \draw[color=ao][very thick] plot [smooth] coordinates {(-0.5,2) (0.5,2)};
 \draw[color=ao][very thick] plot [smooth] coordinates {(-0.5,-2) (0.5,-2)}; 
  \node (0) at (0,0){};  
  \node (1) at (-1.4,0.14){\vdots};
  \node (2) at (-0.4,0){\small $v$};
  \node (3) at (0.44,-0.5){\small $e_{n-1}$};
  \node (4) at (0.24,0.6){\small $e_{1}$};
  \node (5) at (0.8,0.2){\small $e_{n}$};
  \fill (0) circle (0.1);
  \draw[very thick]
  (0,1.98) -- (0,0)
  (0,-1.98) -- (0,0)
  (0,0) -- (1.98,0); 
\end{tikzpicture}
\end{center}
We define the thickening of $\mathcal{T}$ to be the surface $\Sigma_\mathcal{T}$, obtained from gluing the surfaces $\Sigma_v$ at their boundaries whenever two vertices are incident to the same edge. The surface $\Sigma_{\mathcal{T}}$ comes with an embedding of $|\mathcal{T}|$, which is also a homotopy equivalence. We can further find an embedding and homotopy equivalence $\Sigma_{\mathcal{T}}\rightarrow {\bf S}$, such that the composite embedding $|\mathcal{T}|\rightarrow  \Sigma_\mathcal{T} \rightarrow {\bf S}\backslash M$ is the given embedding of the spanning ribbon graph $\mathcal{T}$. 
\end{remark}

\subsection{Definition and properties}\label{sec:defschober}

\begin{definition}\label{def:spherical_adjunction}
Let $F\colon \mathcal{A}\leftrightarrow \mathcal{B}\noloc G$ be an adjunction between stable $\infty$-categories. We define 
\begin{itemize}
\item the twist functor $T_\mathcal{A}\colon \mathcal{A}\rightarrow \mathcal{A}$ as the cofiber $T_\mathcal{A}=\on{cof}(\on{id}_{\mathcal{A}}\xrightarrow{u} GF)$ of the unit of $F\dashv G$ in the stable $\infty$-category $\on{Fun}(\mathcal{A},\mathcal{A})$ and 
\item the cotwist functor as the fiber $T_{\mathcal{B}}=\on{fib}(FG\xrightarrow{\on{cu}} id_{\mathcal{B}})$ of the counit of $F\dashv G$ in the stable $\infty$-category $\on{Fun}(\mathcal{B},\mathcal{B})$.
\end{itemize}
The adjunction $F\dashv G$ is called spherical if both $T_\mathcal{A}$ and $T_{\mathcal{B}}$ are equivalences. In this case, we also call $F$ a spherical functor.
\end{definition}

 The notion of a spherical functor in the context of dg categories is due to \cite{AL17} and generalizes the notion of a spherical object. For treatments of spherical adjunctions in the setting of stable $\infty$-categories, we refer to \cite{DKSS24,Chr20}. We proceed by introducing the local model for a parametrized perverse schober, based on the datum of a spherical functor.

Let $F\colon \mathcal{V}\rightarrow \mathcal{N}_F$ be a spherical functor and let $[n-1]$ be the poset with elements $0,\dots,n-1$. Consider the diagram $D(F)\colon [n-1]\rightarrow \on{Set}_{\Delta}$ obtained from the $n-1$ composable functors
\begin{equation}\label{comfun} \mathcal{V}\xlongrightarrow{F}\underbrace{\mathcal{N}_F\xlongrightarrow{\on{id}}\dots\xrightarrow{\on{id}}\mathcal{N}_F}_{(n-1)\text{-many}}\,.
\end{equation}
We define the stable $\infty$-category $\mathcal{V}^n_{F}$ as the $\infty$-category of sections of the Grothendieck construction $p\colon \Gamma(D(F))\rightarrow N([n-1])\simeq \Delta^{n-1}$ (see \Cref{sec:limits}) of the diagram $D(F)$. The objects of $\mathcal{V}^n_F$ are therefore diagrams of the form
\begin{equation}\label{objeq}
 a\xrightarrow{\alpha_1} b_1\xrightarrow{\alpha_2} \dots \xrightarrow{\alpha_{n-1}} b_{n-1}
\end{equation} 
with $a\in \mathcal{V}, b_i\in \mathcal{N}_F$ and the morphism $\alpha_1$ encoding the datum of a morphism $F(a)\rightarrow b_1$ in $\mathcal{N}_F$.

The $\infty$-category $\mathcal{V}^n_{F}$ comes with a semiorthogonal decomposition (SOD) $(\mathcal{V},\mathcal{N}_F,\dots,\mathcal{N}_F)$ of length $n$, see \cite[Section 2.6]{Chr22} for background. The $i$-th component of the SOD is spanned by objects of the form \eqref{objeq} satisfying that $b_l\simeq 0$ for all $l$ if $i=1$ and $a\simeq 0\simeq b_l$ for all $l\neq i-1$ for $2\leq i \leq n$. 

We consider an $n$-valent vertex $v$ of a ribbon graph $\mathcal{T}$ with incident halfedges $e_1,\dots,e_n$, ordered compatibly with their cyclic order. We consider the poset $C_v$ with objects $v,e_1,\dots,e_n$ and morphisms $v\rightarrow e_i$. It is equivalent to the under category $\on{Exit}(\mathcal{T})_{v/}$. There is a resulting map $C_v\simeq \on{Exit}_{v/}\rightarrow \on{Exit}(\mathcal{T})$ which maps $v$ to $v$ and each halfedge $e_i$ to the edge of $\mathcal{T}$ which contains that halfedge. The local model at $v$ of a parametrized perverse schober consists of a functor 
\begin{equation}\label{locseq}
 \mathcal{F}_v(F)\colon C_v\longrightarrow \on{St}\,,
\end{equation} 
defined as follows. It is defined on objects via $\mathcal{F}_v(F)(v)=\mathcal{V}^n_F$, $\mathcal{F}_v(F)(e_i)=\mathcal{N}_{F}$. It is defined on morphisms via 
\begin{equation}\label{rhoeq} 
\mathcal{F}_v(F)(v\rightarrow e_i)=\varrho_i\coloneqq \begin{cases}
\pi_n & i=1\\
\on{cof}_{n-i,n-i+1}[i-2] & 2\leq i \leq n-1\\
\on{rcof}_{1,2}[n-2] & i=n
\end{cases}
\end{equation}
where 
\begin{itemize}
\item $\pi_n$ denotes the projection functor to the last component of the SOD (i.e.~maps \eqref{objeq} to $b_{n-1}$), 
\item $\on{cof}_{i,i+1}$ denotes the composite of the projection to the $i$-th and $(i+1)$-th components of the SOD with the cofiber functor (i.e.~maps \eqref{objeq} to the cofiber of $\alpha_{i}$) 
\item and $\on{rcof}_{1,2}$ denotes the composite of the projection to the first two components of the SOD with the relative cofiber functor (i.e.~maps \eqref{objeq} to the cofiber of $F(a)\rightarrow b_1$). 
\end{itemize}

\begin{remark}\label{adjsrem}
The functors $\varrho_1,\dots,\varrho_n$ in \eqref{rhoeq} admit right adjoints $\varsigma_i$ which form a sequence of adjunctions 
\begin{equation}\label{sigmaadjeq}
\varsigma_1T_{\mathcal{N}_F}^{-1}[1-n]\dashv \varrho_n\dashv \varsigma_n\dashv \varrho_{n-1}\dashv \varsigma_{n-1}\dashv \dots \dashv \varrho_1\dashv \varsigma_1\,,
\end{equation}
with $T_{\mathcal{N}_F}$ the cotwist functor of the adjunction $F\dashv \on{radj}(F)$, see \cite[Section 3]{Chr22}.
\end{remark}

We are now ready to give the full definition of a parametrized perverse schober. 

\begin{definition}\label{def:schober}
Let $\mathcal{T}$ be a ribbon graph. A $\mathcal{T}$-parametrized perverse schober is a functor 
\[ \mathcal{F}\colon \on{Exit}(\mathcal{T})\longrightarrow \on{St}\,,\]
such that for each vertex $v$ of $\mathcal{T}$ there exists a spherical functor $F_v\colon \mathcal{V}_{v}\rightarrow \mathcal{N}_{F_v}$ together with an equivalence between the restriction of $\mathcal{F}$ along $C_v\rightarrow \on{Exit}(\mathcal{T})$ and $\mathcal{F}_v(F_v)$ in $\on{Fun}(C_v,\on{St})$. 

A vertex $v$ of $\mathcal{T}$ is called a singularity of $\mathcal{F}$ if the domain $\mathcal{V}_{v}$ of the spherical functor $F_v$ is not equivalent to the zero stable $\infty$-category.
\end{definition}

\begin{definition}\label{glsecdef}
Let $\mathcal{F}$ be a $\mathcal{T}$-parametrized perverse schober. 
\begin{itemize}
\item The stable $\infty$-category of global sections $\mathcal{H}(\mathcal{T},\mathcal{F})$ is defined as the limit of $\mathcal{F}$ in $\on{St}$. As the limit, we usually choose the $\infty$-category of coCartesian sections of the Grothendieck construction $p\colon \Gamma(\mathcal{F})\rightarrow \on{Exit}(\mathcal{T})$.
\item The stable $\infty$-category $\mathcal{L}$ of lax sections of $\mathcal{F}$ is the $\infty$-category of (all) sections of the Grothendieck construction $p\colon \Gamma(\mathcal{F})\rightarrow \on{Exit}(\mathcal{T})$.
\end{itemize}
\end{definition}

The \Cref{def:schober} of a $\mathcal{T}$-parametrized perverse schober involves for each vertex of $\mathcal{T}$ a choice of a total order on the a priori cyclically ordered set of incident edges. Different choices lead to different local models $\mathcal{F}_v(F_v)$, which are related by natural equivalences obtained from powers of the equivalence $T_{\mathcal{V}^n_F}$ described in \Cref{paratwprop}. \Cref{def:schober} is thus independent of these choices.

\begin{proposition}[$\!\!$\cite{Chr22}]\label{paratwprop}
Let $F\colon \mathcal{V}\leftrightarrow \mathcal{N}_F\noloc G$ be a spherical adjunction. The functor 
\[ (\varrho_1,\dots,\varrho_n)\colon \mathcal{V}^n_{F} \longrightarrow \mathcal{N}_F^{\times n}\] 
is spherical. The twist functor $T_{\mathcal{V}^n_{F}}$ satisfies 
\[ \varrho_i\circ T_{\mathcal{V}^n_{F}} =\begin{cases} 
\varrho_{i+1}& 1\leq i\leq n-1\\
T_{\mathcal{N}_F}[n-1]\circ \varrho_1 & i=n
\end{cases}\,,\]
where $T_{\mathcal{N}_F}$ is the cotwist functor of $F\dashv G$. 
\end{proposition}

\begin{remark}\label{cyctwrem}
We call the equivalence $T_{\mathcal{V}^n_F}$ from \Cref{paratwprop} the paracyclic twist functor as it realizes the paracyclic rotational symmetry of perverse schobers at their vertices. 
\end{remark}

We can change the parametrizing ribbon graph of a perverse schober via contractions of ribbon graphs. A contraction $c\colon \mathcal{T}\rightarrow \mathcal{T}'$ of ribbon graphs contracts by definition finitely many edges of $\mathcal{T}$ such that no loops are collapsed. A more formal definition of a contraction of ribbon graphs can be found in \cite[Def. 4.24]{Chr22}. 

Assume that $c$ contracts a single edge $e$ of $\mathcal{T}$, connecting two vertices $v,v'$. Given a $\mathcal{T}$-parametrized perverse schober $\mathcal{F}$, satisfying that $v$ or $v'$ is not a singularity of $\mathcal{F}$, it is shown in \cite[Section 4.4]{Chr22} that there exists a $\mathcal{T}'$-parametrized perverse schober $c_*\mathcal{F}$ satisfying the following.
\begin{itemize}
\item $c_*\mathcal{F}(\tilde{v})= \mathcal{F}(\tilde{v})$ for each vertex $\tilde{v}\neq v,v'$ of $\mathcal{T}'$.
\item $c_*\mathcal{F}(\tilde{e})= \mathcal{F}(\tilde{e})$ for each edge $\tilde{e}\neq e$ of $\mathcal{T}'$.
\item $c_*\mathcal{F}(\tilde{v}\rightarrow \tilde{e})=\mathcal{F}(\tilde{v}\rightarrow \tilde{e})$ for each edge $\tilde{e}$ incident to a vertex $\tilde{v}$ such that $\tilde{v}\neq v,v'$ and $\tilde{e}\neq e$.
\item For the vertex $v=v'$ of $\mathcal{T}'$, there exists an equivalence of $c_*\mathcal{F}(v)$ and the pullback in $\on{St}$ of the following diagram.
\[
\begin{tikzcd}
                                                          & \mathcal{F}(v) \arrow[d, "\mathcal{F}(v\rightarrow e)"] \\
\mathcal{F}(v') \arrow[r, "\mathcal{F}(v'\rightarrow e)"] & \mathcal{F}(e)                                         
\end{tikzcd}
\]
\item For each morphism $v'=v\rightarrow \tilde{e}$ in $\on{Exit}(\mathcal{T}')$ arising from a morphism $v\rightarrow \tilde{e}$ in $\on{Exit}(\mathcal{T})$, there exists an equivalence of functors between $c_*\mathcal{F}(v\rightarrow \tilde{e})$ and  the composite
\[ c_*\mathcal{F}(v)\rightarrow \mathcal{F}(v)\xrightarrow{\mathcal{F}(v\rightarrow \tilde{e})} \mathcal{F}(\tilde{e})= c_*\mathcal{F}(\tilde{e})\,.\] An analogous statement holds with $v$ replaced by $v'$.
\end{itemize}

Since each contraction can be written as the composite of contractions which each contract a single edge, we can associate to each contraction $c$ from $\mathcal{T}$ to $\mathcal{T}'$ and each $\mathcal{T}$-parametrized perverse schober $\mathcal{F}$, satisfying that no edges connecting two singularities of $\mathcal{F}$ are contracted by $c$, a $\mathcal{T}'$-parametrized perverse schober $c_*\mathcal{F}$. As shown in \cite[Prop.~4.28]{Chr22}, the functor $c_*$ commutes with taking global sections.

\begin{proposition}\label{contrprop}
Let $c\colon \mathcal{T}\rightarrow \mathcal{T}'$ be a contraction of ribbon graphs and let $\mathcal{F}$ be a $\mathcal{T}$-parametrized perverse schober, such that $c$ contracts no edges connecting two singularities of $\mathcal{F}$. Then there exists an equivalence of $\infty$-categories 
\[ \mathcal{H}(\mathcal{T},\mathcal{F})\simeq \mathcal{H}(\mathcal{T}',c_*\mathcal{F})\,.\]
\end{proposition}

\subsection{Relative Ginzburg algebras}\label{sec:relGinzburg}

We begin with the definition of the relative higher Ginzburg algebra associated with an $n$-valent spanning graph, for any $n\geq 2$, generalizing the definition in the case $n=3$ from \cite{Chr22}. For the following, we fix a base commutative ring $k$.

\begin{definition}\label{gqdef}
Let $\mathcal{T}$ be an $n$-valent spanning graph of a marked surface. We define a graded quiver $\tilde{Q}_{\mathcal{T}}$ with vertices the edges of $\mathcal{T}$ and the following graded arrows.
\begin{itemize}
\item An arrow $a_{v,i,j}\colon i\rightarrow j$ of degree $l-1$ for each vertex $v\in \mathcal{T}_0$ at which a halfedge of $i$ follows a halfedge of $j$ in the (counterclockwise) cyclic order after $1\leq l\leq n-1$ steps. The arrows thus go in the clockwise direction and the loops of $\mathcal{T}$ give rise to loops in $\tilde{Q}_{\mathcal{T}}$.
\item A further loop $l_i\colon i\rightarrow i$ of degree $n-1$ for each internal edge $i$.
\end{itemize}
Given two edges $i,j\in \mathcal{T}_1$ incident to $v\in \mathcal{T}_0$, we denote by $j-i\in \{0,\dots,n-1\}$ the number of steps after which $i$ follows $j$ in the cyclic order at $v$. The arrow $a_{v,i,j}$ thus lies in degree $j-i-1$.

The relative Ginzburg algebra $\mathscr{G}_\mathcal{T}$ is defined as the $k$-linear dg algebra with underlying graded algebra $k\tilde{Q}_\mathcal{T}$ and with differential $d$ determined on the generators as follows. 

For the generators $a_{v,i,j}$, we set
\[d(a_{v,i,j})= \sum_{j<k<i} (-1)^{j-k-1} a_{v,k,j}a_{v,i,k}\,,\]
where the sum runs over all edges $k$ appearing between $j$ and $i$ in the cyclic order. 
For an internal edge $i$ incident to two (possibly identical) vertices $v_1,v_2\in \mathcal{T}_0$, we set
\begin{equation}\label{dlieq} d(l_i)=\sum_{j\neq i}\,(-1)^{j-i}a_{v_1,j,i}a_{v_1,i,j}+(-1)^{n-1}\sum_{j\neq i}\,(-1)^{j-i}a_{v_2,j,i}a_{v_2,i,j}\,.\end{equation}
\end{definition}

Note that a different choice of labeling of $v_1,v_2$ in \eqref{dlieq} changes $\mathscr{G}_\mathcal{T}$ by an isomorphism of dg algebras (by mapping $l_i$ by $-l_i$).

\begin{remark}
The formula for the differential of $a_{v,i,j}$ in \Cref{gqdef} should be considered as a version of graded cyclic derivative of a potential 
\[ W_{\mathcal{T}}'=\sum_{v\in \mathcal{T}_0} \sum_{j<k<i} \pm a_{v,k,j}a_{v,i,k}a_{v,j,i}\,.\]
We note further that if an edge, with halfedges $i,j$ is a loop, then its incident vertices $v_1,v_2$ are identical but $a_{v_1,j,i}a_{v_1,i,j}\neq a_{v_2,j,i}a_{v_2,i,j}$, as the arrows arise from the two different halfedges of $i$. The notation $a_{v,i,j}$ is thus abusive in this case.
\end{remark}

\begin{example}\label{locGinzex}
For $n=2,3,4$, let ${\bf S}$ be the $n$-gon and $\mathcal{T}$ the $n$-valent spanning graph of ${\bf S}$ (unique up to homotopy). The graded algebras underlying the relative Ginzburg algebras $\mathscr{G}_\mathcal{T}$ are given by the path algebras of the following graded quivers. The labels indicate degrees of the arrows.
\[
\begin{tikzcd}[row sep=large, column sep=large]
\cdot \arrow[rr, "0", bend left=20] &  & \cdot \arrow[ll, "0", bend left=20]
\end{tikzcd}\hspace{2em}
\begin{tikzcd}[row sep=large, column sep=large]
                                                 & \cdot \arrow[rd, "0", bend left] \arrow[ld, "1"] &                                                  \\
\cdot \arrow[ru, "0", bend left] \arrow[rr, "1"] &                                                  & \cdot \arrow[ll, "0", bend left] \arrow[lu, "1"]
\end{tikzcd} \hspace{2em}
\begin{tikzcd}[row sep=large, column sep=large]
                                                                                        & \cdot \arrow[rd, "0", bend left] \arrow[ld, "2" description] \arrow[dd, "1", bend left]  &                                                                                         \\
\cdot \arrow[ru, "0", bend left] \arrow[rd, "2" description] \arrow[rr, "1", bend left] &                                                                                          & \cdot \arrow[ld, "0", bend left] \arrow[lu, "2" description] \arrow[ll, "1", bend left] \\
                                                                                        & \cdot \arrow[lu, "0", bend left] \arrow[ru, "2" description] \arrow[uu, "1"', bend left] &                                                                                        
\end{tikzcd}
\]
The differentials of the arrows in degree $0$ vanish and the differential of an arrow $a\colon x\rightarrow y$ of degree $m$ consists, modulo signs, of the sum of all paths composed of two arrows of degrees less than $m$ which compose to a path $x\rightarrow y$.  

More generally, for the disc equipped with a $2$-angulation, the relative Ginzburg algebra $\mathscr{G}_\mathcal{T}$ describes a relative version of the derived preprojective algebra of a  type $A$ quiver. The Koszul dual of such a dg algebra, namely a relative zig-zag algebra, was considered in \cite{KS02}.
\end{example}

\begin{construction}\label{constr:Ginzburgschober}
We fix a marked surface ${\bf S}$ with an $n$-valent spanning graph $\mathcal{T}$ with $n\geq 2$ and a choice of $\mathbb{E}_\infty$-ring spectrum $R$. In this construction, we describe a $\mathcal{T}$-parametrized perverse schober $\mathcal{F}_\mathcal{T}(R)$. We will show in \Cref{gluethm} that in the case that $R=k$ is a commutative ring the $\infty$-category of global sections $\mathcal{H}(\mathcal{T},\mathcal{F}_\mathcal{T}(R))$ is equivalent to $\mathcal{D}(\mathscr{G}_\mathcal{T})$.

Locally at each vertex of $\mathcal{T}$, the perverse schober $\mathcal{F}_\mathcal{T}(R)$ is described by the spherical adjunction
\[ \phi^*\colon \on{RMod}_R\longleftrightarrow \on{RMod}_{R[t_{n-2}]}\noloc \phi_*\]
where $R[t_{n-2}]$ denotes the free $R$-linear ring spectrum generated by $R[n-2]$ and $\phi^*$ is the pullback map along the morphism of ring spectra $\phi\colon R[t_{n-2}]\xrightarrow{t_{n-2}\mapsto 0} R$. We note that for $n\geq 3$ the adjunction $\phi^*\dashv \phi_*$ is equivalent to the spherical adjunction $f^*\colon \on{RMod}_R\longleftrightarrow \on{Fun}(S^{n-1},\on{RMod}_R)\noloc f_*$, arising from the pullback functor along $f\colon S^{n-1}\rightarrow \ast$, see \Cref{model1}. In the case $n=2$, the adjunction is still spherical by \Cref{rem:n=2_spherical_adjunction}.

We thus want to define $\mathcal{F}_\mathcal{T}(R)$ as the gluing of the local perverse schobers $\mathcal{F}_v(\phi^*)\colon C_v\rightarrow \on{St}$, see \eqref{locseq}, i.e.~as the diagram $\on{Exit}(\mathcal{T})\rightarrow \on{St}$ which restricts at $C_v$ to $\mathcal{F}_v(\phi^*)$. Note that this definition involves making choices, see \Cref{choicerem}. In the case that $n$ is odd, these choices matter and we need to slightly modify the perverse schobers $\mathcal{F}_v(\phi^*)$, to match the signs in the differentials of the Ginzburg algebras.

For each edge $e$ of $\mathcal{T}$, we consider its two incident (possibly identical) vertices $v_1,v_2$. We denote by $i_1\in \{1,\dots,n\}$ the position of the halfedge of $e$ lying at $v_1$ in the chosen total order of the $n$ halfedges incident to $v_1$. We similarly denote by $i_2\in \{1,\dots,n\}$ the position of the halfedge of $e$ at $v_2$ in the chosen total order of halfedges at $v_2$. If $i_1-i_2$ is even, we change $\mathcal{F}_{v_1}(\phi^*)$ by composing $\mathcal{F}_{v_1}(\phi^*)(v_1\rightarrow e)$ with the autoequivalence $T$ of $\on{RMod}_{R[t_{n-2}]}$, given by the the pullback functor along the morphism of ring spectra $R[t_{n-2}]\xrightarrow{t_{n-2}\mapsto (-1)^{n}t_{n-2}} R[t_{n-2}]$. Note that for $R=k$ a commutative ring, the functor $T[n-1]$ is equivalent to the cotwist functor of $\phi^*\dashv \phi_*$, see \cite[Prop.~5.7]{Chr22}, and clearly $T\simeq \on{id}_{\on{RMod}_{R[t_{n-2}]}}$ if $n$ is even. If $i_1-i_2$ is odd, we do nothing. The perverse schober $\mathcal{F}_\mathcal{T}(R)$ is now defined as the gluing of the above modifications of the $\mathcal{F}_v(\phi^*)$'s.
\end{construction}

\begin{remark}\label{choicerem}
The definition of $\mathcal{F}_\mathcal{T}(R)$ in \Cref{constr:Ginzburgschober} involves a choice of total order of the halfedges incident to each vertex $v\in \mathcal{T}_0$, which a priori only carry a cyclic order. Different choices of total orders lead to a different perverse schober $\mathcal{F}_\mathcal{T}(R)$. It is shown in \cite[Section 7.1]{Chr22}, that in the case $R=k$ a commutative ring and $n=3$ the equivalence class of the perverse schober $\mathcal{F}_\mathcal{T}(k)$ does not depend on these choices. The proof given there generalizes to any $n\geq 2$ (and $R=k$). In the case that $n$ is even, since the twist functor of $\phi^*\dashv \phi_*$ is equivalent to $[1-n]$, instead of a choice of spin structure on ${\bf S}\backslash M$, only a choice of orientation is required as input for the construction of $\mathcal{F}_\mathcal{T}(k)$. We expect that such a statement also holds for $R$ any $\mathbb{E}_\infty$-ring spectrum. The local modifications affect the monodromy of $\mathcal{F}_\mathcal{T}(k)$ and ensure that it is trivial along loops embedded in $\mathcal{T}$. 
\end{remark}

\begin{theorem}\label{gluethm} 
Let $n\geq 2$ and ${\bf S}$ a marked surface with an $n$-valent spanning graph $\mathcal{T}$. There exists an equivalence of $\infty$-categories
\begin{equation}\label{hgeq1}
 \mathcal{H}(\mathcal{T},\mathcal{F}_\mathcal{T}(k))\simeq \mathcal{D}(\mathscr{G}_\mathcal{T})
\end{equation}
between the $\infty$-category of global sections of the perverse schober $\mathcal{F}_{\mathcal{T}}(k)$ and the derived $\infty$-category of the relative Ginzburg algebra $\mathscr{G}_\mathcal{T}$.
\end{theorem}

\Cref{sec:proofofgluing} is dedicated to the proof of \Cref{gluethm}.

\begin{remark}\label{rlinrem}
The perverse schober $\mathcal{F}_\mathcal{T}(R)$ canonically factors through $\on{LinCat}_R\rightarrow \on{St}$, so that its limit $\mathcal{H}(\mathcal{T},\mathcal{F}_\mathcal{T}(R))$ inherits a canonical $R$-linear structure. \Cref{gluethm} can be strengthened to the statement that \eqref{hgeq1} is an equivalence of $k$-linear $\infty$-categories.
\end{remark}

\begin{proposition}\label{compgenprop}
Let $\mathcal{T}$ be an $n$-valent spanning graph of a marked surface. Given an edge $e$ of $\mathcal{T}$, we denote by 
\[ \on{ev}_e\colon \mathcal{H}(\mathcal{T},\mathcal{F}_\mathcal{T}(R))\rightarrow \mathcal{F}_\mathcal{T}(R)(e)=\on{RMod}_{R[t_{n-2}]}
\]
the evaluation functor (which evaluates a coCartesian section at $e$) and by $\on{ev}_e^*$ its left adjoint. 
\begin{enumerate}[(1)]
\item The direct sum $\bigoplus_{e} \on{ev}_e^*(R[t_{n-2}])$ over all edges of $\mathcal{T}$ is a compact generator of $\mathcal{H}(\mathcal{T},\mathcal{F}_\mathcal{T}(R))$.
\item If $R=k$, the projective $\mathscr{G}_\mathcal{T}$-module $p_e\mathscr{G}_\mathcal{T}$, where $p_e\in k\tilde{Q}_{\mathcal{T}}$ is the lazy path at $e$, is identified under the equivalence \eqref{hgeq1} with $\on{ev}^*_e(k[t_{n-2}])$.
\end{enumerate}
\end{proposition}

\begin{proof}
For part (1), it suffices to observe that the functor $(\varrho_1,\dots,\varrho_n)\colon \mathcal{V}^n_{\phi^*}\rightarrow \on{RMod}_{R[t_{n-2}]}^{\times n}$ is conservative, since this implies that a global section $X\in \mathcal{H}(\mathcal{T},\mathcal{F}_\mathcal{T}(R))$ is zero if and only if $\on{ev}_e(X)=0$ for all edges $e$ of $\mathcal{T}$. 

For part (2), the corresponding proof from \cite[Prop. 6.7]{Chr22} for the case $n=3$ directly generalizes to any $n\geq 2$.
\end{proof}

\subsection{The proof of \texorpdfstring{\Cref{gluethm}}{Theorem 4.19}}\label{sec:proofofgluing}

For the entirety of this section, we fix an integer $n\geq 2$ and a commutative ring $k$. Let $\mathcal{T}$ be an $n$-valent spanning graph of a marked surface. The goal of this section is to show that the $\infty$-category of global sections $\mathcal{H}(\mathcal{T},\mathcal{F}_\mathcal{T})$ is equivalent to the derived $\infty$-category of the relative Ginzburg algebra of $\mathcal{T}$. Our approach is similar to the approach in the case $n=3$ in Sections 5 and 6 of \cite{Chr22}.

In the following, we very briefly recall the relation between dg categories and stable, presentable $\infty$-categories and some sign conventions. More details can be found in \cite[Section 2]{Chr22}. We then give an algebraic description of the perverse schober on the $n$-gon obtained from the spherical adjunction $\phi^*\dashv \phi_*$. This step constitutes the main work in the proof of $\mathcal{H}(\mathcal{T},\mathcal{F}_\mathcal{T})\simeq \mathcal{D}(\mathscr{G}_\mathcal{T})$, the remaining steps are analogous to the proof in the case $n=3$ in \cite{Chr22}.

We denote by $\on{dgCat}_k$ the category of $k$-linear dg categories. A dg functor $F\colon A\rightarrow B$ is called a quasi-equivalence if for all $a,a'\in A$, the map between morphism complexes $F(a,a')\colon \on{Hom}_{A}(a,a')\rightarrow \on{Hom}_{{B}}(F(a),F(a'))$ is a quasi-isomorphism and the induced functor on homotopy categories is an equivalence. The dg functor $F$ is called a Morita equivalence if the induced functor $F^{\on{perf}}\colon A^{\on{perf}}\rightarrow B^{\on{perf}}$ on dg categories of perfect modules is a quasi-equivalence. 
The category $\on{dgCat}_k$ admits the quasi-equivalence model structure whose weak equivalences are the quasi-equivalences. There is a functor $\mathcal{D}(\mhyphen)\colon N(\on{dgCat}_k)\rightarrow \mathcal{P}r^L$ from the nerve of $\on{dgCat}_k$. It maps dg algebras (i.e.~dg categories with a single object) to their unbounded derived $\infty$-categories, c.f.~\cite[Section 2.5]{Chr22}. The functor $\mathcal{D}(\mhyphen)$ maps homotopy colimits with respect to the quasi-equivalence structure to colimits in $\mathcal{P}r^L$ and maps Morita equivalences to equivalences of $\infty$-categories. 

Given a dg category $C$, we denote by $\on{dgMod}(C)$ the dg category of right $C$-modules. We denote the mapping complex between two objects $x,y\in C$ by $\on{Hom}_C(x,y)$ or $\on{Hom}(x,y)$. Given two dg modules $x,y\in \on{dgMod}(C)$ and a morphism $\alpha\colon x\rightarrow y$, we denote by $\on{cone}(\alpha)=x[1]\oplus y$ the cone with differential $d(x,y)=(-d(x),d(y)-\alpha(x))$. Given two dg modules $x,y$, the morphism complex $\on{Hom}(x,y)$ has differential $d(f)=d_y\circ f - (-1)^{\on{deg}(f)}f\circ d_x$.

\begin{lemma}
Let 
\[\bf{A}_m=
\begin{pmatrix}
k & k[1-n] &0 & \dots &0 &0\\
0   & k[t_{n-2}] & k[t_{n-2}]&  0 & \dots &0\\
0   & 0   & k[t_{n-2}]& k[t_{n-2}] & \dots &0\\
\vdots& \vdots & \vdots & \ddots & \vdots & \vdots \\
0   & 0   & 0  & \dots  & k[t_{n-2}] & k[t_{n-2}]\\
0   & 0   & 0  & \dots  & 0 & k[t_{n-2}]
\end{pmatrix}
\] 
be the upper triangular dg algebra. Then there exists an equivalence of $\infty$-categories
\begin{equation}\label{model2}
\mathcal{V}^m_{\phi^*}\simeq \mathcal{D}(\bf{A}_m)
\end{equation}
\end{lemma}

\begin{proof}
This follows from \cite[Prop. 2.39]{Chr22}, using that 
\[ \phi_*\simeq \phi_![1-n]=\on{ladj}(\phi^*)[1-n]\simeq \left(\mhyphen\otimes^L_{k[t_{n-2}]}k\right)[1-n]\,.\]
\end{proof}

\begin{remark}
Consider the graded quiver $Q_m$
\begin{equation}\label{appquiv1}
\begin{tikzcd}[column sep=large]
x_0 \arrow[r, "a_{0,1}"] & x_1 \arrow[r, "{a_{1,2}}"] \arrow["l_1"', loop, distance=2em, in=125, out=55] & \dots \arrow[r, "{a_{m-2,m-1}}"] & x_{m-1} \arrow["l_{m-1}"', loop, distance=2em, in=125, out=55]
\end{tikzcd}
\end{equation}
with $|a_{i,i+1}|=0$ and $|l_i|=n-2$. The dg algebra $\bf{A}_m$ is Morita-equivalent to the dg category ${B}_m$ with objects the vertices of $Q_m$ and morphisms freely generated by the arrows of $Q_m$ subject to the relations $l_1\circ a_{0,1}=0$, $a_{i+1,i+2}\circ a_{i,i+1}=0$ for $i\geq 0$ and $a_{i,i+1}l_i=l_{i+1}a_{i,i+1}$ for $i\geq 1$.
\end{remark}

For $m\geq 2$, we define $D_{m}$ to be the dg category with objects $z_1,\dots,z_m$ and morphisms freely generated by $b_{i,j}\colon z_i\rightarrow z_j$ for all $i\neq j$ in degree $j-i-1$ if $j>i$ and $n+j-i-1$ if $j<i$ and with differentials determined by
\begin{align*}
&d(b_{i,j})=\\
&\begin{cases} \sum_{i<k<j} (-1)^{j-k+1}b_{k,j} b_{i,k} & \text{if }j>i\,,\\ \sum_{i<k\leq m}(-1)^{j-k+n+1}b_{k,j} b_{i,k} + \sum_{1\leq k<j} (-1)^{j-k+1}b_{k,j}\circ b_{i,k}& \text{if }j<i\,.\end{cases}
\end{align*}
Note that if $m=n$, the dg category $D_m$ is Morita equivalent to the relative Ginzburg algebra of the $n$-gon, and is depicted in the cases $m=n=3$ and $m=n=4$ in \Cref{locGinzex}.

\begin{lemma}\label{dmhom}
The homology of the mapping complexes in $D_m$ is given by
\begin{align*}
\on{H}_\ast\on{Hom}_{D_m}(z_i,z_j)& \simeq \begin{cases} 0 & j\neq i,i+1\\ k[t_{n-2}] & j=i\\ k[t_{n-2}]& j=i+1, i\not =m\\
k[t_{n-2}][n-m]& j=1,i=m
\end{cases}    
\end{align*}
\end{lemma}

\begin{proof}
Let $c$ be a cycle in a morphism complex in $D_m$. We can decompose $c$ into the $k$-linear sum of morphisms composed of the generating morphisms of $D_m$. We call the length of $c$ the maximal number of generating morphisms appearing in a summand of $c$. We show the following statements via an induction on the length of $c$.
\begin{enumerate}[1)]
\item If $c\colon z_i\rightarrow z_i$, then $c$ is homologous to 
\begin{equation}\label{cycle1}
\lambda \left(\sum_{j<i} (-1)^{n+j-i}b_{j,i}b_{i,j}+ \sum_{i<j} (-1)^{j-i}b_{j,i}b_{i,j}\right)^l
\end{equation} with $\lambda\in k$ and $l\in \mathbb{N}$.
\item If $c\colon z_{i}\rightarrow z_{i+1}$, with $i+1$ considered modulo $m$, then $c$ is homologous to
\begin{equation}\label{cycle2}
\lambda b_{i,i+1}\left(\sum_{j<i} (-1)^{n+j-i}b_{j,i}b_{i,j}+ \sum_{i<j} (-1)^{j-i}b_{j,i}b_{i,j}\right)^l
\end{equation}
with $\lambda\in k$ and $l\in \mathbb{N}$.
\item Otherwise, $c$ is nullhomologous.
\end{enumerate} 
Note that a cycle cannot be a sum of the above morphisms for different $l$, as these are of different degrees. 

The cycles \eqref{cycle1} and \eqref{cycle2} define nonzero homology classes. One simple way to see this is to observe that their image in $C_m$, see the proof of \Cref{algmodprop} below, define nonzero homology classes. The assertion follows. 

We continue with showing 1),2) and 3). We denote the cycles of the form \eqref{cycle1} by $l_i$ and the cycles of the form \eqref{cycle2} by $b_{i,i+1}l_i$. We consider all indices $i,j$ of the $b_{i,j}$ modulo $m$.

We begin with 2), as it is the easiest case. We consider a cycle $c\colon z_i\rightarrow z_{i+1}$. Since the morphisms $b_{j,l}$ freely generate $D_m$, we can write $c$ as $c=\sum_{k\neq i+1}b_{k,i+1}u_{i,k}$ for some chains $u_{i,k}$ (of smaller length than $c$). The condition $d(c)=0$ implies $d(u_{i,i+2})=0$. By the induction assumption there exists a chain $v_{i,i+2}$, which can be chosen to be of smaller length than $u_{i,i+2}$, with $d(v_{i,i+2})=u_{i,i+2}$. We thus find 
\begin{align*}
&c+(-1)^{n-1}d(b_{i+2,i+1}v_{i,i+2})\\
& =\sum_{k\neq i+1,i+2}b_{k,i+1}\left(u_{i,k}+(-1)^{s_{k,i}+i-k}b_{i+2,k}v_{i,i+2}\right)\end{align*}
with $s_{k,i}=0$ if $i+2<k\leq m$ and $s_{k,i}=n$ if $1\leq k\leq i$. This shows that $c$ is homologous to a cycle $c_2=\sum_{k\neq i+1,i+2}b_{k,i+1}u_{i,k}$ for some other chains also denoted $u_{i,k}$. Repeating this argument $m-2$ times, we see that $c$ is homologous to a cycle of the form $b_{i,i+1}u_{i,i}$ and by the induction hypothesis we find $u_{i,i}=l_i+d(v_{i,i})$ for some chain $v_{i,i}$. It follows that $c$ is homologous to $b_{i,i+1}l_i$, showing 2).

For 3), we consider a cycle $c\colon z_i\rightarrow z_j$ with $j\neq i,i+1$ and assume without loss of generality that $i<j$. We write $c=\sum_{k\neq j}b_{k,j}u_{i,k}$ for some chains $u_{i,k}$ with $d(u_{i,j+1})=0$. If $j+1\neq i$, then by the induction assumption $u_{i,j+1}=d(v_{i,j+1})$. As in the case 2), we thus find that $c$ is homologous to a cycle of the form $\sum_{k\neq j,j+1}b_{k,j}u_{i,k}$ for some chains also labeled $u_{i,k}$. Repeating this process a few times, we find that $c$ is homologous to a cycle of the form $\sum_{k,\neq j,\dots,i-1}b_{k,j}u_{i,k}$ with $d(u_{i,i})=0$. Applying the induction assumption, we find that $u_{i,i}=l_i+d(v_{i,i})$ for some chain $v_{i,i}$. The condition $d(c)=0$ then implies that $(-1)^{j-i}b_{i+1,j}b_{i,i+1}l_i=b_{i+1,j}d(u_{i,i+1})$. 
Since $b_{i,i+1}l_i$ is not a boundary unless $l_i=0$, it follows that $l_i=0$. We obtain that $c$ is homologous to $\sum_{k\neq j,\dots,i}b_{k,j}u_{i,k}$, for some chains also labeled $u_{i,k}$ with $d(u_{i,i+1})=0$. If $j=i+2$, the assertion now follows and otherwise we argue as before to obtain that $c$ is homologous to $\sum_{k,\neq j,\dots,i,i+1}b_{k,j}u_{i,k}$ with $d(u_{i,i+2})=0$. The induction assumption implies that $u_{i,i+2}$ is a boundary, from which we obtain that $c$ is homologous to $\sum_{k\neq j,\dots,i+2} b_{k,j}u_{i,k}$ for some chains also labeled $u_{i,k}$. Repeating this argument a few times, we can finally conclude that $c$ is a boundary.

For 1), we consider a cycle $c\colon z_i\rightarrow z_i$, which we can write as $c=\sum_{k\neq i}b_{k,i}u_{i,k}$ for some chains $u_{i,k}$ with $d(u_{i,i+1})=0$. Using the induction assumption, we find a chain $v_{i,i+1}$ with $u_{i,i+1}=d(v_{i,i+1})-b_{i,i+1}l_i$. It follows that $c$ is homologous to a cycle of the form $c_1=\sum_{k\neq i,i+1}b_{k,i}u_{i,k}-b_{i+1,i}b_{i,i+1}l_i$ for some other chains also labeled $u_{i,k}$. This constitutes the base case for an induction on $j$ of the following assertion.

{\it For all $1\leq j\leq m-1$, the cycle $c$ is homologous to 
\begin{align*} 
c_j&=\sum_{k\in I}b_{k,i}u_{i,k}\\
&~ +\left(\sum_{1\leq k\leq i+j-m<i} (-1)^{n+k-i}b_{k,i}b_{i,k}+ \sum_{i<k\leq i+j,m} (-1)^{k-i}b_{k,i}b_{i,k}\right)l_i\,,\end{align*}
where $I$ is the set of $1\leq k\leq m$ such that $k>i+j$ or $k<i+j-m$ and the $u_{i,k}$ are some chains.}

For the induction step, we consider the case $i+j\leq m$. The case $i+j>m$ is dealt with analogously. Suppose that $c$ is homologous to $c_j$. Evaluating the condition $d(c_j)=0$ at the summands beginning with $b_{i+j+1,i}$ yields  
\begin{align*}
0&=(-1)^{n-j}b_{i+j+1,i}d(u_{i,i+j+1})\\
& ~ +\left(\sum_{i<k\leq i+j} (-1)^{n-j+k-i}b_{i+j+1,i}b_{k,i+j+1}b_{i,k}\right)l_i\,,
\end{align*}
so that by the induction hypothesis (of the induction over the length of $c$) we have $u_{i,i+j+1}=d(v_{i,i+j+1})+(-1)^{j+1}b_{i,i+j+1}l_i$ for some chain $v_{i,i+j+1}$. It follows that $c$ is homologous to $c_{j+1}$. This completes the induction step. Setting $j=m-1$, we obtain that $c$ is homologous to
\[ \left(\sum_{k<i} (-1)^{n+k-i}b_{k,i}b_{i,k}+ \sum_{i<k} (-1)^{k-i}b_{k,i}b_{i,k}\right)l_i\]
and thus of the form \eqref{cycle1}. This concludes the proof.
\end{proof}

\begin{proposition}\label{algmodprop}~
\begin{enumerate}[(1)]
\item There exists an equivalence of $\infty$-categories 
\begin{equation}\label{model3}
\mathcal{V}^m_{\phi^*}\simeq \mathcal{D}(D_m)\,.
\end{equation}
\item For $1\leq i \leq m-1$, the composite of the equivalence \eqref{model3} with the functor $\mathcal{D}(k[t_{n-2}])\simeq \on{RMod}_{k[t_{n-2}]}\xrightarrow{\varsigma_{m-i+1}} \mathcal{V}^m_{\phi^*}$ is equivalent to the image under $\mathcal{D}(\mhyphen)$ of the dg functor $\iota_i\colon k[t_{n-2}]\rightarrow D_m$ determined by 
\[ \iota_i(t_{n-2})=(-1)^{m+in}\left(\sum_{j<i} (-1)^{n+j-i}b_{j,i}b_{i,j}+ \sum_{i<j} (-1)^{j-i}b_{j,i}b_{i,j}\right)\,.\]
Furthermore, if $n=m$, the composite of the equivalence \eqref{model3} with the functor $\mathcal{D}(k[t_{n-2}])\simeq \on{RMod}_{k[t_{n-2}]}\xrightarrow{\varsigma_{1}} \mathcal{V}^m_{\phi^*}$ is equivalent to the image under $\mathcal{D}(\mhyphen)$ of the dg functor $\iota_m\colon k[t_{n-2}]\rightarrow D_m$ determined by 
\[ \iota_m(t_{n-2})=(-1)^{m}\left(\sum_{j<m} (-1)^{n+j-m}b_{j,m}b_{m,j}\right)\,.\]
\end{enumerate}
\end{proposition}

\begin{proof}
We recursively define objects $y_{i+1}=\on{cone}(y_i\xrightarrow{\alpha_i}x_i)$ for $i\geq 0$ in $\on{dgMod}({B}_m)$, where $y_1=x_0$ and for $i\geq 1$
\begin{align*} 
\alpha_i=(0,\dots,0,a_{{i-1},i})\in& \bigoplus_{j=0}^{i-1} \on{Hom}_{{B}_m}\left(x_{j}[i-j-1],x_i\right)\\
& \simeq \on{Hom}_{\on{dgMod}({B}_m)}(y_i,x_i)\end{align*}
where the splitting holds only on the level of graded $k$-modules. We denote by $\langle x_1,\dots,x_{m-1},y_m\rangle\subset \on{dgMod}({B}_m)$ the full dg subcategory spanned by $x_1,\dots,x_{m-1},y_m$. Note that $x_1,\dots,x_{m-1},y_m$ compactly generate $\mathcal{D}({B}_m)$ so that there exists an equivalence of $\infty$-categories 
\[\mathcal{D}(\langle x_1,\dots,x_{m-1},y_m\rangle)\simeq \mathcal{D}({B}_m)\,.\]
A direct computation shows that $\langle x_1,\dots,x_{m-1},y_m\rangle$ is quasi-equivalent to the dg category $C_m$ with objects $x_1,\dots,x_{m-1},y_m$, generated by the morphisms
\begin{itemize}
\item $a_{i,i+1}\colon x_i\rightarrow x_{i+1}$ in degree $0$,
\item $a_{i,m}\colon x_i\rightarrow y_m$ in degree $m-i-1$ and
\item $a_{m,i}\colon y_m\rightarrow x_i$ in degree $n-m+i-1$
\end{itemize}
subject to the relations $a_{i,i+1} a_{i-1,i}=0$ for $2\leq i \leq m-2$ and $a_{m,i}a_{j,m}=0$ for $i\neq j$ and with differentials determined on generators by
\begin{itemize}
\item $d(a_{i,i+1})=0$ for $1\leq i \leq m-1$ and $d(a_{m,1})=0$,
\item $d(a_{i,m})=(-1)^{m-i}a_{i+1,m} a_{i,i+1}$ for $i\neq m-1$,
\item $d(a_{m,i})=a_{i-1,i} a_{m,i-1}$ for $i\neq 1$.
\end{itemize} 
The morphisms $a_{i,m}$ and $a_{m,i}$ are the images under the quasi-equivalence $\langle x_1,\dots,x_{m-1},y_m\rangle\rightarrow C_m$ of
\[(0,\dots,\on{id}_{x_i},\dots,0)\in \bigoplus_{j=0}^{m-1}\on{Hom}_{B_m}(x_i,x_j[m-j-1])\simeq \on{Hom}_{\on{dgMod}({B}_m)}(x_i,y_m)\]
and
\begin{align}\begin{split}\label{lsgneq}
(0,\dots,(-1)^{i(n-1)}l_i,\dots,0)\in &\bigoplus_{j=0}^{m-1}\on{Hom}_{B_m}(x_j[m-j-1],x_i)\\
\simeq& \on{Hom}_{\on{dgMod}({B}_m)}(y_m,x_i)\,,
\end{split}\end{align}
respectively.
For example, for $m=4$, we can depict the generating morphisms of $C_m$ as follows.
\[
\begin{tikzcd}
                                                                           & x_2   \arrow[dd, "{a_{2,4}}"', bend right] \arrow[rd, "{a_{2,3}}", bend left]                                  &                                        \\
x_1 \arrow[rd, "{a_{1,4}}" description] \arrow[ru, "{a_{1,2}}", bend left] &                                                                                                                & x_3 \arrow[ld, "{a_{3,4}}", bend left] \\
                                                                           & y_4 \arrow[lu, "{a_{4,1}}", bend left] \arrow[ru, "{a_{4,3}}" description] \arrow[uu, "{a_{4,2}}", bend right] &                                       
\end{tikzcd}
\]
Using \Cref{dmhom}, we find that the dg functor $\mu_m\colon D_m\rightarrow C_m$ determined by
\begin{itemize}
\item $\mu_m(z_i)=x_i$ for $i\neq m$ and $\mu_m(z_m)=y_m$, 
\item $\mu_m(b_{i,j})=\begin{cases} a_{i,j} &\text{if }j=i+1\text{ or }i=m\text{ or }j=m\\ 0 & \text{else}\end{cases}$
\end{itemize}
is a quasi-equivalence. We thus find equivalences of $\infty$-categories 
\[\mathcal{D}(D_m)\xrightarrow{\mathcal{D}(\mu_m)}\mathcal{D}(C_m)\simeq  \mathcal{D}({B}_m)\simeq \mathcal{V}^m_{\phi^*}\,,\]
showing part (1).

We proceed with part (2). Inspecting the construction of the equivalence $\mathcal{D}({B}_m)\simeq \mathcal{V}^m_{\phi^*}$, one finds that for $1\leq i\leq m-1$ the functor $\varsigma_{m-i+1}$ is modeled by the dg functor $k[t_{n-2}]\rightarrow B_m$, determined by mapping $t_{n-2}$ to $l_{i}$, note for this also the commutative diagram in \cite[Prop.~2.39]{Chr22}. The commutative diagram of dg categories
\[
\begin{tikzcd}
              & {k[t_{n-2}]} \arrow[ld, "\iota_i"'] \arrow[rd, "t_{n-2}\mapsto l_{i}"] &                 \\
D_m \arrow[r] & C_m \arrow[r]                                                                                & \on{dgMod}(B_m)
\end{tikzcd}
\]
whose horizontal morphisms are Morita equivalences therefore shows that $\varsigma_{m-i+1}$ is modeled by $\iota_i$; note that the sign in $\iota_i$ follows from the sign $(-1)^{m-i}$ of the summand $b_{m,i}b_{i,m}$ in \eqref{cycle1} and the sign $(-1)^{i(n-1)}$ in \eqref{lsgneq}. In the case $n=m$, the remaining assertion that $\varsigma_1$ is modeled by $\iota_m$ is a consequence of the cyclic, rotational symmetry of $D_m$ and the sequence of adjunctions \eqref{sigmaadjeq}.
\end{proof}

\begin{proof}[Proof of \Cref{gluethm}]
Using \Cref{algmodprop}, the proof of \cite[Theorem 6.1]{Chr22} directly translates from the case $n=3$ to the general case.
\end{proof}

\section{Objects from curves}\label{sec:objsfromcurves}

In \Cref{sec:curves}, we introduce a class of curves used for the geometric model called matching curves. We proceed in \Cref{sec:objectsfromcurves} with the construction of the global sections associated with matching curves equipped with extra data, referred to as matching data. In \Cref{sec:projmodules}, we show that the projective $\mathscr{G}_\mathcal{T}$-modules associated with the vertices of the underlying quiver can be realized in terms of global sections associated with matching data with underlying pure matching curves. Finally, in \Cref{subsec:flips}, we show that a flip of the $n$-angulation yields an equivalence on the global sections of the associated perverse schobers and describe the action of this equivalence on the global sections associated with matching data. 
For the entirety of the section, we fix an integer $n\geq 2$.

\subsection{Matching curves}\label{sec:curves} 

We fix a marked surface ${\bf S}$ with an $n$-valent spanning graph $\mathcal{T}$. We also fix a base $\mathbb{E}_\infty$-ring spectrum $R$. For each vertex $v$ of $\mathcal{T}$, we have an immersion $\Sigma_v\rightarrow \Sigma_\mathcal{T}$, see \Cref{ssurfrem}. This immersion is an embedding if no edge of $\mathcal{T}$ incident to $v$ is a loop.

\begin{definition}\label{def:segm}
Let $v\in \mathcal{T}_0$. A segment at $v$ is an embedded curve $\delta\colon [0,1]\rightarrow \Sigma_v$, which does not hit $v$ away from the endpoints and which is of one of the following two types.
\begin{enumerate}[(1)]
\item One end lies at $v$, the other on the boundary of $\Sigma_v$.
\item Both ends lie on the boundary of $\Sigma_v$. In this case, since the segment $\delta$ is embedded, it wraps $1\leq a\leq n$ steps around the vertex $v$.
\end{enumerate}
We consider segments at $v$ as equivalence classes under homotopies relative $\partial \Sigma_v\cup \{v\}$.

A segment in $\Sigma_\mathcal{T}$ is a segment at any $v\in \mathcal{T}_0$. We do not distinguish in notation its representatives from the curves in $\Sigma_{\mathcal{T}}$ obtained from composing with the immersion $\Sigma_v\rightarrow \Sigma_\mathcal{T}$.
\end{definition}

The two types of segments are depicted in \Cref{mcfig2}.

\begin{figure}[ht]
\begin{center}
\begin{tikzpicture}[scale=2]
 \draw[color=blue][very thick] plot [smooth] coordinates {(0,0) (0.2,0.07) (1,0.07)};
 \draw[color=ao][very thick][densely dotted] plot [smooth] coordinates {(0.25,1) (0.35,0.35) (1,0.25)};
 \draw[color=ao][very thick][densely dotted] plot [smooth] coordinates {(0.25,-1) (0.35,-0.35) (1,-0.25)};
 \draw[color=ao][very thick][densely dotted] plot [smooth] coordinates {(-0.25,-1) (-0.35,-0.35) (-0.5,-0.3)};
 \draw[color=ao][very thick][densely dotted] plot [smooth] coordinates {(-0.25,1) (-0.35,0.35) (-0.5,0.3)};
 \draw[color=ao][very thick] plot [smooth] coordinates {(1,-0.25) (1,0.25)};
 \draw[color=ao][very thick] plot [smooth] coordinates {(-0.25,1) (0.25,1)};
 \draw[color=ao][very thick] plot [smooth] coordinates {(-0.25,-1) (0.25,-1)}; 
  \node (0) at (0,0){};  
  \node (1) at (-0.7,0.07){\vdots};
  \node (2) at (-0.2,0){\small $v$};
  \fill (0) circle (0.05);
  \draw[very thick]
  (0,0.99) -- (0,0)
  (0,-0.99) -- (0,0)
  (0,0) -- (0.99,0); 
\end{tikzpicture}
\hspace{4em}
\begin{tikzpicture}[decoration={markings, 
	mark= at position 0.6 with {\arrow{stealth}}}, scale=2]
 \draw[color=blue][very thick][postaction={decorate}] plot [smooth] coordinates {(0.1,-1) (0.2,-0.2) (1,-0.1)};
 \draw[color=blue][very thick][postaction={decorate}]  plot [smooth] coordinates {(-0.1,-1) (-0.1,-0.2) (-0.1,0.2) (-0.1,1)};

 \draw[color=ao][very thick][densely dotted] plot [smooth] coordinates {(0.25,1) (0.35,0.35) (1,0.25)};
 \draw[color=ao][very thick][densely dotted] plot [smooth] coordinates {(0.25,-1) (0.35,-0.35) (1,-0.25)};
 \draw[color=ao][very thick][densely dotted] plot [smooth] coordinates {(-0.25,-1) (-0.35,-0.35) (-0.5,-0.3)};
 \draw[color=ao][very thick][densely dotted] plot [smooth] coordinates {(-0.25,1) (-0.35,0.35) (-0.5,0.3)};
 \draw[color=ao][very thick] plot [smooth] coordinates {(1,-0.25) (1,0.25)};
 \draw[color=ao][very thick] plot [smooth] coordinates {(-0.25,1) (0.25,1)};
 \draw[color=ao][very thick] plot [smooth] coordinates {(-0.25,-1) (0.25,-1)}; 
  \node (0) at (0,0){};  
  \node (1) at (-0.7,0.07){\vdots};
  \node (2) at (-0.2,0){\small $v$};
  \fill (0) circle (0.05);
  \draw[very thick]
  (0,0.99) -- (0,0)
  (0,-0.99) -- (0,0)
  (0,0) -- (0.99,0); 
\end{tikzpicture}
\end{center}
\caption{A segment of the first type at $v$ (in blue, on the left) with unspecified direction and two segments of the second type at $v$ (in blue, on the right), wrapping around $v$ in the counterclockwise and clockwise direction by $a=1$ and $a=n-2$ steps, respectively.}\label{mcfig2}
\end{figure}

We can always assume that a given end of a segment which does not lie at a vertex of $\mathcal{T}_0$ ends on an edge $e$ of $\mathcal{T}$. This will be useful to specify at which boundary component of $\Sigma_v$ the segments begins or ends. If the segment is of the first type and ends at $e$, we also say that the segment exits the vertex through $e$.

\begin{definition}
Let $\delta$ be a segment in $\Sigma_\mathcal{T}$. If $\delta$ is of the first type, we define its degree as $d(\delta)=0$ and if $\delta$ is of the second type, wrapping $1\leq a \leq n$ steps around the vertex $v$, we define its degree as $d(\delta)=a-1$ if $\delta$ goes in the counterclockwise direction and as $d(\delta)=1-a$ if $\delta$ goes in the clockwise direction. We call a segment pure if $d(\delta)=0$.
\end{definition}

For $N\in \mathbb{N}=\{1,2,3,\dots\}$, we consider the sets $[N]\coloneqq \{1,\dots,N\}$ and $\mathbb{Z}/N\mathbb{Z}\coloneqq \{1,\dots,N+1\}/N+1\sim 1$. The set $[N]$ is ordered linearly, the set $\mathbb{Z}/N\mathbb{Z}$ has a cyclic order. 

\begin{definition}\label{def:compseg}
Let $I=\mathbb{Z}/N\mathbb{Z},[N],\mathbb{N},\mathbb{Z}$ for $N\in \mathbb{N}$. Consider a collection of representatives of segments $\{\delta_i\}_{i\in I}$ in $\Sigma_\mathcal{T}$ satisfying for all $i\in I$ and $i\neq N$ if $I=[N]$ that
\begin{itemize}
\item $\delta_{i+1}(0)=\delta_i(1)$ in $\Sigma_{\mathcal{T}}$.
\item if the segments $\delta_i$ and $\delta_{i+1}$ both lie at the same vertex $v$ of $ \mathcal{T}$, then $\delta_{i}(1)$ lies on a loop of $\mathcal{T}$ and the two points $\delta_{i}(1),\delta_{i+1}(0)$ lie on different boundary components of $\Sigma_v$, (this means that their composite wraps along the loop). 
\end{itemize}
Consider the curve $\gamma\colon U\rightarrow \Sigma_{\mathcal{T}}$ arising from composing the curves $\{\delta_i\}_{i\in I}$ in their given order. We further suppose that
\begin{itemize}
\item the curve $\gamma$ does not cut out (by tracing along a connected part of the curve) any unmarked discs in $\Sigma_{\mathcal{T}}$.
\item if $I=\mathbb{Z}/N\mathbb{Z}$, that $d(\gamma)=0$, see \Cref{def:regdeg}, and that $\gamma$ is not homotopic relative $\partial \Sigma_{\mathcal{T}} \cup\mathcal{T}_0$ to the composite of multiple identical closed curves in $\Sigma_{\mathcal{T}}$. 
\end{itemize}
We call the curve $\gamma\colon U\rightarrow \Sigma_{\mathcal{T}}$ a curve in $\Sigma_{\mathcal{T}}$ composed of segments. Being composed of segments is a property, the segments can be recovered from $\gamma$ by intersecting with the $\Sigma_{v}$'s.

Reparametrizing $\gamma$ if necessary, we can assume that $U=S^1$ if $I=\mathbb{Z}/N\mathbb{Z}$, $U=[0,1]$ if $I=[N]$, $U=[0,\infty)$ if $I=\mathbb{N}$ and $U=(-\infty,\infty)$ if $I=\mathbb{Z}$. If $U=S^1$, we call $\gamma$ closed. Otherwise, we call $\gamma$ open.
\end{definition}

\begin{figure}[ht]
\begin{center}
\begin{tikzpicture}
\draw[color=blue!255, very thick] plot [smooth] coordinates {(2,1) (3,0.5) (3.5,1) (3,1.5) (1,0.5) (0.5,1) (1,1.5) (2,1)};

  \draw[color=ao, very thick]
    (0, 0)
    (2, 2)
    (0, 2)
    (2, 0)
    (4, 0)
    (4, 2) 
    (0, 2) -- (0, 0)
    (2, 0) -- (0, 0)
    (0, 2) -- (2, 2)
    (2, 2) -- (4, 2)
    (4, 0) -- (4, 2)
    (2, 0) -- (4, 0); 
  \node (0) at (0, 0){};
  \node (1) at (0, 2){};
  \node (2) at (2, 0){};
  \node (3) at (2, 2){}; 
  \node (4) at (4, 0){};
  \node (5) at (4, 2){};  
  \fill[color=orange] (0) circle (0.1);
  \fill[color=orange] (1) circle (0.1);
  \fill[color=orange] (2) circle (0.1);
  \fill[color=orange] (3) circle (0.1); 
  \fill[color=orange] (4) circle (0.1); 
  \fill[color=orange] (5) circle (0.1);     

\fill (1,1) circle (0.1);
\fill (3,1) circle (0.1);

  \draw[very thick] (-0,1)--(4,1)
        (1,-0)--(1,2) (3,-0)--(3,2);

\end{tikzpicture}
\caption{An example of a closed curve in the $6$-gon, equipped with a $4$-valent spanning graph, which is not allowed as curves composed of segments. The curve cuts out two unpunctured discs (the enclosed vertices of $\mathcal{T}$ do not matter).}
\end{center}
\end{figure}

\begin{definition}\label{singfreedef}\label{def:regdeg}
Let $\gamma$ be a curve in $\Sigma_\mathcal{T}$ composed of segments.
\begin{enumerate}[(1)]
\item We call $\gamma$ regular, if it is only composed of segments of the second type which do not wrap around the vertices by $n$ steps. We call $\gamma$ singular, if it is not regular. 
\item Suppose that $\gamma$ is composed of finitely many segments $\delta_1,\dots,\delta_m$. We define the degree of $\gamma$ as 
\[ d(\gamma)=\sum_{i=1}^m d(\delta_i)\,.\] 
\end{enumerate}
\end{definition}

\begin{definition}\label{matchingcurve}~
\begin{itemize}
\item Let $U=S^1,[0,1],[0,\infty),(-\infty,\infty)$. Consider a curve $\gamma\colon U\rightarrow \Sigma_\mathcal{T}$ composed of segments, see \Cref{def:compseg}. The curve $\gamma$ is called a matching curve in $\Sigma_\mathcal{T}$ if for all $x\in \partial U$, $\gamma(x)$ lies in $\mathcal{T}_0$ or in $\partial \Sigma_\mathcal{T}$. We consider matching curves as equivalence classes under homotopies relative $\partial \Sigma_\mathcal{T}\cup \mathcal{T}_0$.
\item A matching curve in ${\bf S}\backslash M$ is defined to be a homotopy class relative $(\partial {\bf S}\backslash M)\cup \mathcal{T}_0$ of curves $U\rightarrow {\bf S}\backslash M$ which contains a representative given by the composite of a matching curve in $\Sigma_\mathcal{T}$ with the embedding $\Sigma_\mathcal{T}\rightarrow {\bf S}\backslash M$ of \Cref{ssurfrem}.
\item A matching curve in $\Sigma_{\mathcal{T}}$ or ${\bf S}\backslash M$ is called finite if it is composed of finitely many segments.
\item A matching curve in $\Sigma_{\mathcal{T}}$ or ${\bf S}\backslash M$ is called pure if it is only composed of pure segments.
\end{itemize}
\end{definition}

Note that matching curves do not intersect $\mathcal{T}_0$ nor the boundary of the surface except at the endpoints. 

\begin{remark}
The notion of matching curve in ${\bf S}\backslash M$ does not depend on the choice of $n$-valent spanning graph $\mathcal{T}$. An open curve in ${\bf S}\backslash M$ with endpoints in $(\partial {\bf S}\backslash M)\cup \mathcal{T}_0$, but which is away from its endpoints disjoint from $(\partial {\bf S}\backslash M)\cup \mathcal{T}_0$, arises as a matching curve if and only if it cuts out no discs and is not contractible to a point in $(\partial {\bf S}\backslash M)\cup \mathcal{T}_0$. For closed curves, additionally the degree needs to vanish.
\end{remark}

\begin{notation}\label{geqsegdef}
Let $\gamma$ be an open curve in ${\bf S}\backslash M$ composed of segments with index set $I$. Let $i\in I$ and $\delta_i$ be the corresponding segment of $\gamma$. We denote by $\gamma<\delta_i$ the curve obtained by the composite of the segments $\{\delta_j\}_{j\in I,j<i}$ and by $\gamma\leq \delta_i$ the curve obtained by the composite the segments $\{\delta_j\}_{j\in I,j\leq i}$. We similarly define the curves $\delta_i<\gamma$ and $\delta_i \leq \gamma$, and given two segments $\delta_i,\delta_{i'}$ of $\gamma$, the curves $\delta_i <\gamma <\delta_{i'}$, $\delta_i \leq \gamma <\delta_{i'}$, $\delta_i <\gamma \leq \delta_{i'}$ and  $\delta_i \leq \gamma \leq\delta_{i'}$. 
\end{notation}

\begin{lemma}\label{mcprop1}
Let $n=3$. There exists a bijection between
\begin{enumerate}[1)]
\item pure matching curves in ${\bf S}\backslash M$ and
\item curves in ${\bf S}\backslash M$ which do not cut out any discs in ${\bf S}\backslash M$ and whose endpoints lie in $\mathcal{T}_0$ or $\partial {\bf S}\backslash M$, considered modulo homotopies relative $\partial {\bf S}\backslash M$ which fix the endpoints in $\mathcal{T}_0$. 
\end{enumerate}
\end{lemma}

\begin{proof}
Let $v$ be a vertex of $\mathcal{T}$. In $\Sigma_v$, each segment of the second type wrapping around $v$ by two steps is homotopic relative $\partial \Sigma_v$ to a segment of the second type wrapping only a single step around $v$, which is thus pure. Note that this uses, that the homotopy is allowed to cross the vertex $v$. Similarly, each segment of the second type wrapping around $v$ by $3$ steps is homotopic relative $\partial \Sigma_v$ to a point. 

Given a curve in ${\bf S}\backslash M$ which cuts out no discs, it is homotopic relative $\partial {\bf S}\backslash M$ to a matching curve. It follows by the above that this curve is homotopy equivalent, relative $\partial {\bf S}\backslash M$ and fixing endpoints in $\mathcal{T}_0$, to a pure matching curve. We can thus produce from each curve as in 2) a pure matching curve in ${\bf S}\backslash M$. Conversely, any pure matching curve clearly defines a curve as in 2). These assignments are inverse bijections.
\end{proof}

\subsection{Objects from matching curves}\label{sec:objectsfromcurves}

We fix an $n$-valent spanning graph $\mathcal{T}$ of a marked surface ${\bf S}$ and an $\mathbb{E}_\infty$-ring spectrum $R$. 

\begin{definition}
A matching datum $(\gamma,L)$ in ${\bf S}\backslash M$ consists of 
\begin{itemize}
\item a matching curve $\gamma$ in ${\bf S}\backslash M$,
\item an object $L\in \on{RMod}_{R[t_{n-2}]}$ called the local value,
\item if $\gamma$ is singular an object $Q\in \on{RMod}_R$ and an equivalence $\phi^*(Q)\simeq L$, and
\item if $\gamma$ is closed, an integer $a\geq 1$ and a monodromy equivalence 
\[ \mu\in\on{Map}_{\on{RMod}_{R[t_{n-2}]}}(L^{\oplus a},L^{\oplus a})\,.\]
\end{itemize}
We further assume that $L$ satisfies a technical condition explained in \Cref{twshrem}, which is always fulfilled if $R=k$ is a commutative ring.

We call $(\gamma,L)$ open if $\gamma$ is open and closed if $\gamma$ is closed. We define the rank of $(\gamma,L)$ as $a$ if $\gamma$ is closed and as $1$ if $\gamma$ is open\footnote{The datum of a closed matching datum $(\gamma,L)$ of rank $a>1$ is equivalent to the datum of a closed matching datum $(\gamma,L^{\oplus a})$ of a rank $1$. The distinction is thus for reasons of convention.}.
\end{definition}

Given a segment $\delta$ and an object $L\in\on{RMod}_{R[t_{n-2}]}$ (also subject to the conditions in \Cref{twshrem}), we associate below a local section $M^L_\delta$ of the $\mathcal{T}$-parametrized perverse schober $\mathcal{F}_\mathcal{T}(R)$ from \Cref{constr:Ginzburgschober}. By gluing sections, we then produce for each matching datum $(\gamma,L)$ a global section $M_{\gamma}^L$ of $\mathcal{F}_\mathcal{T}(R)$, see \Cref{clprop2}.

\begin{remark}\label{twshrem}
For computational reasons, we always assume that $L\in \on{RMod}_{R[t_{n-2}]}$ satisfies $T_{\on{RMod}_{R[t_{n-2}]}}(L)\simeq L[1-n]$, where $T_{\on{RMod}_{R[t_{n-2}]}}$ is the cotwist functor of the spherical adjunction $\phi^*\dashv \phi_*$. If $R=k$ is a commutative ring, then this is satisfied for any $L$, as follows from \cite[Prop. 5.7]{Chr22}. If $R$ is arbitrary, then $L=\phi^*(R)$ and $L=R[t_{n-2}]$ also satisfy the requirement, see \Cref{lem:twist=n}.
\end{remark}

\begin{construction}\label{segcstr}
Let $p\colon \Gamma(\mathcal{F}_\mathcal{T}(R))\rightarrow \on{Exit}(\mathcal{T})$ be the Grothendieck construction of $\mathcal{F}_\mathcal{T}(R)$ and $\mathcal{L}$ the $\infty$-category of lax sections of $\mathcal{F}_\mathcal{T}(R)$ (i.e.~sections of $p$).

We consider a vertex $v\in \mathcal{T}_0$, with cyclically ordered incident halfedges $a_1,\dots,a_n$, which are part of the edges $e_1,\dots,e_n$.

{\bf 1)} For $1\leq i\leq n$, we denote by $\delta^i$ the segment of the first type at $v$ ending at $e_i\cap \partial \Sigma_v$.  Consider the paracyclic twist functor $T_{\mathcal{V}^n_{\phi^*}}$ from \Cref{cyctwrem}. Let $L\simeq \phi^*(Q)$. We define the section $M^{L}_{\delta^i}$ of $\mathcal{F}_\mathcal{T}(R)$ as the $p$-relative left Kan extension along the inclusion $\Delta^0\xrightarrow{v} \on{Exit}(\mathcal{T})$ of the functor $\Delta^0\rightarrow \mathcal{F}_\mathcal{T}(R)(v)\subset \mathcal{L}$ with value 
\[ M_{\delta^i}(v)=T_{\mathcal{V}^n_{\phi^*}}^{i-n}\left( Q\rightarrow 0\rightarrow \dots \right)[1-n]\in\mathcal{V}^n_{\phi^*}=\mathcal{F}_\mathcal{T}(R)(v)\,.\]

Spelling out the definition, one sees that the section $M_{\delta^i}$ is concentrated at the elements $v,e_i\in \on{Exit}(\mathcal{T})$ and takes up to equivalence the values 
\begin{align} 
\label{mdieq0}
 M_{\delta^i}^L(v)& \simeq \bigg(Q\xrightarrow{!} L\xrightarrow{\on{id}}\dots\xrightarrow{\on{id}} \underbrace{L}_{(n-i+1)\text{-th}}\rightarrow  0 \rightarrow \dots \rightarrow  0\bigg)[-i+1]\\
 & \in \mathcal{V}^n_{\phi^*}=\mathcal{F}_\mathcal{T}(R)(v)\nonumber \\
 M_{\delta^i}^L(e_i)&\simeq L\in \on{RMod}_{R[t_{n-2}]}=\mathcal{F}_\mathcal{T}(R)(e_i)\label{mdieq}
 \end{align}
and assigns to the edge $v\rightarrow e_i$ a coCartesian morphism, describing an apparent equivalence $\varrho_i(M_{\delta^i}(v))\simeq L\simeq M_{\delta}(e_i)$. The notation $\xrightarrow{\ast}$ and $\xrightarrow{!}$ in \eqref{mdieq0} and below refer to $p$-Cartesian and $p$-coCartesian morphisms, respectively, see also \Cref{sec:limits}.

{\bf 2)} For $1\leq i,j\leq n$ and $i\neq j$, we denote by $\delta^{i,j}$ the segment at $v$ which starts at $e_i\cap \partial \Sigma_v$ and ends at $e_j\cap \partial \Sigma_v$ going in the counterclockwise direction. For $L\in \on{RMod}_{R[t_{n-2}]}$, we define $M^L_{\delta^{i,j}}$ as the $p$-relative left Kan extension along $\Delta^0\xrightarrow{v}\on{Exit}(\mathcal{T})$ of the functor $\Delta^0\rightarrow \mathcal{F}_\mathcal{T}(R)(v)\subset \mathcal{L}$ with value 
\begin{align*} M^L_{\delta^{i,j}}(v)&= T_{\mathcal{V}^n_{\phi^*}}^{i-1}\bigg(0\rightarrow \dots \rightarrow 0\rightarrow \underbrace{L}_{(n-j+i+1)\text{-th}}\xrightarrow{\on{id}}\dots\xrightarrow{\on{id}} L\bigg)\\
& \in \mathcal{V}^n_{\phi^*}=\mathcal{F}_\mathcal{T}(R)(v)\,.\end{align*}

Spelling out the definition, one sees that the section $M^L_{\delta^{i,j}}$ is concentrated at $e_i,e_j$ and $v$. To the edges it assigns 
\begin{equation}\label{mdijeq}
 M^L_{\delta^{i,j}}(e_l)\simeq 
\begin{cases} 
L & l=i\,,\\
L[j-i-1] & l=j>i \,,\\
L[n+j-i-1] & l=j< i \,,\\
0 & \text{else.}
\end{cases}
\end{equation}
This uses that $T_{\on{RMod}_{R[t_{n-2}]}}(L)\simeq L[1-n]$, see \Cref{twshrem}.

The value of $M_{\delta^{i,j}}$ at $v$ is given as follows.
If $i<j$, we have
\begin{align*} M_{\delta^{i,j}}(v)\simeq & \bigg(0\rightarrow \dots \rightarrow 0\rightarrow \underbrace{L}_{(n-j+2)\text{-th}}\xrightarrow{\simeq}\dots\\
& ~~~ \dots \xrightarrow{\simeq}\underbrace{L}_{(n-i+1)\text{-th}}\rightarrow 0\rightarrow \dots \rightarrow 0 \bigg)[-i+1]\,,\end{align*}
if $j<i<n$
\begin{align*}
M_{\delta^{i,j}}(v)\simeq  &\bigg(\phi_*(L)\xrightarrow{!} \phi^*\phi_*(L)\xrightarrow{\simeq}\dots\xrightarrow{\on{\simeq}} \underbrace{\phi^*\phi_*(L)}_{(n-i+1)\text{-th}}\xrightarrow{\on{cu}_{L}}L\xrightarrow{\simeq} \dots \\
 & ~~~ \dots  \xrightarrow{\simeq}\underbrace{L}_{(n-j+1)\text{-th}}\rightarrow 0\rightarrow \dots \rightarrow 0 \bigg)[n-i]\,,
\end{align*}
where $\on{cu}_{L}$ denotes the counit map of the adjunction $\phi^*\dashv \phi_*$ at $L$, 
and in the case $j<i=n$ we have 
\[ M_{\delta^{n,j}}(v)\simeq \bigg(\phi_*(L)\xrightarrow{\ast} L\xrightarrow{\simeq}\dots \xrightarrow{\simeq} \underbrace{L}_{(n-j+1)\text{-th}}\rightarrow 0\rightarrow \dots \rightarrow 0 \bigg)\,.\]

\noindent {\bf 3)} Let $1\leq i\leq n$ and $L\simeq \phi^*(Q)$. Denote by $\delta^{i,i}$ the segment of the second type starting and ending at $e_i\cap \partial \Sigma_v$, wrapping around $v$ by $n$ steps in the counterclockwise direction. We define $M_{\delta^{i,i}}^L$ as the $p$-relative left Kan extension along $\Delta^0\xrightarrow{v}\on{Exit}(\mathcal{T})$ of the functor $\Delta^0\rightarrow \mathcal{F}_{\mathcal{T}}(R)(v)\subset \mathcal{L}$ with value 
\[ 
M_{\delta^{i,i}}^L(v) \simeq \bigg(\phi_*(L)\xrightarrow{!} \phi^*\phi_*(L)\xrightarrow{\simeq}\dots \xrightarrow{\simeq} \underbrace{\phi^*\phi_*(L)}_{(n-i+1)\text{-th}}\rightarrow 0\rightarrow \dots \rightarrow 0 \bigg)[n-i]
\]
if $i\neq n$ and value 
\[ M_{\delta^{n,n}}^L(v)\simeq\bigg (\phi_*(L)\rightarrow 0\rightarrow \dots \rightarrow 0\bigg)\]
if $i=n$. One finds apparent equivalences 
\[M_{\delta^{i,i}}(e_l)\simeq \begin{cases} \phi^*\phi_*(L)[n-1]\simeq \phi^*\phi_!(L)  & l=i\\ 0 & l\neq i\end{cases}\]
We have a splitting $\phi^*\phi_!(L)\simeq \phi^*\phi_!\phi^*(Q)\simeq \phi^*(Q)\oplus \phi^*(Q)[n-1]\simeq L\oplus L[n-1]$ arising from the following diagram: 
\[
\begin{tikzcd}[column sep=large]
\phi^*(Q) \arrow[rd, "\square", phantom] \arrow[r, "\on{u}_{\phi^*(Q)}"] \arrow[d] & \phi^*\phi_!\phi^*(Q) \arrow[d] \arrow[r, "\phi^*\circ \on{cu}_{Q}"] \arrow[rd, "\square", phantom] & \phi^*(Q) \arrow[d] \\
0 \arrow[r]                                                                          & {\phi^*(Q)[n-1]} \arrow[r]                                                                                  & 0                  
\end{tikzcd}
\]
Above $\on{u}$ and $\on{cu}$ denote the unit and counit of $\phi_!\dashv \phi^*$, respectively.
\end{construction}

\begin{remark}\label{degrem2}
Consider a segment $\delta^{i,j}$ of the second type with $i\neq j$. The suspensions arising in the definition of the $M_{\delta^{i,j}}^L$ in \eqref{mdijeq} correspond to the degree of $\delta^{i,j}$. Namely, one finds $M^L_{\delta^{i,j}}(e_j)\simeq L[d(\delta^{i,j})]$ if $i\neq j$ and $M^L_{\delta^{i,i}}(e_j)=L\oplus L[d(\delta^{i,j})]$ if $i=j$.
\end{remark}

\begin{construction}\label{constr:open1}\label{clprop1}
Let $\gamma$ be an open curve composed of segments in ${\bf S}\backslash M$. Let $L\in \on{RMod}_{R[t_{n-2}]}$. In the following we give the construction of a section $M^L_{\gamma}\in \mathcal{L}$ of $\mathcal{F}_\mathcal{T}(R)$. If $L=\phi^*(R)$, we will also use the notation $M_{\gamma}=M_{\gamma}^{\phi^*(R)}$.

Let $I$ be the index set of the segments of $\gamma$. We denote (if existent) the minimal and maximal elements of $I$ by $\on{min}(I),\on{max}(I)$. We set $I'=I\backslash \{\on{max}(I)\}$, and if $I$ has no maximal element, we agree that $\{\on{max}(I)\}=\emptyset$. We set $I''=I\backslash \{\on{min}(I),\on{max}(I)\}$, and if $I$ has no minimal element, we again set $\{\on{min}(I),\on{max}(I)\}=\emptyset$. Recall that the segments of $\gamma$ are denoted by $\delta_i$ with $i\in I$ (ordered compatibly with their appearance in $\gamma$). We denote the edge of $\mathcal{T}$ where $\delta_i$ begins by $e^i$ and the edge where $\delta_i$ ends by $e^{i+1}$. For later use, we also denote by $v^i\in \mathcal{T}_0$ the vertex at which $\delta_i$ lies. We define $M^L_{\gamma}$ as the colimit of a diagram $D_{\gamma}$ in the $\infty$-category $\mathcal{L}$ of lax sections of $\mathcal{F}_\mathcal{T}(R)$, see \Cref{glsecdef}, which is given as follows. 

The domain of $D_{\gamma}$ is the coequalizer $E_{\gamma}$ in the $1$-category of simplicial sets of the diagram
\[ 
\begin{tikzcd}[column sep=7em]
\coprod_{i\in I''}\Delta^0 \arrow[r, "\coprod_{i\in I''}\Delta^{\{1\}}\times \{i\}", shift left] \arrow[r, "\coprod_{i\in I''}\Delta^{\{2\}}\times \{i-1\}"', shift right] & \coprod_{i\in I'} \Lambda^2_0\times \{i\}
\end{tikzcd}
\]
where the horn $\Lambda^{2}_0\times \{i\}\simeq \Lambda^2_0$ is the poset with objects $(0,i),(1,i),(2,i)$ and morphisms $(0,i)\rightarrow (1,i),(2,i)$, i.e.~a span. For $l=1,2$ and $j\in I'$, the morphism $\Delta^{\{l\}}\times\{j\}$ is the inclusion $\Delta^0\rightarrow \amalg_{i\in I'}\Lambda^2_0\times\{i\}$ determined by mapping $0\in (\Delta^0)_0$ to $(l,j)\in \Lambda^2_0\times\{j\}\subset  \amalg_{i\in I'}\Lambda^2_0\times\{i\}$. Recall that for each segment $\delta_i$, we defined an associated section $M^L_{\delta_i}\in \mathcal{L}$ in \Cref{segcstr}. We further denote by $Z^L_{e^j}\in \mathcal{L}$ the lax section concentrated at $e^j$ with value $L$. The diagram $D_{\gamma}$ is determined via its restrictions to $\Lambda^2_0\times \{i\}$ with $i\in I'$, which is for $i\geq 1$ given by
\[ 
\begin{tikzcd}[row sep=small, column sep=0]
                                   & {Z^L_{e^i}[d(\delta_1\leq \gamma<\delta_{i+1})]} \arrow[rd, "\beta"] \arrow[ld, "\alpha"'] &                                            \\
{M^L_{\delta_i}[d(\delta_1\leq \gamma<\delta_i)]} &                                                         & {M^L_{\delta_{i+1}}[d(\delta_1\leq\gamma<\delta_{i+1})]}
\end{tikzcd}
\]
and for $i\leq 0$ given by
\[
\begin{tikzcd}[row sep=small, column sep=0]
                                   & {Z^L_{e^i}[-d(\delta_{i+1}\leq \gamma<\delta_{1})]} \arrow[rd, "\beta"] \arrow[ld, "\alpha"'] &                                            \\
{M^L_{\delta_i}[-d(\delta_{i}\leq \gamma<\delta_{1})]} &                                                         & {M^L_{\delta_{i+1}}[-d(\delta_{i+1}\leq \gamma<\delta_{1})]}
\end{tikzcd}
\]
where $\alpha$ and $\beta$ are the apparent (pointwise in $\on{Exit}(\mathcal{T})$) inclusions.
\end{construction}

\begin{construction}\label{constr:closed}
Let $(\gamma,L)$ be a closed matching datum in ${\bf S}\backslash M$ of rank $a$ with monodromy equivalence $\mu$. Let $\eta$ be the open curve composed of the segments $\delta_1,\dots,\delta_N$ of $\gamma$, where $\delta_N$ and $\delta_1$ have not been composed. Let $e$ be the edge of $\mathcal{T}$ where the curve $\eta$ starts and ends. We consider the section $M^L_{\eta}$ from \Cref{constr:open1}. We define $M^L_{\gamma}$ as the coequalizer in $\mathcal{L}$ of
\begin{equation}\label{s1obeq}
\begin{tikzcd}
(Z^L_{e})^{\oplus a} \arrow[r, "\iota", shift left] \arrow[r, "\iota'\circ \mu^{-1}"', shift right] & (M^L_{\eta})^{\oplus a}
\end{tikzcd}
\end{equation}
where $\iota$ and $\iota'$ are the morphisms with support at $e$ determined by the inclusions of $L^{\oplus {a}}$ into $M^L_{\eta}(e)^{\oplus a}$ arising from the two ends of $\eta$ at $e$ and $\mu\colon L^{\oplus a}\rightarrow L^{\oplus a}$ is the monodromy equivalence. We note that upon changing the basepoint $1\in I=\mathbb{Z}/N\mathbb{Z}$, i.e.~relabeling the elements of $I$ but keeping the cyclic order, the curve $\gamma$ does not change. One can show, that the section $M_{\gamma}^L$ also only changes up to equivalence by such a relabeling.

The section $M_{\eta}^L$ is obtained by gluing together local sections defined via Kan extensions. Making different choices of these Kan extensions yields a different section $M_{\gamma}^L$, which however only differs by a change in the monodromy equivalence $\mu$ given by composition with an invertible diagonal $a\times a$-matrix with entries in $\on{Map}_{R[t_{n-2}]}(L,L)$. By varying over all $\mu$, we thus always construct the same class of objects, independent of the choices of Kan extensions. One can remove this ambiguity by fixing choices of these Kan extensions.
\end{construction}

\begin{proposition}\label{clprop2}
Let $(\gamma,L)$ be matching datum. The section $M_{\gamma}^L$ of $\mathcal{F}_\mathcal{T}(R)$ defined in \Cref{constr:open1} or \Cref{constr:closed} is a global section.
\end{proposition}

\begin{proof}
One needs to show that $M_{\gamma}^L$ is a coCartesian section. This directly follows from unraveling the construction of $M_{\gamma}^L$ and is left to the reader.
\end{proof}

\begin{remark}
It is apparent that \Cref{segcstr} can be translated to a construction, which associates for an arbitrary perverse schober $\mathcal{F}$ with generic stalk $\mathcal{N}\in \on{St}$ to segments of the second type and choice of object $L\in \mathcal{N}$ a local section of $\mathcal{F}$. Additional data is required to associate local sections to segments of the first type. One can then proceed to glue these local sections to associate global sections to certain curves, as in the above constructions.
\end{remark}

\begin{remark}
Given a matching datum $(\gamma,L)$, such that $\phi_*(L)\simeq 0$, the global section $M_{\gamma}^L$ does not change under homotopies of $\gamma$ which fix the endpoints but are allowed to cross the vertices of $\mathcal{T}$. This follows from inspecting \Cref{segcstr} and noting that $M_{\delta^{i,j}}(v)\simeq M_{\delta^{j,i}}(v)$. Notably, in the case $R=k$ a commutative ring, this applies to $L=k[t_n^\pm]$. 

In the case $n=3$, there is thus a canonical matching datum $(\gamma',L)$ with $\gamma'$ pure such that $M_{\gamma'}^L\simeq M_{\gamma}^L$, see \Cref{mcprop1}.
\end{remark}

\begin{example}\label{glsecex}

We illustrate \Cref{clprop1} and \Cref{constr:closed} in an example with $R=k$ and $n=3$. Consider the once-punctured $3$-gon (in green), with the puncture called $p$ and the ideal triangulation depicted in black.  

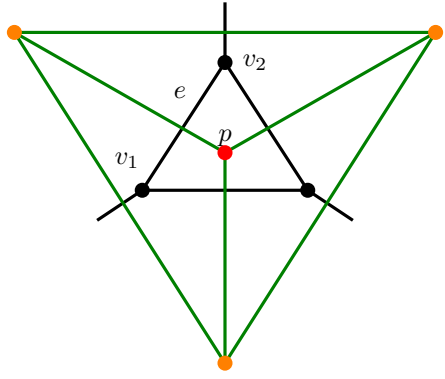
\begin{center}
\begin{tikzpicture}[scale=1]
  \node (0) at (0,0){};
  \node ()  at (0,0.3){$p$}; 
  \node (1) at (0,1.2){};
  \node (2) at (-1.1,-0.5){};
  \node (3) at (1.1,-0.5){};
  \node (4) at (2.8,1.6){};
  \node (5) at (-2.8,1.6){};  
  \node (6) at (0,-2.8){};
  \node ()  at (0.4,1.2){$v_2$}; 
  \node ()  at (-1.3,-0.1){$v_1$}; 
  \node ()  at (-0.6,0.8){$e$};
  \fill (1) circle (0.1);
  \fill (2) circle (0.1);
  \fill (3) circle (0.1);
  
  \draw[very thick]
  (0,1.2) -- (0,2)
  (-1.1,-0.5) -- (-1.7,-0.9)
  (1.1,-0.5) -- (1.7,-0.9)
  (-1.1,-0.5) -- (1.1,-0.5)
  (1.1,-0.5) -- (0,1.2)
  (0,1.2) -- (-1.1,-0.5); 
  
  \draw[color=ao, very thick] 
  (2.8,1.6) -- (-2.8,1.6)
  (-2.8,1.6) -- (0,-2.8)
  (0,-2.8) -- (2.8,1.6)
  (0,0) -- (2.8,1.6)
  (0,0) -- (-2.8,1.6)
  (0,0) -- (0,-2.8);
  \fill[color=orange]  (4) circle (0.1);
  \fill[color=orange]  (5) circle (0.1);
  \fill[color=orange]  (6) circle (0.1);
  \fill[color=red] (0) circle (0.1);
\end{tikzpicture}
\end{center}

The perverse schober $\mathcal{F}_\mathcal{T}(k)$ is up to natural equivalence given by the following diagram,

\[
\begin{tikzcd}
                 &                                                                                          & \mathcal{N}_{\phi^*}                                                                         &                                                                                          &                  \\
                 & \mathcal{N}_{\phi^*}                                                                         & \mathcal{V}^3_{\phi^*} \arrow[l, "\varrho_1"] \arrow[r, "\varrho_2"] \arrow[u, "\varrho_3"] & \mathcal{N}_{\phi^*}                                                                         &                  \\
\mathcal{N}_{\phi^*} & \mathcal{V}^3_{\phi^*} \arrow[u, "\varrho_2"] \arrow[l, "\varrho_3"] \arrow[r, "\varrho_1"] & \mathcal{N}_{\phi^*}                                                                         & \mathcal{V}^3_{\phi^*} \arrow[u, "\varrho_1"] \arrow[r, "\varrho_3"] \arrow[l, "\varrho_2"] & \mathcal{N}_{\phi^*}
\end{tikzcd}
\]
where $\mathcal{V}^3_{\phi^*}$ denotes the value of $\mathcal{F}_{\mathcal{T}}(k)$ at the vertices and $\mathcal{N}_{\phi^*}=\on{RMod}_{k[t_1]}\simeq \mathcal{D}(k[t_1])$ denotes the value of $\mathcal{F}_{\mathcal{T}}(k)$ at the edges of $\mathcal{T}$. Algebraically, one can describe $\mathcal{V}^3_{\phi^*}$ as the derived $\infty$-category of the relative Ginzburg algebra of the $3$-gon, which we denote in the following by $\mathscr{G}_{\Delta}$. For a depiction of $\mathscr{G}_{\Delta}$ see \Cref{locGinzex}.

The singular matching curve $\gamma_e$ given by the edge $e$ connecting $v_1$ and $v_2$ gives rise to the following global section $M_{\gamma_e}$, which describes a $3$-spherical object. 
\[
\begin{tikzcd}
  &                                                                                     & 0                                                                                   &                                                                        &   \\
  & \phi^*(k)                                                                              & s_1 \arrow[l] \arrow[r] \arrow[u] & 0                                                                      &   \\
0 & s_2 \arrow[u] \arrow[l] \arrow[r] & 0                                                                                   & 0 \arrow[u]\arrow[r] \arrow[l] & 0
\end{tikzcd}
\]
Above, we denote $s_1=(k\xrightarrow{\ast} \phi^*(k)\xrightarrow{\on{id}}\phi^*(k)),\, s_2=(k\xrightarrow{\ast}\phi^*(k)\rightarrow 0)\in \mathcal{V}^3_{\phi^*}$. Algebraically, $s_1$ and $s_2$ describe simple modules (in the standard heart), each associated with a vertex of the underlying quiver, and $\phi^*(k)$ is the unique $k[t_1]$-module with underlying chain complex $k$. 

Consider any module $L\in \on{RMod}_{k[t_1]}$, and a matching datum $(\gamma,L)$ of rank $1$ with $\gamma$ the closed pure matching curve wrapping around $p$. The associated global section $M_{\gamma}^L$ is of the form 
\[
\begin{tikzcd}
  &                                                                                      & 0                                                                                    &                                                                                      &   \\
  & L                                                                               & \iota_3(L) \arrow[l] \arrow[r] \arrow[u] & L                                                                               &   \\
0 & \iota_3(L) \arrow[u] \arrow[l] \arrow[r] & L                                                                               & \iota_3(L) \arrow[u] \arrow[r] \arrow[l] & 0
\end{tikzcd}
\]
where $\iota_3(L)=(0\rightarrow 0\rightarrow L)\in \mathcal{V}^3_{\phi^*}$ is the object concentrated in the third component of the semiorthogonal decomposition of $\mathcal{V}^3_{\phi^*}$ with value $L$. Algebraically, one can describe $\iota_3(L)$ as the derived tensor product $L\otimes_{k[t_{1}]}^Lp_e\mathscr{G}_\Delta$ with the projective $\mathscr{G}_\Delta$-module associated to some vertex of the quiver underlying $\mathscr{G}_\Delta$.
\end{example}

\subsection{Projective modules via pure matching curves}\label{sec:projmodules}

We fix a marked surface ${\bf S}$ with an $n$-valent spanning graph $\mathcal{T}$ and an $\mathbb{E}_\infty$-ring spectrum $R$.

Let $e$ be an edge of $\mathcal{T}$ and $v_1,v_2$ the vertices incident to $e$. Consider the two curves composed of segments $c_e^i$ with $i=1,2$ whose first segment each lies at $v_i$ and begins at $e$, and whose segments are all pure of the second type, wrapping exactly one step in the counterclockwise direction around a vertex of $\mathcal{T}$. We define the curve $c_e$ as the composite of $c_e^1$ with the curve obtained by reversing the orientation of $c_e^2$. Note that $c_e$ is a pure, regular matching curve. 

\begin{example}\label{ceiex}
Below, we depict the $4$-gon with an ideal triangulation $\mathcal{T}$ with edges $e_1,\dots,e_5$ and associated matching curves $c_{e_1},\dots,c_{e_5}$.
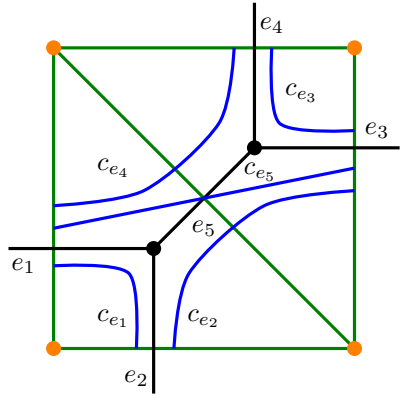
\begin{center}
\begin{tikzpicture}
  \draw[color=ao, very thick]
    (0, 4) -- (0, 0)
    (4, 0) -- (0, 0)
    (4, 4) -- (4, 0)
    (0, 4) -- (4, 4)
    (0, 4) -- (4, 0)
    ;
  \draw[very thick]
   (1.33,1.33) -- (2.67,2.67)
   (1.33,1.33) -- (1.33,-0.6)
   (1.33,1.33) -- (-0.6,1.33)
   (2.67,2.67) -- (4.6,2.67)
   (2.67,2.67) -- (2.67,4.6);
   \draw[color=blue][very thick] plot [smooth] coordinates {(0,1.9) (1.2, 2.1) (2.2, 3) (2.4, 4)};
   \draw[color=blue][very thick] plot [smooth] coordinates {(4,2.1) (2.8,1.9) (1.8, 1) (1.6, 0)};
   \draw[color=blue][very thick] plot [smooth] coordinates {(0,1.6) (4, 2.4)};
   \draw[color=blue][very thick] plot [smooth] coordinates {(0,1.1) (1, 1) (1.1, 0)};
   \draw[color=blue][very thick] plot [smooth] coordinates {(4,2.9) (3, 3) (2.9, 4)};
  \node (0) at (2.67, 2.67){};
  \node (1) at (1.33, 1.33){};
  \node (2) at (0,0){};
  \node (3) at (4,0){};
  \node (4) at (4,4){};
  \node (5) at (0,4){};
  
  \node () at (-0.4,1.1){$e_1$};
  \node () at (1.1,-0.4){$e_2$};
  \node () at (4.3,2.9){$e_3$};
  \node () at (2.9,4.3){$e_4$};
  \node () at (2,1.6){$e_5$};
  
  \node () at (0.8,0.4){$c_{e_1}$};
  \node () at (2,0.4){$c_{e_2}$};
  \node () at (3.3,3.4){$c_{e_3}$};
  \node () at (0.8,2.4){$c_{e_4}$};
  \node () at (2.75,2.35){$c_{e_5}$};
  \fill (0) circle (0.1);
  \fill (1) circle (0.1);
  \fill[color=orange] (2) circle (0.1);
  \fill[color=orange] (3) circle (0.1);
  \fill[color=orange] (4) circle (0.1);
  \fill[color=orange] (5) circle (0.1);
\end{tikzpicture}
\end{center}
\end{example}

\begin{proposition}\label{geomprprop}
For each $L\in \on{RMod}_{R[t_{n-2}]}$, there exists an equivalence in $\mathcal{H}(\mathcal{T},\mathcal{F}_\mathcal{T}(R))$ 
\begin{equation}\label{geomprpropeq}
 M_{c_e}^{L}\simeq \on{ev}^*_e(L)
\end{equation}
with $\on{ev}^*_e$ the functor from \Cref{compgenprop}.

In particular, the section $M_{c_e}^{R[t_{n-2}]}$ is the direct summand of the compact generator of the $\infty$-category $\mathcal{H}(\mathcal{T},\mathcal{F}_\mathcal{T}(R))$ associated with the edge $e$. 
\end{proposition}

\begin{proof}
Consider the curve $c_{e}$ as above and let $I$ be the index set of its segments. The two first segments of the curves $c_e^1$ and $c_e^2$, lying at $v_1$, respectively, $v_2$, yield segments $\delta_{x}$ and $\delta_{x+1}$ of $c_e$ with $x\in I$.
 
Consider the $R$-linear $\infty$-categories $\mathcal{C}=\mathcal{H}(\mathcal{T},\mathcal{F}_\mathcal{T}(R))$ and $\mathcal{L}$ of global sections, respectively, lax sections of $\mathcal{F}_\mathcal{T}$. Let $v\in \mathcal{T}_0$ be a vertex with incident edges labeled $e_1,\dots,e_n$ and let $\delta^{i,i+1}$ be the pure segment of the second type lying at $v$ passing from $e_i$ to $e_{i+1}$. The functor $\on{Mor}_{\mathcal{L}}(M^{L}_{\delta^{i,i+1}},\mhyphen)\colon \mathcal{L}\rightarrow \on{RMod}_R$ is equivalent to the functor 
\[ \widetilde{\on{ev}}_{v,i}^L\colon \mathcal{L}\xrightarrow{\on{ev}_v}\mathcal{V}^n_{\phi^*}\xrightarrow{\varrho_i} \on{RMod}_{R[t_{n-2}]}\xrightarrow{\on{Mor}(L,\mhyphen)}\on{RMod}_R\,,\]
as can be seen using that $M^{L}_{\delta^{i,i+1}}$ is a $p$-relative left Kan extension of its restriction to $v$. Similarly, for each edge $e$ of $\mathcal{T}$, we find the functor $\on{Mor}_{\mathcal{L}}(Z_e^{L},\mhyphen)\colon \mathcal{L}\rightarrow \on{RMod}_R$ to be equivalent to 
\[ \widetilde{\on{ev}}_e^L\colon \mathcal{L}\xrightarrow{\on{ev}_e}\on{RMod}_{R[t_{n-2}]}\xrightarrow{\on{Mor}(L,\mhyphen)}\on{RMod}_R\,.\]
Note that the composites of $\widetilde{\on{ev}}_{v,i}$ and $\widetilde{\on{ev}}_{e_i}$ with the inclusion $\mathcal{C}\rightarrow \mathcal{L}$ are equivalent, we denote this functor by $\on{ev}_{e_i}^L$. 

Using the definition of $M_{c_e}^L$ and that $\on{Mor}_{\mathcal{L}}(\mhyphen,\mhyphen)$ preserves limits in the first entry, we obtain that $\on{Mor}_{\mathcal{L}}(M^{L}_{c_e},\mhyphen)\colon \mathcal{L}\rightarrow \on{RMod}_R$ is given by the limit of the diagram $\on{Mor}_{\mathcal{L}}(D_{c_e},\mhyphen)\colon E_{c_e}^{\on{op}}\rightarrow \on{Fun}(\mathcal{L},\on{RMod}_R)$. Composing with the limit preserving pullback functor $\on{Fun}(\mathcal{L},\on{RMod}_R)\rightarrow \on{Fun}(\mathcal{C},\on{RMod}_R)$ along the inclusion $\mathcal{C}\rightarrow \mathcal{L}$, one obtains the diagram in $\on{Fun}(\mathcal{C},\on{RMod}_R)$, which assigns, up to equivalence, to $\Lambda^2_0\times \{x\}$ the constant diagram with value ${\on{ev}}^L_e$, to $(0,j)\rightarrow (2,j)$ for $j> x$ and to $(0,j)\rightarrow (1,j)$ for $j<x$ the identity on ${\on{ev}}^L_{e^{j-1}}$, where $e^{j-1}$ is the edge where $\delta^{j-1}$ ends and $\delta^{j}$ begins. The limit $\on{Mor}_\mathcal{C}(M_{c_e}^L,\mhyphen)$ is thus equivalent to the functor ${\on{ev}}^L_e$. Evaluating the left adjoints at $R$ shows the desired equivalence \eqref{geomprpropeq}.
\end{proof}

\subsection{Derived equivalences arising from flips of the \texorpdfstring{$n$}{n}-angulation}\label{subsec:flips}

In this section, we construct derived equivalences between the global sections of the perverse schobers $\mathcal{F}_{\mathcal{T}}(R)$ arising from changing the ideal $n$-angulation by a flip of an edge. The construction is a direct generalization of the one given in \cite[Section 6.4]{Chr22} for the global sections of $\mathcal{F}_\mathcal{T}(k)$ in the case $n=3$. In the case $R=k$, we thus obtain derived equivalences between relative Ginzburg algebras. Related results on derived equivalences of (relative) Ginzburg algebras from mutations of quivers with potentials were obtained in \cite{KY11,Wu21b}. We then further describe these derived equivalences in terms of the partial geometric model.

Given a decomposition of a surface into $n$-gons and an edge $e$ of one of the $n$-gons which is not a loop, there are $n-1$ possible \textit{flips} of the decomposition, which are obtained by replacing $e$ by another (possibly identical) diagonal contained in the $(2n-2)$-gon formed by the two adjacent $n$-gons of $e$. An example of a flip of a decomposition into $4$-gons is depicted in \Cref{flipfig}. We note that in the case $n=2$, there is the unique trivial flip.

\begin{figure}[ht]
\begin{center}
\begin{tikzpicture}[scale=0.85]
  \draw[color=ao, very thick]
    (0, 0)
    (2, 2)
    (0, 2)
    (2, 0)
    (4, 0)
    (4, 2) 
    (0, 2) -- (0, 0)
    (2, 0) -- (0, 0)
    (2, 2) -- (2, 0)
    (0, 2) -- (2, 2)
    (2, 2) -- (4, 2)
    (4, 0) -- (4, 2)
    (2, 0) -- (4, 0); 
  \node (0) at (0, 0){};
  \node (1) at (0, 2){};
  \node (2) at (2, 0){};
  \node (3) at (2, 2){}; 
  \node (4) at (4, 0){};
  \node (5) at (4, 2){};  
  \fill[color=orange] (0) circle (0.1);
  \fill[color=orange] (1) circle (0.1);
  \fill[color=orange] (2) circle (0.1);
  \fill[color=orange] (3) circle (0.1); 
  \fill[color=orange] (4) circle (0.1); 
  \fill[color=orange] (5) circle (0.1);     

  \node (21) at (1,1){\large $\times$};
  \node (22) at (3,1){\large $\times$};

  \draw[very thick] (-0.3,1)--(0.7,1) (1.3,1)--(2.7,1) (3.3,1)--(4.3,1)
        (1,1.3)--(1,2.3) (1,0.7)--(1,-0.3) (3,1.3)--(3,2.3) (3,0.7)--(3,-0.3);

  \node (23) at (10,1){\large $\times$};
  \node (24) at (8,1){\large $\times$};

  \draw[very thick] (6.7,1)--(7.7,1) (8.3,1)--(9.7,1) (10.3,1)--(11.3,1)
        (9.7,1.2)--(7.7,2.3) (8,0.7)--(8,-0.3) (10,1.3)--(10,2.3) (8.3,0.8)--(10.3,-0.3);

  \node (A) at (5, 1){};
  \node (B) at (6, 1){};
  \node at (5.5, 1.25){flip};
  \draw [->] (A) edge (B);  
  
  \draw[color=ao, very thick]
    (7, 0)
    (9, 2)
    (7, 2)
    (9, 0) 
    (7, 2) -- (7, 0)
    (9, 0) -- (7, 0)
    (7, 2) -- (9, 2)
    (11, 2) -- (11, 0)
    (9, 0) -- (11, 0)
    (9, 2) -- (11, 2)
    (11, 0) -- (7, 2); 
  \node (5) at (7, 0){};
  \node (6) at (7, 2){};
  \node (7) at (9, 0){};
  \node (8) at (9, 2){};   
  \node (9) at (11, 2){};
  \node (10) at (11, 0){};    
  \fill[color=orange] (6) circle (0.1);
  \fill[color=orange] (7) circle (0.1);
  \fill[color=orange] (8) circle (0.1);
  \fill[color=orange] (9) circle (0.1); 
  \fill[color=orange] (10) circle (0.1); 
  \fill[color=orange] (5) circle (0.1);   
  \end{tikzpicture}
\end{center}
\caption{A flip of a decomposition into $4$-gons (in green) at an edge and the corresponding change in the dual ribbon graph (in black).}\label{flipfig} 
\end{figure}

Starting with a flip of a decomposition into $n$-gons of the $2n-2$-gon and passing to the dual ribbon graphs, we obtain the local description of a flip of a ribbon graph at an edge $e$. The flip of an $n$-valent ribbon graph $\mathcal{T}$ at a non-loop edge $e$ is defined by locally at $e$ changing $\mathcal{T}$ as above and away from $e$ not changing $\mathcal{T}$. 

We proceed with describing flips of $n$-valent ribbon graphs in terms of contractions of ribbon graphs. It suffices to restrict to a flip at an edge, which moves the edge by one step in the counterclockwise direction (or is trivial if $n=2$). We use the following graphical notation for ribbon graphs. The edges (interior and exterior) of a ribbon graph are denoted by straight lines. The vertices are denoted by $\times$ and $\cdot$, we use $\times$ if we think of a vertex as a singularity of a not yet specified parametrized perverse schober. If an edge ends in an integer, that means that this edge represents that number of edges. 

The flip by one step is realized by two spans of contractions which are everywhere trivial, except near the edge which is being flipped, where they can be depicted as in \eqref{span1},\eqref{span2}. Note that to flip from one $n$-valent ribbon graph to another, we must pass in the zig-zag through ribbon graphs with vertices of other valencies.

\begin{adjustwidth}{-0.8in}{-0.8in}

\begin{equation}\label{span1}
\begin{tikzcd}
{} \arrow[d, no head] & n-2 \arrow[d, no head]         \\
\times \arrow[r, no head]      & \times                         \\
n-2 \arrow[u, no head]         & {} \arrow[u, no head]
\end{tikzcd}
\quad \xleftarrow{c_1}\quad
\begin{tikzcd}
                          & {} \arrow[d, no head] &                                & n-2 \arrow[d, no head]    \\
\times \arrow[r, no head] & \cdot \arrow[r, no head]       & \cdot                          & \times \arrow[l, no head] \\
n-2 \arrow[u, no head]    &                                & {} \arrow[u, no head] &                          
\end{tikzcd}
\quad \xrightarrow{c_2} \quad
\begin{tikzcd}
                          & {} \arrow[d, no head] & n-2 \arrow[d, no head]    \\
\times \arrow[r, no head] & \cdot                          & \times \arrow[l, no head] \\
n-2 \arrow[u, no head]    & {} \arrow[u, no head] &                          
\end{tikzcd}
\end{equation}

\begin{equation}\label{span2}
\begin{tikzcd}
                          & {} \arrow[d, no head] & n-2 \arrow[d, no head]    \\
\times \arrow[r, no head] & \cdot                          & \times \arrow[l, no head] \\
n-2 \arrow[u, no head]    & {} \arrow[u, no head] &                          
\end{tikzcd}
\quad \xleftarrow{c_3} \quad
\begin{tikzcd}
                          &                                & {} \arrow[d, no head] & n-2 \arrow[d, no head]    \\
\times \arrow[r, no head] & \cdot \arrow[r, no head]       & \cdot                          & \times \arrow[l, no head] \\
n-2 \arrow[u, no head]    & {} \arrow[u, no head] &                                &                          
\end{tikzcd}
\quad \xrightarrow{c_4}\quad
\begin{tikzcd}
                       & n-1 \arrow[d, no head]    \\
\times                 & \times \arrow[l, no head] \\
n-1 \arrow[u, no head] &                          
\end{tikzcd}
\end{equation}
\end{adjustwidth}

Applying \Cref{contrprop}, we can use the above contractions of ribbon graphs to produce an equivalence between $\infty$-categories of global sections of perverse schobers parametrized by the involved ribbon graphs. To this end, we describe below a collection of parametrized perverse schobers, also based on a graphical notation. Each vertex of the underlying ribbon graph is decorated with a spherical functor, which in the following will be either $\phi^*\colon \on{RMod}_R\rightarrow \on{RMod}_{R[t_{n-2}]}$ or the zero functor $0\colon 0\rightarrow \on{RMod}_{R[t_{n-2}]}$. To a vertex labeled by a spherical functor $F\colon \mathcal{V}_F\rightarrow \mathcal{N}_F$, the perverse schober assigns the $\infty$-category $\mathcal{V}^n_{F}$. Each incidence of an edge with a vertex in the underlying ribbon graph is decorated with a functor $\varrho_i$ from \eqref{rhoeq}, possibly composed with an autoequivalence of $\mathcal{N}_{F}$, which describes the functor $\mathcal{V}^n_{F}\rightarrow \mathcal{N}_{F}$ assigned by the perverse schober to the given incidence. Below, $T$ denotes the autoequivalence of $\on{RMod}_{R[t_{n-2}]}$ from \Cref{constr:Ginzburgschober}.

From now on we fix a marked surface ${\bf S}$ and an $\mathbb{E}_\infty$-ring spectrum $R$. Let $\mathcal{T}_1,\mathcal{T}_2$ be two $n$-valent spanning graphs of ${\bf S}$ which differ by a flip at any non-loop edge $e$ of $\mathcal{T}_1$ by one step in the counterclockwise direction. We find a collection of parametrized perverse schobers, related by equivalences and contractions, which are everywhere identical except at $e$ and its two incident vertices $v,v'$, where they are given as follows, starting with $\mathcal{F}_{\mathcal{T}_1}(R)$ and ending with $\mathcal{F}_{\mathcal{T}_2}(R)$. For better readability, we do not depict all edges below.

\begin{adjustwidth}{-1.2in}{-1.2in}

\begin{equation}\label{schobers0}
\begin{tikzcd}[row sep=12]
                                            & {} \arrow[d, "\varrho_2", no head]             & {} \arrow[d, "\varrho_n", no head]  \\
{} \arrow[r, "\varrho_3", no head] & \phi^* \arrow[r, "{(\varrho_1,T\circ\varrho_1)}", no head] & \phi^*                                         & {} \arrow[l, "\varrho_3", no head] \\
                                      & {} \arrow[u, "\varrho_n", no head]             & {} \arrow[u, "\varrho_2", no head] &                                            
\end{tikzcd}\simeq
\end{equation}

\begin{equation}\label{schobers1}
 \simeq
\begin{tikzcd}[row sep=12]
                                            & {} \arrow[d, "\varrho_2", no head]       & {} \arrow[d, "T^{-1}\circ \varrho_n", no head]  \\
{} \arrow[r, "\varrho_3", no head] & \phi^* \arrow[r, "{(\varrho_1,\varrho_1)}", no head] & \phi^*                                                & {} \arrow[l, "T^{-1}\circ \varrho_3", no head] \\
                                     & {} \arrow[u, "\varrho_n", no head]       & {} \arrow[u, "T^{-1}\circ \varrho_2", no head] &                                                   
\end{tikzcd}
 \xlongleftarrow{~(c_1)_*~} 
\begin{tikzcd}[row sep=12]
                                            &                                                   & {} \arrow[d, "\varrho_2", no head]                         &                                                                     & {} \arrow[d, "T^{-1}\circ \varrho_{n-1}", no head]  \\
{} \arrow[r, "\varrho_2", no head] & \phi^* \arrow[r, "{(\varrho_1,\varrho_3)}", no head] & {0} \arrow[r, "{(\varrho_1,\varrho_1)}", no head] & {0} \arrow[r, "{(\varrho_3,\varrho_1)}", no head] & \phi^*                                                     & {} \arrow[l, "T^{-1}\circ\varrho_2", no head] \\
                                  & {} \arrow[u, "\varrho_{n-1}", no head]   &                                                                     & {} \arrow[u, "T^{-1}\circ\varrho_2", no head]                   &                                                         &                                                  
\end{tikzcd}
\simeq
\end{equation}

\begin{equation}\label{schobers2} 
\simeq
\begin{tikzcd}[row sep=12]
                                            &                                                       & {} \arrow[d, "{\varrho_1[1]}", no head]                    &                                                                     & {} \arrow[d, "T^{-1}\circ \varrho_{n-1}", no head]  \\
{} \arrow[r, "\varrho_2", no head] & \phi^* \arrow[r, "{(\varrho_1,\varrho_2[1])}"', no head] & {0} \arrow[r, "{(\varrho_3,\varrho_1)}", no head] & {0} \arrow[r, "{(\varrho_3,\varrho_1)}", no head] & \phi^*                                                    & {} \arrow[l, "T^{-1}\circ\varrho_2", no head] \\
                                     & {} \arrow[u, "\varrho_{n-1}", no head]       &                                                                     & {} \arrow[u, "T^{-1}\circ\varrho_2", no head]                   &                                                        &                                                  
\end{tikzcd}
\xlongrightarrow{~(c_2)_*~}
\begin{tikzcd}[row sep=12]
                                            &                                                        & {} \arrow[d, "{\varrho_1[1]}", no head]                    & {} \arrow[d, "T^{-1}\circ \varrho_{n-1}", no head] \\
{} \arrow[r, "\varrho_2", no head] & \phi^* \arrow[r, "{(\varrho_1,\varrho_2[1])}"', no head]  & {0} \arrow[r, "{(\varrho_4,\varrho_1)}", no head] & \phi^*                                                    & {} \arrow[l, "T^{-1}\circ\varrho_2", no head] \\
      & {} \arrow[u, "\varrho_{n-1}", no head] & {} \arrow[u, "T^{-1}\circ\varrho_3"', no head]                  &                                                        &                                                   
\end{tikzcd}
\simeq 
\end{equation}

\begin{equation}\label{schobers3}
\simeq
\begin{tikzcd}[row sep=12]
                                            &                                                        & {} \arrow[d, "{\varrho_4[1]}", no head]                       & {} \arrow[d, "T^{-1}\circ \varrho_{n-1}", no head] \\
{} \arrow[r, "\varrho_2", no head] & \phi^* \arrow[r, "{(\varrho_1,\varrho_1[3])}", no head]   & {0} \arrow[r, "{(\varrho_3[2],\varrho_1)}", no head] & \phi^*                                                    & {} \arrow[l, "T^{-1}\circ \varrho_2", no head] \\
                                      & {} \arrow[u, "\varrho_{n-1}", no head] & {} \arrow[u, "{T^{-1}\circ \varrho_2[2]}"', no head]               &                                                        &                                                    
\end{tikzcd}
\xlongleftarrow{~(c_3)_*~}
\begin{tikzcd}[row sep=12]
                                            &                                                      &                                                         & {} \arrow[d, "{\varrho_3[1]}", no head]                    & {} \arrow[d, "T^{-1}\circ \varrho_{n-1}", no head]  \\
{} \arrow[r, "\varrho_2", no head] & \phi^* \arrow[r, "{(\varrho_1,\varrho_1[3])}", no head] & {0}                                  & {0} \arrow[l, "{(\varrho_3,\varrho_1)}", no head] & \phi^* \arrow[l, "{(\varrho_2[2],\varrho_1)}"', no head]   & {} \arrow[l, "T^{-1}\circ \varrho_2", no head] \\
                                    & {} \arrow[u, "\varrho_{n-1}", no head]      & {} \arrow[u, "{T^{-1}\circ \varrho_2[2]}"', no head] &                                                                     &                                                        &                                                   
\end{tikzcd}
\simeq
\end{equation}

\begin{equation}\label{schobers4}
\simeq
\begin{tikzcd}[row sep=12]
                                            &                                                      &                                                        & {} \arrow[d, "{\varrho_1[2]}", no head]                                                                        & {} \arrow[d, "T^{-1}\circ \varrho_{n-1}", no head] &  \\
{} \arrow[r, "\varrho_2", no head] & \phi^* \arrow[r, "{(\varrho_1,\varrho_3[3])}", no head] & {0}                                  & {0} \arrow[r, "{(\varrho_3[2],\varrho_1)}", no head] \arrow[l, "{(\varrho_2[1],\varrho_2)}", no head] & \phi^*                                                    & {} \arrow[l, "T^{-1}\circ \varrho_2", no head] \\
                                      & {} \arrow[u, "\varrho_{n-1}", no head]      & {} \arrow[u, "{T^{-1}\circ\varrho_1[3]}"', no head] &                                                                                                                         &                                                        &                                                   
\end{tikzcd}
\xlongrightarrow{~(c_4)_*~}
\begin{tikzcd}[row sep=12, column sep=large]
                                             & {} \arrow[d, "{\varrho_3[3]}", no head]       & {} \arrow[d, "{\varrho_1[2]}", no head]        &                                                            \\
{} \arrow[r, "{\varrho_{n}[3]}", no head] & \phi^*                                                    & \phi^* \arrow[l, "{(\varrho_2[1],\varrho_2)}"', no head]    & {} \arrow[l, "{T^{-1}\circ\varrho_{n}[2]}", no head] \\
                                                     & {} \arrow[u, "{T^{-1}\circ\varrho_1[3]}", no head] & {} \arrow[u, "{T^{-1}\circ \varrho_3[2]}", no head] &                                                      
\end{tikzcd}
\simeq 
\end{equation}

\begin{equation}\label{schobers5}
\simeq 
\begin{tikzcd}[row sep=12]
                                      & {} \arrow[d, "\varrho_3", no head]              & {} \arrow[d, "T\circ \varrho_1", no head] &                                             \\
{} \arrow[r, "\varrho_n", no head] & \phi^* \arrow[r, "{(\varrho_2,T\circ \varrho_2)}", no head] & \phi^*                                                & {} \arrow[l, "\varrho_n", no head] \\
                                            & {} \arrow[u, "T^{-1}\circ \varrho_1", no head]       & {} \arrow[u, "\varrho_3", no head]        
\end{tikzcd}
\end{equation}

\end{adjustwidth}

The above equivalences of parametrized perverse schober are each nontrivial only at one or two vertices with label $0$, where they are each given by a power of the paracyclic twist functor $T_{\mathcal{V}^i_{0}}$ with $i=3,4$, see \Cref{sec:defschober}, except for the equivalence between the parametrized perverse schober in \eqref{schobers0} and the left parametrized perverse schober in \eqref{schobers1} and the equivalence between the right parametrized perverse schober of \eqref{schobers4} and the parametrized perverse schober of \eqref{schobers5}. The former is nontrivial only at the right vertex labeled $\phi^*$, where it is given by the autoequivalence $\epsilon$ of $\mathcal{V}^n_{\phi^*}$, defined by restricting on each of the $n-1$ components $\on{RMod}_{R[t_{n-2}]}$ of the semiorthogonal decomposition to $T$ and on the component $\on{RMod}_R$ of the semiorthogonal decomposition to the identity functor. The latter equivalence of parametrized perverse schobers is nontrivial at the three objects of $\on{Exit}(\mathcal{T}_2)$ corresponding to $e$ and the two incident vertices. At the left vertex, the equivalence is given by $[3]$, at the right vertex by $\epsilon^{-1} \circ [2]$ and at $e$ by $[-2]$.

We thus obtain an equivalence 
\begin{equation}\label{eq:mueq}
 \mu_e^{1}\colon \mathcal{H}(\mathcal{T}_1,\mathcal{F}_{\mathcal{T}_1}(R))\xlongrightarrow{\simeq} \mathcal{H}(\mathcal{T}_2,\mathcal{F}_{\mathcal{T}_2}(R))\,,
\end{equation}
which we call the mutation equivalence. We denote the repeated mutation by $\mu_e^i\coloneqq (\mu_e^1)^i$ for $i\in \mathbb{Z}$.

In the remainder of this section, we give a geometric description of $\mu_e^1$ in terms of a homeomorphism 
\[ D_e\left(\frac{1}{n-1}\pi\right)\colon {\bf S}\backslash M\rightarrow {\bf S}\backslash M\]
which we now describe. 

Let $v,v'$ be the two vertices of $\mathcal{T}_1$ and $\mathcal{T}_2$ incident to $e$. Recall from \Cref{ssurfrem}, that $\Sigma_{\mathcal{T}_1}$ and $\Sigma_{\mathcal{T}_2}$ are embedded in ${\bf S}\backslash M$. The two subspaces $\Sigma_v\cup_{\Sigma_{\mathcal{T}_1}} \Sigma_{v'}$ and $\Sigma_v\cup_{\Sigma_{\mathcal{T}_2}} \Sigma_{v'}$ of ${\bf S}\backslash M$ are both clearly homeomorphic to the closed unit disc in $\mathbb{R}^2$ with $2n-2$ intervals removed from the boundary. For concreteness, we arrange the homeomorphism so that it maps $v$ to $(-\frac{1}{4},0)$ and $v'$ to $(\frac{1}{4},0)$. 

We set $D_e(\frac{1}{n-1}\pi)$ to be any homeomorphism that 
\begin{itemize}
\item restricts to a homeomorphisms between $\Sigma_v\cup_{\Sigma_{\mathcal{T}_1}} \Sigma_{v'}$ and $\Sigma_v\cup_{\Sigma_{\mathcal{T}_2}} \Sigma_{v'}$ which under the above homeomorphisms with the unit disc is an automorphism of the disc which keeps the boundary fixed and rotates the convex hull of $v$ and $v'$ by $\frac{1}{n-1}\pi$.
\item is the identity on the remainder of ${\bf S}\backslash M$.
\end{itemize}

\begin{figure}[ht]
\begin{center}
\begin{tikzpicture}

\draw[very thick, color=ao]  (1.732,1)
  arc[start angle=30,end angle=60, radius=2]; 
\draw[very thick, color=ao]  (1.732,-1)
  arc[start angle=330,end angle=300, radius=2]; 
\draw[very thick, color=ao]  (-1,1.732)
  arc[start angle=120,end angle=150, radius=2]; 
\draw[very thick, color=ao]  (-1,-1.732)
  arc[start angle=240,end angle=210, radius=2]; 
\draw[very thick, color=ao, densely dotted] (1.732,1)
  arc[start angle=30, end angle=-30, radius=2];
\draw[very thick, color=ao, densely dotted] (-1.732,-1)
  arc[start angle=210, end angle=150, radius=2];
\draw[very thick, color=ao, densely dotted]  (1,1.732)
  arc[start angle=60,end angle=120, radius=2]; 
\draw[very thick, color=ao, densely dotted]  (-1,-1.732)
  arc[start angle=240,end angle=300, radius=2]; 
  
  \draw[color=blue!255, very thick] plot [smooth] coordinates {(1.4142,1.4142) (0.5,0)};
  \draw[color=blue!255, very thick] plot [smooth] coordinates {(1.4142,-1.4142) (0.5,0)};
  \draw[color=blue!255, very thick] plot [smooth] coordinates {(-1.4142,1.4142) (-0.5,0)};
  \draw[color=blue!255, very thick] plot [smooth] coordinates {(-1.4142,-1.4142) (-0.5,0)};
  \node (0) at (0.5,0){};   
  \node (1) at (-0.5,0){};
  \node (3) at (1.4142,1.4142){};
  \node at (1.2,0.6){$\delta_3'$};
  \node (4) at (1.4142,-1.4142){};
  \node at (1.2,-0.6){$\delta_2'$};
  \node at (0,0.2){$e$};
  \node (6) at (-1.4142,1.4142){}; 
  \node at (-1.2,0.6){$\delta_2$};  
  \node (7) at (-1.4142,-1.4142){};
  \node at (-1.2,-0.67){$\delta_3$};   
  \node (14) at (0.9,0){$v'$};
  \node (15) at (-0.9,0){$v$};
  \fill (0) circle (0.1);
  \fill (1) circle (0.1);
  \draw[very thick]
  (0.5,0) -- (-0.5,0);
\end{tikzpicture} 
\hspace{5em}
\begin{tikzpicture}
\draw[very thick, color=ao]  (1.732,1)
  arc[start angle=30,end angle=60, radius=2]; 
\draw[very thick, color=ao]  (1.732,-1)
  arc[start angle=330,end angle=300, radius=2]; 
\draw[very thick, color=ao]  (-1,1.732)
  arc[start angle=120,end angle=150, radius=2]; 
\draw[very thick, color=ao]  (-1,-1.732)
  arc[start angle=240,end angle=210, radius=2]; 
\draw[very thick, color=ao, densely dotted] (1.732,1)
  arc[start angle=30, end angle=-30, radius=2];
\draw[very thick, color=ao, densely dotted] (-1.732,-1)
  arc[start angle=210, end angle=150, radius=2];
\draw[very thick, color=ao, densely dotted]  (1,1.732)
  arc[start angle=60,end angle=120, radius=2]; 
\draw[very thick, color=ao, densely dotted]  (-1,-1.732)
  arc[start angle=240,end angle=300, radius=2]; 
  
  \draw[color=blue!255, very thick] plot [smooth] coordinates {(-1.4142,1.4142) (-0.6,-0.2)  (0,-0.5)};
  \draw[color=blue!255, very thick] plot [smooth] coordinates {(1.4142,-1.4142) (0.6,0.2)  (0,0.5)};
  \draw[color=blue!255, very thick] plot [smooth] coordinates {(1.4142,1.4142) (0,0.5)};
  \draw[color=blue!255, very thick] plot [smooth] coordinates {(-1.4142,-1.4142) (0,-0.5)};
  \node (1) at (0,0.5){};   
  \node (0) at (0,-0.5){};
  \node (3) at (1.4142,1.4142){};
  \node at (0.6,1.25){$\delta_3'$};
  \node (4) at (1.4142,-1.4142){};
  \node at (1.4,-0.6){$\delta_2'$};
  \node at (0.2,0){$e$};
  \node (6) at (-1.4142,1.4142){}; 
  \node at (-0.6,-1.2){$\delta_3$};  
  \node (7) at (-1.4142,-1.4142){};
  \node at (-1.25,0.6){$\delta_2$};
  \node (14) at (0,0.9){$v'$};
  \node (15) at (0,-0.9){$v$};
  \fill (0) circle (0.1);
  \fill (1) circle (0.1);
  \draw[very thick]
  (0,0.5) -- (0,-0.5);
\end{tikzpicture}
\end{center}
\caption{The $4$-gon with two vertices $v,v'$, the edge $e$ and some matching curves (in blue) on the left and their image under $D_e(\frac{1}{2}\pi)$ on the right.}\label{Drotfig}
\end{figure}

\begin{theorem}\label{mutthm}
Let $\mathcal{T}_1$ be an $n$-valent spanning graph of ${\bf S}$ and let $\mathcal{T}_2$ be the $n$-valent spanning graph of ${\bf S}$ obtained by a flip of an edge $e$ of $\mathcal{T}_1$ which is not a loop by one step in the counterclockwise direction. Let $(\gamma,L)$ be a matching datum in ${\bf S}\backslash M$, such that $\gamma$ is pure. There exists a matching datum $(D_e(\frac{1}{n-1}\pi)\circ \gamma,L)$ and an equivalence in $\mathcal{H}(\mathcal{T}_2,\mathcal{F}_{\mathcal{T}_2}(R))$
\begin{equation}\label{steq1} 
\mu_e^{1}(M^L_{\gamma})\simeq M^L_{D_e(\frac{1}{n-1}\pi)\circ \gamma}[m]\,,
\end{equation}
where 
\begin{itemize}
\item $m=1$ if the first or last segment of $\gamma$ is of the first type and lies at $v$, exiting $v$ through the edge $e$. 
\item $m=0$ if $\gamma$ is not as above. This includes all cases in which $\gamma$ is regular, i.e.~all its endpoints lie in $\partial {\bf S}\backslash M$.
\end{itemize}
\end{theorem}

\begin{remark}
Note that the matching curve $D_e(\frac{1}{n-1}\pi)\circ \gamma$ in \Cref{mutthm} is not necessarily pure. We expect that \Cref{mutthm} can be extended to the global sections associated to arbitrary matching data.
\end{remark}

\begin{proof}[Proof of \Cref{mutthm}.]
We note that it suffices to show that 
\begin{equation}\label{steq}
\mu_e^{1}(M^L_{\gamma})\simeq M^L_{D_e(\frac{1}{n-1}\pi)\circ \gamma}[m]
\end{equation}
 for each pure matching curve $\gamma$ in the $(2n-2)$-gon. The theorem then follows, using that $\mu_e^1$ and the object $M^L_{\gamma}$ associated to the matching curve $\gamma$ in ${\bf S}$ are defined via gluing. This leaves finitely many cases, which are directly verified by tracing through the equivalences between the global sections of the parametrized perverse schobers defining $\mu_e^1$. 

To help the reader appreciate the appearance or absence of suspensions in the theorem, without having to trace through the definition of $\mu_e^1$, we offer the following hints.

To see the absence of suspensions for pure, regular segments, one simply observes that the values of the corresponding sections at the external edges of the $(2n-2)$-gon remain unchanged under $\mu_e^1$ and in particular do not acquire any suspensions. 

Denote the vertices of $\mathcal{T}_1$ incident to $e$ by $v,v'$. Let $\gamma$ be a pure matching curve starting at $v$ (for $v'$ the argument is analogous). One has for $Q\in \on{RMod}_R$ an equivalence
\[ M_{\gamma}^{L}(v)\simeq \Big( Q\xrightarrow{!}\phi^*(Q)\xrightarrow{\on{id}}\dots\xrightarrow{\on{id}}\phi^*(Q)\Big)\in \mathcal{V}^n_{\phi^*}\] (where $M_{\gamma}^L$ is the section of the schober \eqref{schobers0}) and to verify the appearance of the suspension, one computes 
\[ \mu_e^1(M_{\gamma}^L)(v)\simeq \Big( Q\xrightarrow{!}\phi^*(Q)\xrightarrow{\on{id}}\dots\xrightarrow{\on{id}}\phi^*(Q)\rightarrow 0\Big)\in \mathcal{V}^n_{\phi^*}\,,\]
(where $\mu_e^1(M_{\gamma}^L)$ is a section of the schober \eqref{schobers5}), which is the suspension of \eqref{mdieq0} (for $i=2$). 
\end{proof}

\begin{remark}
Descriptions of derived equivalences in terms of rotations of a disc by fractions of $\pi$ also appear in \cite[Prop. 3.5.1]{DJL21} and for gentle algebras associated to unpunctured $n$-gons in \cite[Thm. 5.1]{OPS18}. A similar geometric description for $\mu_{e}^{n-1}$ in the case $n=3$ also appears in \cite{Qiu16}. The automorphism $\mu_e^{n-1}$ acts objectwise as the cotwist functor of the $n$-spherical object $M_{\gamma_e}$. 
\end{remark}

\section{Morphisms from intersections}\label{sec:homs}

For the entirety of \Cref{sec:homs}, we fix an $\mathbb{E}_\infty$-ring spectrum $R$ and a marked surface ${\bf S}$ with an $n$-valent spanning graph $\mathcal{T}$. We will only consider the case $n\geq 3$, except in \Cref{subsec:n=2}, where we describe some results which hold in the case $n=2$. In this section, we describe the morphism objects between the global sections of $\mathscr{F}_\mathcal{T}(R)$ arising from matching data with pure underlying matching curves. We begin by describing the different types of intersections between matching curves.

\begin{definition}\label{intnumdef}
Let $\gamma,\gamma'$ be two matching curves in ${\bf S}$. We choose representatives of $\gamma$ and $\gamma'$ with the minimal number of intersections.  
\begin{itemize}
\item We define the set of singular intersections $i^{\on{sg}}(\gamma,\gamma')$ as the set of intersections of $\gamma$ and $\gamma'$ at their endpoints in $\mathcal{T}_0$. If $\gamma=\gamma'$ with distinct endpoints, we define $i^{\on{sg}}(\gamma,\gamma)$ as the set of endpoints of $\gamma$ in $\mathcal{T}_0$. If $\gamma=\gamma'$ with two identical endpoints, the set $i^{\on{sg}}(\gamma,\gamma)$ has four elements. We also abusively write $i^{\on{sg}}(\gamma,\gamma')$ for the number of singular intersections, i.e.~the cardinality of $i^{\on{sg}}(\gamma,\gamma')$.
\item We define the number of directed boundary intersections $i^{\on{bdry}}(\gamma,\gamma')$ as the number of intersections of $\gamma$ and $\gamma'$ with the same connected component of $\partial {\bf S}\backslash M$ such that the intersection of $\gamma$ and $\partial {\bf S}\backslash M$ precedes the intersection of $\gamma'$ and $\partial{\bf S}\backslash M$ in the orientation of $\partial {\bf S}\backslash M$ induced by the clockwise orientation of ${\bf S}$. If $\gamma=\gamma'$, we only consider directed boundary intersections of distinct endpoints.
\item We denote by $i^{\on{cr}}(\gamma,\gamma')$ the number of crossings from $\gamma$ to $\gamma'$, see \Cref{def:cross}. If $\gamma=\gamma'$, then $i^{\on{cr}}(\gamma,\gamma)$ counts each self-crossing only once. 
\end{itemize}
\end{definition}

\begin{definition}\label{def:cross}
Consider an intersection $x$ of two matching curves $\gamma,\gamma'$ in ${\bf S}\backslash M$ away from their endpoints. The intersection $x$ can be chosen to lie in a small neighborhood of an edge $e$ of $\mathcal{T}$. We say that $x$ is a crossing from $\gamma$ to $\gamma'$ if in this neighborhood, the curves are arranged as follows.
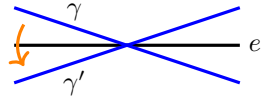
\[
\begin{tikzpicture}
\draw[very thick] (0,0)--(3,0);
\draw[color=blue, very thick] (0,0.5)--(3,-0.5) (0,-0.5)--(3,0.5);
\node (1) at (0.2,0.4){};
\node (2) at (0.2,-0.4){} edge [<-, bend left, color=orange, very thick] (1);
\node () at (0.8,0.45){$\gamma$};
\node () at (0.8,-0.5){$\gamma'$};
\node () at (3.2,0){$e$};
\end{tikzpicture}
\]
The orange arrow goes in the counterclockwise direction. If the curves are arranged in the opposite way, we say that the intersection is a crossing from $\gamma'$ to $\gamma$.
\end{definition}

\begin{notation}~
\begin{itemize}
\item Consider a singular intersection $x$ of $\gamma$ and $\gamma'$ at a vertex $v$, with segments $\delta=\delta^i$ and $\delta'=\delta^j$ at $v$ and $1\leq i,j\leq n$ in the notation of part {\bf 1)} of \Cref{segcstr}. We denote by $\on{deg}_x=i-j$, if $i<j$, and $\on{deg}_x=i-j-n$, if $j<i$, the number of steps after which the segment $\delta'$ follows the segment $\delta$ at $v$ in the counterclockwise direction. If $i=j$, then both $\delta$ and $\delta'$ exit $\Sigma_v$ through the same boundary component $f\subset \partial \Sigma_v$. Choosing $\gamma,\gamma'$ with the minimal number of intersections, we set $\on{deg}_x=0$ if the intersection of $\delta$ with $f$ precedes the intersection of $\delta'$ with $f$; otherwise we set $\on{deg}_x=-n$.
\item Given $L,L'\in \on{RMod}_{R[t_{n-2}]}$, we denote by $\on{Mor}(L,L')$ the morphism object. 
\item Given $Q,Q'\in \on{RMod}_R$, we denote by $\on{Mor}_R(Q,Q')=\on{Mor}_{\on{RMod}_R}(Q,Q')$ the morphism object.
\end{itemize}
\end{notation}

\begin{theorem}\label{homthm}
Let $(\gamma,L),(\gamma',L')$ be two matching data in ${\bf S}\backslash M$, such that $\gamma\neq \gamma'$, $\gamma,\gamma'$ are pure and have no common infinite ends, see \Cref{infrem}. Let $a$ be the rank of $(\gamma,L)$ and $a'$ the rank of $(\gamma',L')$. Consider the associated global sections $M^L_{\gamma},M^{L'}_{\gamma'}\in \mathcal{C}\coloneqq\mathcal{H}(\mathcal{T},\mathcal{F}_\mathcal{T}(R))$. 

The morphism object $\on{Mor}_{\mathcal{C}}(M^L_{\gamma},M^{L'}_{\gamma'})\in \on{RMod}_R$ is equivalent to 
\begin{align*} \on{Mor}(L,L')^{\oplus i^{\on{bdry}}(\gamma,\gamma')+ aa'i^{\on{cr}}(\gamma,\gamma')} & \oplus \on{Mor}(L,L')[-1]^{\oplus aa' i^{\on{cr}}(\gamma',\gamma)} \\
& \oplus \bigoplus_{x\in i^{\on{sg}}(\gamma,\gamma')}\on{Mor}_R(Q,Q')[\on{deg}_x]\,.\end{align*}
\end{theorem}

\begin{theorem}\label{homthm2}
Let $(\gamma,L)$ and $(\gamma,L')$ be two matching data in ${\bf S}\backslash M$, whose underlying matching curves are identical. Suppose that $\gamma$ is pure and finite.
\begin{enumerate}[i)]
\item If $\gamma$ is open and regular, the morphism object $\on{Mor}_{\mathcal{C}}(M_{\gamma}^L,M_{\gamma}^{L'}) \in\linebreak \on{RMod}_R$ is equivalent to 
\[ \on{Mor}(L,L')^{\oplus 1+i^{\on{cr}}(\gamma,\gamma)+ i^{\on{bdry}}(\gamma,\gamma)}\oplus \on{Mor}(L,L')[-1]^{\oplus i^{\on{cr}}(\gamma,\gamma)}\]
\item If $\gamma$ is open and singular, the morphism object $\on{Mor}_{\mathcal{C}}(M_{\gamma}^L,M_{\gamma}^{L'})\in \on{RMod}_R$ is equivalent to 
\[ \left(\on{Mor}(L,L')\oplus \on{Mor}(L,L')[-1]\right)^{\oplus i^{\on{cr}}(\gamma,\gamma)}\oplus \bigoplus_{x\in i^{\on{sg}}(\gamma,\gamma')}\on{Mor}_R(Q,Q')[\on{deg}_x]\,. 
\] 
\item Let $R=k$ be a field and assume that $\gamma$ is a closed matching curve (thus automatically regular) and that $\on{Map}(L,L)\simeq k$. Assume further that the monodromy equivalence of $(\gamma,L)$ is given by a single $a\times a$-Jordan block. The morphism object $\on{Mor}_{\mathcal{C}}(M^L_{\gamma},M^{L}_{\gamma})\in \on{RMod}_k$ is equivalent to  
\[
\left(\on{Mor}(L,L)\oplus \on{Mor}(L,L)[-1]\right)^{\oplus a+a^2i^{\on{cr}}(\gamma,\gamma)}\,.
\]
\end{enumerate}
\end{theorem}

The next example considers special cases of \Cref{homthm,homthm2} and point out the analogies with symplectic matching constructions. 

\begin{example}\label{homex}
Let $(\gamma,\phi^*(R))$ be a matching datum in ${\bf S}\backslash M$, such that $\gamma$ is finite, pure and has no self-intersections. 
\begin{enumerate}
\item If both ends of $\gamma$ lie at vertices of $\mathcal{T}_0$, we have 
\[\on{Mor}_{\mathcal{C}}(M_{\gamma},M_{\gamma})\simeq R\oplus R[-n]\,,\]
meaning that $M_{\gamma}$ is an $n$-spherical object. In this case, $\gamma$ is the analog of a matching path in the sense of \cite[\S (16g)]{Sei08} and $M_{\gamma}$ and its construction are the algebraic analogues of the corresponding matching sphere and its construction.
\item If $\gamma$ begins at a vertex of $\mathcal{T}$ and ends on the boundary of ${\bf S}$, then $M_{\gamma}$ is an exceptional object, i.e.~$\on{Mor}_{\mathcal{C}}(M_{\gamma},M_{\gamma})\simeq R$. The object $M_{\gamma}$ is the algebraic analog of a Lefschetz thimble as in \cite[\S (16b)]{Sei08}. 
\item If $\gamma$ begins and ends on the boundary of ${\bf S}$, then 
\[ \on{Mor}_{\mathcal{C}}(M_{\gamma},M_{\gamma})\simeq R\oplus R[1-n]\,,\]
meaning that $M_{\gamma}$ is an $(n-1)$-spherical object. The object $M_{\gamma}$ is the algebraic analog of the parallel transport of a Lagrangian sphere $S^{n-1}$ along $\gamma$, yielding a Lagrangian homeomorphic to $[0,1]\times S^{n-1}$. Symplectically, to get a similar morphism object with partial wrapping, one would have to consider the object not in the Fukaya--Seidel category, but a partially wrapped Fukaya categories with more stops. 
\item If $(\gamma,\phi^*(R))$ is closed and of rank $1$ and $R=k$ a field, then
\[ \on{Mor}_{\mathcal{C}}(M_{\gamma},M_{\gamma})\simeq k\oplus k[-1]\oplus k[1-n]\oplus k[-n]\,.\] 
The object $M_{\gamma}$ is the algebraic analog of a Lagrangian torus $S^{n-1}\times S^1$ arising from parallel transport along the closed curve $\gamma$.
\end{enumerate}
\end{example}

\begin{remark}\label{rem:Lagrangian_intersections_crossings}
The results of \Cref{homthm,homthm2} may be considered as algebraic analogs of Lagrangian Floer homology computations, describing the morphisms in partially wrapped Fukaya categories of symplectic manifolds equipped with a Lefschetz fibration to the surface. Directed boundary intersections correspond to morphisms arising from partial wrapping at the boundary, which causes their directedness. 

Crossing give rise to Lagrangian intersections which cannot be disjoined via Hamiltonian isotopy or wrapping. Thus, the corresponding morphisms go in both directions. Due to the local nature of Lagrangians arising from matching construction (locally a Lagrangian in the fiber times an interval), the Lagrangian Floer homology can be computed in the fiber at the crossing. The algebraic analog of the Fukaya category of the fiber is the generic stalk of the perverse schober. 

Finally, singular intersections should correspond to single transverse intersections of the corresponding Lagrangians, see \cite[Fig.~20.1]{Sei08} for an illustration.
\end{remark}

\begin{remark}\label{rem:moregen}
Let $(\gamma,L)$ and $(\gamma',L')$ be two matching data, such that $\gamma$ and $\gamma'$ are not necessarily pure. The method of proof of \Cref{homthm} also applies to compute the morphism object $\on{Mor}_\mathcal{C}(M_{\gamma}^L,M_{\gamma'}^{L'})$. In this more general setting, there are however exceptions to the simple rule that $\on{Mor}_\mathcal{C}(M_{\gamma}^L,M_{\gamma'}^{L'})$ counts intersections, unless $L,L'\in \on{Im}(\phi^*)$. We discuss this in \Cref{sec:localHom} and \Cref{ex:counter}. In the cases that there are segments of the second type wrapping around a vertex by $n$ steps, the proofs of \Cref{homthm,homthm2} do not directly apply and would need some minor adaptions. Giving a systematic description of the morphism objects in the non-pure setting would also require introducing gradings of the objects and the surface, as for example done in \cite{IQZ20}. For many applications, it suffices to consider pure matching curves. 
\end{remark}

\begin{remark}\label{infrem}
Consider two distinct matching curves $\gamma\colon U\rightarrow \Sigma$ and $\gamma'\colon U'\rightarrow \Sigma$. We say that $\gamma$ and $\gamma'$ have a common infinite end if there exist immersions $I_1\colon \mathbb{R}_{\geq 0}\rightarrow U$, $I_2\colon \mathbb{R}_{\geq 0}\rightarrow U'$ such that the curves $\gamma|_{I_1}$ and $\gamma'|_{I_2}$ are composed of infinitely many identical segments. Note that this definition allows that one of the two curves $\gamma$ or $\gamma'$ is closed. If in \Cref{homthm}, $\gamma$ and $\gamma'$ are open with a common infinite end, then the $R$-module $\on{Mor}_{\mathcal{C}}(M^L_{\gamma},M^{L'}_{\gamma'})$ consists of infinitely many copies of $\on{Mor}(L,L')$. 
\end{remark}

The proofs of \Cref{homthm,homthm2} consist of gluing arguments. We decompose $\gamma$ and $\gamma'$ into segments and begin in \Cref{sec:localHom} by describing all morphisms between the associated local sections. In \Cref{sec:localHom}, we also allow non-pure segments. In \Cref{sec:globalHom}, we then describe $\on{Mor}_{\mathcal{C}}(M^L_{\gamma},M^{L'}_{\gamma'})$ via the colimit of a diagram of morphism objects between the local sections associated to the segments. In \Cref{sec:Homproof} we combine the findings of \Cref{sec:localHom} and \Cref{sec:globalHom} to prove \Cref{homthm,homthm2}.

\subsection{Intersections locally}\label{sec:localHom}

In this section, we exclude all segments of the second type wrapping around a vertex by $n$ steps and all curves composed of segments in which such a segment appear.

Let $\mathcal{L}$ denote the $R$-linear $\infty$-category of lax sections of $\mathcal{F}_\mathcal{T}(R)$, see \Cref{glsecdef}. In the following, we describe the morphism objects 
\[ \on{Mor}_{\mathcal{L}}(M_\delta^L,M_{\eta}^{L'}),~\on{Mor}_{\mathcal{L}}(Z_e^L,M_{\eta}^{L'})\,,\] 
where $L,L'\in \on{RMod}_{R[t_{n-2}]}$, $\delta$ is a segment in $\Sigma_\mathcal{T}$, $e$ is an edge of $\mathcal{T}$ and $\eta$ is any matching curve in ${\bf S}\backslash M$ or an open curve composed of segments, see also \Cref{constr:open1} for the notation. If $\delta$ (or $\eta$) is singular, we require as always that $L\simeq \phi^*(Q)$ (or $L'\simeq \phi^*(Q')$).

We begin by determining the morphism objects between sections $M_{\delta}^L$ and $M_{\delta'}^{L'}$ associated to segments $\delta,\delta'$.  

If the segments $\delta$ and $\delta'$ are not located at the same vertex of $\mathcal{T}$, one finds $\on{Mor}_{\mathcal{L}}(M^L_{\delta},M^{L'}_{\delta'})\simeq 0$, see also \Cref{lilem1}. We thus assume that $\delta$, $\delta'$ are located at the same vertex $v\in \mathcal{T}_0$. We choose representatives of $\delta$ and $\delta'$ with the minimal number of intersections. With the exception of some cases described further below, we find that $\on{Mor}_{\mathcal{L}}(M^{L}_{\delta},M^{L'}_{\delta'})$ is the direct sum of $R$-modules given as follows.

\begin{itemize}

\item Each directed boundary intersection from $\delta$ to $\delta'$ in $\Sigma_v$ contributes a copy of $\on{Mor}(L,L')$, up to suspensions, to $\on{Mor}_{\mathcal{L}}(M^L_{\delta},M^{L'}_{\delta'})$. The corresponding morphisms have support at $v,e_i$, where $e_i$ is the edge of $\mathcal{T}$ intersecting the same component of $\partial \Sigma_v$ as $\delta,\delta'$.  If $\delta,\delta'$ are pure, no suspensions appear. 
\item If $L\simeq\phi^*(Q)$ and $L'\simeq \phi^*(Q')$, then each singular intersection of $\delta=\delta^i$ and $\delta'=\delta^j$ contributes a single copy of $\on{Mor}_R(Q,Q')[b]$ to $\on{Mor}_{\mathcal{L}}(M^L_{\delta},M^{L'}_{\delta'})$, where $b=i-j$ if $j>i$, $b=i-j-n$ if $j<i$, and $b=0$ for $i=j$. The support of the corresponding morphisms is given by $v$ if $\delta\neq \delta'$ and by $v,e_i$ if $\delta=\delta'=\delta^i$. 
\item Each crossing of $\delta$ and $\delta'$ contributes a copy of $\on{Mor}(L,L')$, up to suspensions, to $\on{Mor}_{\mathcal{L}}(M^L_{\delta},M^{L'}_{\delta'})$, corresponding to morphisms with support at $v$. There are no crossings if both $\delta$ and $\delta'$ are pure.
\item If $\delta=\delta'=\delta^{i,j}$ is of the second type with $i\neq j$, then $\on{Mor}_{\mathcal{L}}(M_{\delta}^L,M_{\delta'}^L)\simeq \on{Mor}(L,L')$. 
\end{itemize}

There are three possible exceptions to the above description, which appear if $L$ and $L'$ do not lie in the image of $\phi^*$ and at least one of the segments $\delta,\delta'$ is not pure. In these exceptional cases, we have that $\delta$ and $\delta'$ are of the second type, $\delta\neq \delta'$ and the two segments have either two crossings, two boundary intersections or a crossing and a boundary intersection. Further below, we describe the outcome in these cases in more detail.

The above descriptions of $\on{Mor}_{\mathcal{L}}(M_\delta^L,M_{\delta'}^{L'})$ follow from the universal properties of the involved Kan extensions and some basic computations, as we now explain. We denote the Grothendieck construction of $\mathcal{F}_\mathcal{T}(R)$ by $p\colon \Gamma(\mathcal{F}_\mathcal{T}(R))\rightarrow \on{Exit}(\mathcal{T})$. The sections $M^L_{\delta}$ and $M^{L'}_{\delta'}$ were defined as the $p$-relative left Kan extensions of their restrictions to $\{v\}\subset \on{Exit}(\mathcal{T})$. Using the universal property of Kan extensions and arguing as in the proof of \Cref{lilem1} below, we find that the restriction functor induces an equivalence of $R$-modules 
\begin{align*} \on{Mor}_{\mathcal{L}}(M^L_{\delta},M^{L'}_{\delta'})& \simeq \on{Mor}_{\mathcal{V}^n_{\phi^*}}(M^L_{\delta}(v),M^{L'}_{\delta'}(v))\,,\\
\on{Mor}_{\mathcal{L}}(M^L_{\delta'},M^{L'}_{\delta})&\simeq \on{Mor}_{\mathcal{V}^n_{\phi^*}}(M^L_{\delta'}(v),M^{L'}_{\delta}(v)) \,.
\end{align*}

The resulting morphism objects in $\mathcal{V}^n_{\phi^*}$ can be directly determined by a case by case analysis, using again universal properties of Kan extensions and making use of the paracyclic twist $T_{\mathcal{V}^n_{\phi^*}}$ from \Cref{cyctwrem}. We collect the resulting morphisms objects for all possible pairs $\delta,\delta'$ in \Cref{locmorfig1,locmorfig2,locmorfig3}. For the description of $\delta,\delta'$, we use the notation of \Cref{segcstr}. In \Cref{locmorfig1}, we consider the cases $\delta=\delta^1$ and $\delta'=\delta^i$ with $1\leq i \leq n$ and $L\simeq \phi^*(Q)$, $L'\simeq \phi^*(Q')$. In \Cref{locmorfig2}, we consider the cases $\delta=\delta^i$ and $\delta'=\delta^{j,n}$ with $1\leq i,j\leq n$ and $j\neq n$ and $L\simeq \phi^*(Q)$. In \Cref{locmorfig3}, we consider (a subset of) the cases $\delta=\delta^{1,j}$ and $\delta'=\delta^{i',j'}$ with $1\leq i',j,j'\leq n$ and $1\neq j,i'\neq j'$. Any unordered pair $\delta,\delta'$ is described by one of the pairs considered below, up to the rotational symmetry of $\Sigma_v$, on the categorical level realized by the action of the paracyclic twist functor $T_{\mathcal{V}^n_{\phi^*}}$ on $\mathcal{F}_\mathcal{T}(R)(v)$. 

\begin{table}
\begin{center}
\begin{tabular}{c|c|c|c|c} 
\centering Intersections & $\delta^{1},\delta^{i}$ & \centering $\on{Mor}_{\mathcal{L}}(M_{\delta^{1}}^L,M_{\delta^{i}}^{L'})$ & \centering $\on{Mor}_{\mathcal{L}}(M_{\delta^{i}}^{L'},M_{\delta^{1}}^{L})$ & Support\\ 
\hline
\centering 1x singular & $i>1$ &\centering $\on{Mor}_R(Q,Q')[1-i]$  & \centering $\on{Mor}_R(Q',Q)[i-1-n]$ & at $v$\\
\centering 1x singular & $i=1$ &\centering $\on{Mor}_R(Q,Q')$  & \centering $\on{Mor}_R(Q',Q)$ & at $v,e_1$\\
\end{tabular}
\caption{All possible pairs of two segments $\delta^1,\delta^{i}$ of the first type up to rotational symmetry.}\label{locmorfig1}
\end{center}

\begin{adjustwidth}{-0.2in}{-0.2in}
\begin{center}
\begin{tabular}{c|c|c|c|c} 
\centering Intersections & $\delta^{i},\delta^{j,n}$ & \centering $\on{Mor}_{\mathcal{L}}(M^L_{\delta^{i}},M^{L'}_{\delta^{j,n}})$ & \centering $\on{Mor}_{\mathcal{L}}(M^{L'}_{\delta^{j,n}},M^{L}_{\delta^{i}})$ & Support\\ 
\hline
\centering none & $i<j<n$ &\centering 0 & \centering 0 & / \\
\centering 1x crossing & $j<i<n$ &\centering $\on{Mor}(L,L')[i-j-1]$  & \centering $\on{Mor}(L',L)[j-i]$ & at $v$\\
\centering 1x boundary & $j=i<n$ &\centering 0 &\centering $\on{Mor}(L',L)[j-i]$ & at $v,e_i$ \\
\centering 1x boundary & $j<i=n$ &\centering $\on{Mor}(L,L')[n-j-1]$ &\centering $0$ & at $v,e_n$ \\
\end{tabular}
\caption{All possible pairs of one segment $\delta^{i}$ of the first type and one segments $\delta^{j,n}$ of the second type up to rotational symmetry.}\label{locmorfig2}
\end{center}
\end{adjustwidth}

\begin{adjustwidth}{-0.4in}{-0.4in}
\begin{center}\small
\begin{tabular}{p{19mm}|c|p{34mm}|p{34mm}|c} 
\centering Intersections & $\delta^{1,j},\delta^{i',j'}$ & \centering $\on{Mor}_{\mathcal{L}}(M^L_{\delta^{1,j}},M^{L'}_{\delta^{i',j'}})$ & \centering $\on{Mor}_{\mathcal{L}}(M^{L'}_{\delta^{i',j'}},M^{L}_{\delta^{1,j}})$ & Support\\ 
\hline
\centering none & $1<j<i'<j'$ &\centering 0 & \centering 0 & / \\
\centering none & $1<i'<j'<j$ &\centering 0 &\centering 0 & / \\
\centering 1x crossing & $1<i'<j<j'$ &\centering $\on{Mor}(L,L')[1-i']$ &\centering $\on{Mor}(L',L)[i'-2]$ & at $v$\\
\centering 1x boundary & $1<i'<j=j'$& \centering $\on{Mor}(L,L')[1-i']$ &  \centering $0$ & at $v,e_j$\\
\centering 1x boundary & $1<i'=j<j'$ & \centering $0$ & \centering $\on{Mor}(L',L)[i'-2]$ & at $v,e_j$\\
\centering 1x boundary & $1=i'< j < j'$ &\centering $\on{Mor}(L,L')$ & \centering $0$ &  at $v,e_1$ \\
\centering 2x boundary & $1=i'<j=j'$ &\centering $\on{Mor}(L,L')$ &\centering $\on{Mor}(L',L)$ & at $v,e_1,e_j$\\
\hline
\centering 2x crossing & $1<j'< i'<j$ & \centering $ \on{Mor}(L,\phi^*\phi_*(L'))[n-i']$ &\centering $\on{Mor}(L',\phi^*\phi_*(L))[i'-2]$ & at $v$ \\
\centering 2x boundary & $1=j'< j = i'$ &\centering $\on{Mor}(L,L')[n-i']$ & \centering $\on{Mor}(L',L)[i'-2]$ &  at $v,e_1,e_j$ \\
\centering 1x crossing $+$ 1x boundary & $1< j'< j=i'$ & \centering $\on{Mor}(L,L')[n-i']$ & \centering $\on{Mor}(L',\phi^*\phi_*(L))[i'-2]$ & at $v,e_j$ \\
\end{tabular}
\caption{All possible pairs of two segments $\delta^{1,j}$ and $\delta^{i',j'
}$ of the second type up to a swap and rotational symmetry.
}\label{locmorfig3}
\end{center}
\end{adjustwidth}
\end{table}

In \Cref{locmorfig3}, the three cases where the simple description of the morphisms objects in terms of intersections can fail are separated from the other cases. In the case of $\delta\neq\delta'$ with two boundary intersections, the morphism objects match the number of intersections, but the support of the morphisms does not behave as expected (unless $L,L'\in \on{Im}(\phi^*)$). See also \Cref{ex:counter} for the consequences of this phenomenon. To summarize our computations, the exceptions to the simple rule "\#{}intersections=$\on{dim}$ of Homs" arise if $\delta$ and $\delta'$ cut out a vertex of $\mathcal{T}_0$. 

\begin{lemma}\label{lilem1}
Let $\eta$ be a curve in ${\bf S}\backslash M$ composed of segments $\{\delta_i\}_{i\in I}$. Given a subset $\tilde{I}\subset I$, we denote $\tilde{I}^{\geq 1}=\tilde{I}\cap \mathbb{N}$ if $I\neq \mathbb{Z}/N\mathbb{Z}$ and $\tilde{I}^{\geq 1}=\tilde{I}$ if $I=\mathbb{Z}/N\mathbb{Z}$. We denote $\tilde{I}^{\leq 0}=\tilde{I}\backslash \tilde{I}^{\geq 1}$.
\begin{enumerate}[(1)]
\item Let $\delta$ be a segment lying at a vertex $v\in \mathcal{T}_0$ and let $I_v\subset I$ be the set of segments of $\eta$ lying at $v$. Then there exists an equivalence of $R$-modules 
\begin{align*} \on{Mor}_{\mathcal{L}}(M^L_{\delta},M^{L'}_{\eta})\simeq & \bigoplus_{\delta'\in I_v^{\geq 1}}\on{Mor}_{\mathcal{L}}(M^L_{\delta},M^{L'}_{\delta'}[d(\delta^1\leq \eta<\delta')])\oplus\\ & \bigoplus_{\delta'\in I_v^{\leq 0}}\on{Mor}_{\mathcal{L}}(M^L_{\delta},M^{L'}_{\delta'}[d(\delta' \leq \eta < \delta^1)])\,.
\end{align*}
\item Let $e$ be an edge of $\mathcal{T}$. Then there exists an equivalence of $R$-modules 
\[ \on{Mor}_{\mathcal{L}}(Z^L_e,M^{L'}_{\eta})\simeq \on{Mor}(L,M^{L'}_{\eta}(e))\,.\]
Denote 
\begin{align*} N= & \bigoplus_{\delta'\in I_e^{\geq 1}} \on{Mor}(L,L')[d(\delta^1\leq \eta<\delta')]\\
& \oplus \bigoplus_{\delta'\in I_e^{\leq 0}} \on{Mor}(L,L')[d(\delta'\leq \eta <\delta^1)]\,,\end{align*}
where $I_e\subset I$ is the set of segments of $\eta$ which begin at $e$.

If $\eta$ does not end at $e$, there exists an equivalence of $R$-modules 
\begin{equation}\label{limoreq} 
\on{Mor}_{\mathcal{L}}(Z^L_e,M^{L'}_{\eta})\simeq N
\end{equation}
If $\eta$ ends at $e$, then there exists an equivalence of $R$-modules
\begin{equation}\label{limoreq1} 
\on{Mor}_{\mathcal{L}}(Z^L_e,M^{L'}_{\eta})\simeq N \oplus \on{Mor}(L,L')[d(\delta^1\leq \eta)]\,.
\end{equation}
\end{enumerate}
\end{lemma}

\begin{proof}
Note that $M^L_{\delta}$ and $Z^{L}_{e}$ are given by left Kan extensions relative the Grothendieck construction of $\mathcal{F}_\mathcal{T}(R)$ of their restrictions to $v$, respectively, $e$. Using the universal property of Kan extensions, see \cite[4.3.2.17]{HTT}, it follows that for any section $X\in \mathcal{L}$ the restriction morphisms of $R$-modules 
\begin{align*} \on{Mor}_{\mathcal{L}}(M^L_\delta,X)&\longrightarrow \on{Mor}_{\mathcal{V}^n_{\phi^*}}(M^L_{\delta}(v),X(v))\\
\on{Mor}_{\mathcal{L}}(Z^L_e,X)&\longrightarrow \on{Mor}(L,X(e))
\end{align*}
restrict to equivalences on all homotopy groups so that they are equivalences of $R$-modules. 

By construction of $M^{L'}_{\eta}$, we find an equivalence in $\mathcal{V}^n_{\phi^*}$
\[ M^{L'}_{\eta}(v)\simeq \bigoplus_{\delta'\in I_v^{\geq 1}}M^{L'}_{\delta'}(v)[d(\delta^1\leq \eta<\delta')] \oplus \bigoplus_{\delta'\in I_v^{\leq 0}}M^{L'}_{\delta'}(v)[d(\delta'\leq  \eta< \delta^1)]\]
showing statement (1). Similarly, there exists an equivalence in $\on{RMod}_{R[t_{n-2}]}$
\[ M_{\eta}^{L'}(e)\simeq \bigoplus_{\delta'\in I_e^{\geq 1}}L'[d(\delta^1\leq \eta<\delta')]\oplus \bigoplus_{\delta'\in I_e^{\leq 0}} L'[d(\delta'\leq \eta <\delta^1)]\]
if $\eta$ does not end at $e$ and 
\[ M^{L'}_{\eta}(e)\simeq \bigoplus_{\delta'\in I_e^{\geq 1}}L'[d(\delta^1\leq \eta <\delta')]\oplus \bigoplus_{\delta'\in I_e^{\leq 0}} L'[d(\delta'\leq \eta <\delta^1)]\oplus L'[d(\delta^1\leq \eta)]\]
if $\eta$ ends at $e$, see also \Cref{degrem2} for the shifts. This shows statement (2).
\end{proof}

\begin{remark}\label{supprem}
The support of the morphisms objects in \Cref{locmorfig1,locmorfig2,locmorfig3} refers to the support of the corresponding morphisms between sections in the sense of \Cref{def:supportofmorph}. It has the following further interpretation: let $\delta,\delta'$ be two segment in ${\bf S}$ both lying at a vertex $v\in \mathcal{T}_0$. Suppose that $\delta$ starts at an edge $e$ and ends at an edge $f$. Let $\on{Mor}(L,L')[l]$, with $l\in \mathbb{Z}$, be a summand of $\on{Mor}_{\mathcal{L}}(M^L_\delta,M^{L'}_{\delta'})$ identified in \Cref{locmorfig1,locmorfig2,locmorfig3}, corresponding to morphisms with support at $v$ and the edges $J\subset \{e,f\}$ ($J=\emptyset$ is possible). The composite $c$ of the inclusion morphism of $R$-modules 
\[ \on{Mor}(L,L')[l]\longrightarrow \on{Mor}_{\mathcal{L}}(M^L_\delta,M^{L'}_{\delta'})\]
with the morphism 
\[  \on{Mor}_{\mathcal{L}}(M^L_\delta,M^{L'}_{\delta'})\longrightarrow \on{Mor}_{\mathcal{L}}(Z^L_{e},M^{L'}_{\delta'})\] 
obtained from precomposing with the pointwise inclusion of the section $Z^L_{e}$ into  $M^L_{\delta}$ can be described as follows.
\begin{itemize}
\item If $e\notin J$, then $c$ is zero.
\item If $e\in J$, then $c$ is the inclusion of a direct summand under the equivalence \eqref{limoreq} or \eqref{limoreq1}.
\end{itemize}
An analogous description holds for the morphism $c$ arising by replacing the edge $e$ with $f$.

The three cases at the end of \Cref{locmorfig3} are exceptions to the above descriptions (unless $L,L'\in \on{Im}(\phi^*)$). 
\end{remark}

\subsection{Intersections globally}\label{sec:globalHom}

We fix two matching data $(\gamma,L)$ and $(\gamma',L')$ in ${\bf S}\backslash M$, such that $\gamma,\gamma'$ are composed of pure segments. We also assume that $\gamma,\gamma'$ do not have a common infinite end. We choose representatives of $\gamma$ and $\gamma'$ with the minimal number of intersections. We also assume in this section that $\gamma$ is open (and in particular has rank $1$). To ease notation, we further assume that the rank $a'$ of $\gamma'$ is also $1$, the general case is entirely analogous.

The segments of $\gamma$ are denoted $\{\delta_i\}_{i\in I}$ and the segments of $\gamma'$ are denoted $\{\delta_j'\}_{j\in I'}$. Recall that $M^L_{\gamma}$ is defined in \Cref{clprop1} as the colimit of the diagram $D_{\gamma}\colon E_{\gamma}\rightarrow \mathcal{L}$. Using that the functor 
\[\on{Mor}_{\mathcal{L}}(\mhyphen,M^{L'}_{\gamma'})\colon \mathcal{L}^{\on{op}}\longrightarrow \on{RMod}_R\] 
preserves limits, it follows that $\on{Mor}_{\mathcal{L}}(M^L_{\gamma},M^{L'}_{\gamma'})$ is the limit of the diagram 
\begin{equation}\label{glinteq1}
\on{Mor}_{\mathcal{L}}(\mhyphen,M^{L'}_{\gamma'})\circ D_{\gamma}^{\on{op}}\colon E_{\gamma}^{\on{op}}\longrightarrow \on{RMod}_R\,.
\end{equation}
In the following we fully describe the diagram \eqref{glinteq1}. We will see that the diagram \eqref{glinteq1} is equivalent to the direct sum in the stable functor category $\on{Fun}(E^{\on{op}}_{\gamma},\on{RMod}_R)$ of a collection of very manageable diagrams. A subset of these diagrams correspond to the intersections of $\gamma$ and $\gamma'$ which we will show in the case that $\gamma,\gamma'$ are matching curves to be the only summands with nonzero limits in $\on{RMod}_R$. 

We proceed with the constructions of the summands of \eqref{glinteq1} associated to the different types of intersections.

\vspace{1em}

\noindent {\bf Singular intersections}

For this case, we assume that $\gamma\neq \gamma'$. Assume that the endpoints of $\gamma, \gamma'$ intersect in a vertex $v\in \mathcal{T}_0$. Note that in this case $\gamma$ and $\gamma'$ are singular and thus $L\simeq \phi^*(Q),L'\simeq \phi^*(Q')$. Since $\gamma$ and $\gamma'$ are pure, reversing their orientations does not change $M^L_{\gamma}$ or $M_{\gamma'}^L$. We may thus assume that $\gamma$ and $\gamma'$ both start at $v$. 

Using the rotational symmetry at $v$, we may assume that the first segment $\delta_1$ is given by the segment $\delta^1$ of the first type at $v$. We distinguish two cases. Either the first segment $\delta_1'=\delta^i$ of $\gamma'$ is identical to $\delta_1$, i.e.~$i=1$, or it is not. We begin with the case $i\neq 1$. By \Cref{locmorfig1}, we have $\on{Mor}_{\mathcal{L}}(M_{\delta^1}^L,M_{\delta^i}^{L'})\simeq \on{Mor}_R(Q,Q')[1-i]$ and the corresponding morphisms have support at $v$. Using \Cref{supprem}, it follows that there is a direct summand of the diagram \eqref{glinteq1} which restricts at $(\Lambda^2_0\times \{1\})^{\on{op}}$ to the diagram 
\[
\begin{tikzcd}[row sep=tiny, column sep=tiny]
{\on{Mor}_R(Q,Q')[1-i]} \arrow[rd] &   & 0 \arrow[ld] \\
                                   & 0 &             
\end{tikzcd}
\]
and vanishes on $(\Lambda^2_0\times \{i\})^{\on{op}}$ for $1<i\in I'$. Passing to limits, we get a direct summand $\on{Mor}_R(Q,Q')[1-i]$ of $\on{Mor}_{\mathcal{L}}(M_{\gamma}^L,M_{\gamma'}^{L'})$.

We now consider the case $i=1$. In that case, the matching curves $\gamma$ and $\gamma'$ are composed of $m$ identical segments $\delta_1=\delta_1',\dots,\delta_m=\delta_m'$, starting at $v$, such that $\delta_{m+1}\neq \delta_{m+1}'$ (this uses that $\gamma\neq\gamma'$). Let $v'\in \mathcal{T}_0$ be the vertex where $\delta_{m+1},\delta_{m+1}'$ lie. We choose $\delta_{m+1},\delta_{m+1}'$ such that they have a minimal number of intersections.

We distinguish the following two cases.
\begin{enumerate}[1)]
\item $\delta_{m+1}(0)\in \partial \Sigma_{v'}$ precedes $\delta_{m+1}'(0)\in \partial \Sigma_{v'}$ in the clockwise orientation of $\partial \Sigma_{v'}$.
\item $\delta_{m+1}(0)\in \partial \Sigma_{v'}$ follows $\delta_{m+1}'(0)\in \partial \Sigma_{v'}$ in the clockwise orientation of $\partial \Sigma_{v'}$
\end{enumerate} 

We find in either case for $2\leq i\leq m$
\begin{align*}
\on{Mor}_{\mathcal{L}}(M_{\delta_1}^L,M_{\delta_1'}^{L'})&\simeq \on{Mor}_R(Q,Q') \\
\on{Mor}_{\mathcal{L}}(M_{\delta_i}^L,M^{L'}_{\delta_i'})& \simeq \on{Mor}(L,L')\\
\end{align*} 
In case 1), we find 
\[ \on{Mor}_{\mathcal{L}}(M_{\delta_{m+1}}^L,M^{L'}_{\delta_{m+1}'}) \simeq \on{Mor}(L,L')\,,\]
whereas in case 2), we find
\[\on{Mor}_{\mathcal{L}}(M_{\delta_{m+1}},M_{\delta_{m+1}'}) \simeq 0\,,\]
see \Cref{sec:localHom}.
By part (1) of \Cref{lilem1}, each of the above $R$-modules also gives rise to a direct summand of the morphism object $\on{Mor}_{\mathcal{L}}(M_{\delta}^L,M^{L'}_{\gamma'})$. In the case 1), using again \Cref{supprem}, we thus find a summand of \eqref{glinteq1} which restricts on $(\Lambda^2_0\times \{1\})^{\on{op}}$ up to equivalence to the diagram
\begin{equation}\label{rdiag1}
\begin{tikzcd}[row sep=tiny, column sep=tiny]
\on{Mor}_R(Q,Q') \arrow[rd] &                  & {\on{Mor}(L,L')} \arrow[ld, "\simeq"'] \\
                         & {\on{Mor}(L,L')} &                                  
\end{tikzcd}
\end{equation}
on $(\Lambda^2_0\times \{i\})^{\on{op}}$ for $2\leq i \leq m$ to the constant diagram with value $\on{Mor}(L,L')$ and vanishes on the remaining parts of $E_{\gamma}^{\on{op}}$. The limit of this summand gives us a direct summand $\on{Mor}_R(Q,Q')\subset \on{Mor}_{\mathcal{L}}(M_{\gamma},M_{\gamma'})$, as desired. In the case 2), we analogously find a summand of \eqref{glinteq1} which restricts on $(\Lambda^2_0\times \{1\})^{\on{op}}$ to the diagram \eqref{rdiag1}, on $(\Lambda^2_0\times \{i\})^{\on{op}}$ for $2\leq i <m$ to the constant diagram with value $\on{Mor}(L,L')$, on $(\Lambda^2_0\times \{m\})^{\on{op}}$ to the diagram

\begin{equation*}
\begin{tikzcd}[row sep=tiny, column sep=tiny]
{\on{Mor}(L,L')} \arrow[rd, "\simeq"] &                  & 0 \arrow[ld, "\simeq"'] \\
                                      & {\on{Mor}(L,L')} &                        
\end{tikzcd}
\end{equation*}
and vanishes on the remaining parts of $E_{\gamma}^{\on{op}}$. Using the two equivalences  $\on{Mor}(L,L') \simeq \on{Mor}_R(Q,\phi_*\phi^*(Q'))$ and $\phi_*\phi^*(Q')\simeq Q'\oplus Q'[1-n]$, one finds the limit to be given by the desired direct summand $\on{Mor}_R(Q,Q')[-n]\subset \on{Mor}_{\mathcal{C}}(M_{\gamma'},M_{\gamma})$.

\vspace{1em}

\noindent {\bf Crossings}

Assume that $\gamma$ and $\gamma'$ have a crossing and consider the segments $\delta$ and $\delta'$ of $\gamma$ and $\gamma'$, respectively, describing the curves at the crossing. The segments $\delta$ and $\delta'$ are located at a vertex $v\in \mathcal{T}_0$. Since pure segments cannot have crossings between themselves, the crossing between $\gamma$ and $\gamma'$ does not arise as a crossing between segments in $\Sigma_v$.

Before and after the crossing, the two curves are composed of $m\geq 0$ identical segments. If $m$ were infinite, we could find different representatives for $\gamma,\gamma'$ with one intersection less which would contradict our assumptions. We can thus assume $m$ to be finite. 

We can choose representatives of $\gamma$ and $\gamma'$ such that the crossing lies on an edge connecting two vertices $v,v'\in \mathcal{T}_0$. We assume that $\gamma$ and $\gamma'$ are oriented such that locally around the crossing, they both first pass through $\Sigma_v$ and then $\Sigma_{v'}$. We consider all segments $\delta_{x+i}$ and $\delta'_{y+i}$ with $x\in I$, $y\in I'$ and $0\leq i\leq m+1$ for which there exist representatives of $\gamma$ and $\gamma'$, such that the induced representatives of the segments $\delta_{x+i}$ and $\delta'_{y+i}$ share the crossing in question. Note also that by assumption $\delta_{x+i}=\delta_{y+i}'$ for $1\leq i\leq m$.

The segments $\delta_x$ and $\delta_y'$ both lie at a vertex $v_1\in \mathcal{T}_0$ and the segments $\delta_{x+m+1}$ and $\delta_{y+m+1}'$ also both lie at a vertex $v_2\in \mathcal{T}_0$. We distinguish the following two cases. 

\begin{enumerate}[1)]
\item The point $\delta_y'(1)\in \partial \Sigma_{v_1}$ follows the point $\delta_x(1)\in \partial\Sigma_{v_1}$ in the clockwise direction on the intersected boundary component of $\partial \Sigma_{v_1}$ and the point $\delta_{y+m+1}'(0)\in \partial \Sigma_{v_2}$ follows the point $\delta_{x+m+1}(0)\in \partial \Sigma_{v_2}$  in the clockwise direction on the intersected boundary component of $\partial \Sigma_{v_2}$. This means that the crossing goes from $\gamma$ to $\gamma'$
\item The point $\delta_y'(1)\in \partial \Sigma_{v_1}$ precedes the point $\delta_x(1)\in \partial\Sigma_{v_1}$  in the clockwise direction on the intersected boundary component of $\partial \Sigma_{v_1}$ and the point $\delta_{y+m+1}'(0)\in \partial \Sigma_{v_2}$ precedes the point $\delta_{x+m+1}(0)\in \partial \Sigma_{v_2}$  in the clockwise direction on the intersected boundary component of $\partial \Sigma_{v_2}$. This means that the crossing goes form $\gamma'$ to $\gamma$.
\end{enumerate}

It follows from \Cref{sec:localHom} that there exist direct summands
\[\on{Mor}(L,L')\subset \on{Mor}_{\mathcal{L}}(M^L_{\delta_{x+i}},M^{L'}_{\delta'_{y+i}})\]
for $0\leq i \leq m+1$ in the case 1) and $1\leq i \leq m$ in the case 2).

In the case 1), we thus find a direct summand of the diagram \eqref{glinteq1} which restricts on $(\Lambda^2_0\times \{x+i\})^{\on{op}}$ for $0\leq i\leq m$ up to equivalence to the constant diagram with value $\on{Mor}(L,L')$ and vanishes on the remaining parts of $E_{\gamma}^{\on{op}}$. Passing to limits, we obtain the desired summand $\on{Mor}(L,L')\subset \on{Mor}_{\mathcal{L}}(M^L_{\gamma},M^{L'}_{\gamma'})$. In the case 2), we similarly find a direct summand of the diagram \eqref{glinteq1} which restricts on $(\Lambda^2_0\times \{x\})^{\on{op}}$ to the diagram 
\[
\begin{tikzcd}[row sep=tiny, column sep=tiny]
0 \arrow[rd] &   & \on{Mor}(L,L')\arrow[ld, "\simeq"'] \\
             & \on{Mor}(L,L') &                       
\end{tikzcd}
\]
on $(\Lambda^2_0\times \{x+i\})^{\on{op}}$ for $1\leq i\leq m-1$ to the constant diagram with value $\on{Mor}(L,L')$, on $(\Lambda^2_0\times \{x+m\})^{\on{op}}$ to the diagram 
\[
\begin{tikzcd}[row sep=tiny, column sep=tiny]
\on{Mor}(L,L') \arrow[rd, "\simeq"] &   & 0 \arrow[ld] \\
                       & \on{Mor}(L,L') &             
\end{tikzcd}
\]
and vanishes on the remaining parts of $E_{\gamma}^{\on{op}}$. Passing to limits, we find the direct summand $\on{Mor}(L,L')[-1]\subset \on{Mor}_{\mathcal{L}}(M^L_{\gamma},M^{L'}_{\gamma'})$.

\vspace{1em}

\noindent {\bf Boundary intersections}

We assume that $\gamma$ and $\gamma'$ both intersect a boundary component $B$ of ${\bf S}\backslash M$ and distinguish two cases. 

\begin{enumerate}[1)]
\item The intersection of $\gamma'$ and $B$ follows the intersection of $\gamma$ and $B$ in the orientation of $B$ induced by the clockwise orientation of ${\bf S}$.
\item The intersection of $\gamma'$ and $B$ precedes the intersection of $\gamma$ and $B$ in the orientation of $B$ induced by the clockwise orientation of ${\bf S}$.
\end{enumerate}

We can assume that both $\gamma$ and $\gamma'$ start at $B$ and have $m$ identical segments $\delta_1=\delta_1',\dots,\delta_m=\delta_m'$, and the segments $\delta_{m+1}\neq \delta'_{m+1}$ both lie at a vertex $v\in \mathcal{T}_0$. In the case 1), we find that $\delta_{m+1}'(0)\in \partial \Sigma_{v}$ follows $\delta_{m+1}(0)\in \partial \Sigma_v$ on a boundary component of $\partial \Sigma_v$ in the clockwise direction. In the case 2), we find that $\delta'_{m+1}(0)\in \partial \Sigma_{v}$ precedes $\delta_{m+1}(0)\in \partial \Sigma_v$ in the clockwise orientation of $\partial \Sigma_v$.

We thus find direct summands
\begin{equation}\label{brdsumeq} 
\on{Mor}(L,L')\subset \on{Mor}_{\mathcal{L}}(M^L_{\delta_{i}},M^{L'}_{\delta'_{i}})\,,
\end{equation}
with $1\leq i\leq m+1$ in the case 1) and $1\leq i\leq m$ in the case 2).

In the case 1), the summands \eqref{brdsumeq} assemble using \Cref{lilem1} and \Cref{supprem} into a direct summand of \eqref{glinteq1} which has constant value $\on{Mor}(L,L')$ on $(\Lambda^2_0\times \{i\})^{\on{op}}$ for $1\leq i\leq m$ and vanishes on the remainder of $E_{\gamma}^{\on{op}}$. Passing to limits, we obtain the desired summand $\on{Mor}(L,L')\subset \on{Mor}_{\mathcal{L}}(M^L_{\gamma},M^{L'}_{\gamma'})$. In the case 2), the summands \eqref{brdsumeq} assemble using \Cref{lilem1} into a direct summand of \eqref{glinteq1} which has constant value $\on{Mor}(L,L')$ on $(\Lambda^2_0\times \{i\})^{\on{op}}$ for $1\leq i\leq m-1$, takes the value 
\begin{equation*}
\begin{tikzcd}[row sep=tiny, column sep=tiny]
{\on{Mor}(L,L')} \arrow[rd, "\simeq"] &                  & 0 \arrow[ld] \\
                                      & {\on{Mor}(L,L')} &                        
\end{tikzcd}
\end{equation*}
on $(\Lambda^2_0\times \{m\})^{\on{op}}$ and vanishes on the remainder of $E_{\gamma}^{\on{op}}$. The limit of this summand vanishes.

\begin{remark}\label{intbdryint}
Let $\eta$ be a curve composed of segments, which is not necessarily a matching curve. The above arguments generalize to describe direct summands of $\on{Mor}_{\mathcal{L}}(M_{\eta}^L,M_{\gamma'}^{L'})$ associated to singular intersections, crossings and directed boundary intersections (defined as for matching curves) of $\eta$ and $\gamma'$. 

In the case that $\eta$ begins or ends at an internal edge $e$, and $\gamma'$ has a segment $\delta'$ which begins or ends at $e$, we have the following further direct summands. Reorienting $\eta$ if necessary, we may assume that $\eta$ starts at $e$. We denote by $\delta$ the first segment of $\eta$ and by $v\in \mathcal{T}_0$ the vertex at which $\eta$ is located. We assume that $\delta'$ also lies at $v$ and, reorienting $\gamma'$ if necessary, that $\delta'$ also begins at $e$. We choose $\delta'$ in such a way that it has the minimal number of intersections with $\delta$. This arrangement roughly looks as follows (for $n=3$). 

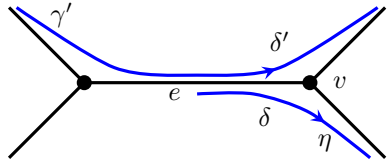
\begin{center}
\begin{tikzpicture}[decoration={markings, 
	mark= at position 0.7 with {\arrow{stealth}}}]
   \draw[color=blue][very thick][postaction={decorate}] plot [smooth] coordinates {(-0.9,1) (0.4,0.2) (1.5,0.1)(2.6,0.2) (3.9,1)};
  \draw[color=blue][very thick][postaction={decorate}] plot [smooth] coordinates {(1.5,-0.15) (2.25,-0.15) (3,-0.4) (3.8,-1)};
  \node (6) at (0,0){};  
  \node (7) at (3,0){};
  \node (3) at (-0.3,0.9){$\gamma'$};
  \node ()  at (2.4,-0.45){$\delta$};
  \node ()  at (2.6,0.55){$\delta'$};
  \node (4) at (3.2,-0.8){$\eta$};
  \node ()  at (1.2,-0.15){$e$};
  \node ()  at (3.4, 0){$v$};
  \fill (6) circle (0.1);
  \fill (7) circle (0.1);
  \draw[very thick]
  (0,0) -- (3,0)--(4,1)
  (3,0)--(4,-1)
  (-1,-1) -- (0,0)
  (-1,1) -- (0,0);
\end{tikzpicture}
\end{center}

If $\delta'(0)\in\partial \Sigma_v$ follows $\delta(0)\in \partial \Sigma_v$ in the clockwise direction in the boundary component of $\partial \Sigma_v$, we find $\on{Mor}_{\mathcal{L}}(M^L_{\delta},M^{L'}_{\delta'})\simeq \on{Mor}(L,L')$, if not then $\on{Mor}_{\mathcal{L}}(M^L_{\delta},M^{L'}_{\delta'})\simeq 0$. Assuming that we are in the former case, the above construction for directed boundary intersections generalizes to this situation and provides us with a direct summand $\on{Mor}(L,L')\subset \on{Mor}_{\mathcal{L}}(M^L_{\eta},M^{L'}_{\gamma'})$.
\end{remark}

\noindent {\bf Non-intersections}

A relevant non-intersection appears every time both curves $\gamma,\gamma'$ pass through $\Sigma_v\subset \Sigma_\mathcal{T}$ for $v\in \mathcal{T}_0$, so that the corresponding sections $\delta,\delta'$ at $v$ satisfy that $\on{Mor}_{\mathcal{L}}(M^L_\delta,M^{L'}_{\delta'})\neq 0$ or $\on{Mor}_{\mathcal{L}}(M^{L'}_{\delta'},M^L_{\delta})\neq 0$, even though $\delta$ and $\delta'$ do not have a singular intersection and are neither part of a crossing or a boundary intersection. 

Since $\delta$ and $\delta'$ do not have a crossing, we find by the computations of \Cref{sec:localHom} that there must exist a boundary component $B$ of $\Sigma_v$ which intersects both $\gamma$ and $\gamma'$. We choose to orient $\gamma$ and $\gamma'$ so that both $\delta$ and $\delta'$ start at $B$. Before and after $\delta$ and $\delta'$, the curves $\gamma$ and $\gamma'$ are composed of $m$ identical segments. If $m$ is infinite, $\gamma$ and $\gamma'$ have a common infinite end. We may thus assume that $m$ is finite. We find $x\in I$ and $y\in I'$, such that the $m$ common segments of $\gamma$ and $\gamma'$ are $\delta_{x+i}=\delta'_{y+i}$ with $1\leq i\leq m$. We assume without loss of generality that $x,y\geq 1$. The segments $\delta_x$ and $\delta'_{y}$ both lie at a vertex $v_1\in \mathcal{T}_0$ and the segments $\delta_{x+m+1}$ and $\delta'_{y+m+1}$ both lie at a vertex $v_2\in \mathcal{T}_0$. The discussion now resembles the discussion in the case of a crossing above. However in contrast to the situation there, we find the following two possibilities.
\begin{enumerate}[1)]
\item The point $\delta'_y(1)\in \partial \Sigma_{v_1}$ follows the point $\delta_x(1)\in \partial\Sigma_{v_1}$  in the clockwise direction on the intersected boundary component of $\partial \Sigma_{v_1}$ and the point $\delta'_{y+m+1}(0)\in \partial \Sigma_{v_2}$ precedes the point $\delta_{x+m+1}(0)\in \partial \Sigma_{v_2}$  in the clockwise direction on the intersected boundary component of $\partial \Sigma_{v_2}$. 
\item The point $\delta_x(1)\in \partial\Sigma_{v_1}$ precedes the point $\delta'_y(1)\in \partial \Sigma_{v_1}$  in the clockwise direction on the intersected boundary component of $\partial \Sigma_{v_1}$ and the point $\delta_{x+m+1}(0)\in \partial \Sigma_{v_2}$ follows the point $\delta_{y+m+1}'(0)\in \partial \Sigma_{v_2}$  in the clockwise direction on the intersected boundary component of $\partial \Sigma_{v_2}$. 
\end{enumerate}

We continue with the case 1), the case 2) is analogous. The direct summand of \eqref{glinteq1} corresponding to the non-intersection is given by the diagram which restricts on each $\Lambda^2_0\times \{x+i\}$ for $0\leq i\leq m$ to the diagram 
\[
\begin{tikzcd}[column sep=0]
{\on{Mor}(L,L')} \arrow[rd, "\simeq"] &                       & {\on{Mor}(L,L')} \arrow[ld, "\simeq"] \\
                                           & {\on{Mor}(L,L')} &                                           
\end{tikzcd}
\]
restricts on $\Lambda^2_0\times \{x+m+1\}$ to the following diagram 
\[
\begin{tikzcd}[column sep=0]
{\on{Mor}(L,L')} \arrow[rd, "\simeq"] &                       & 0 \arrow[ld] \\
                                           & {\on{Mor}(L,L')} &             
\end{tikzcd}
\]
and vanishes on the remainder of $E_\gamma^{\on{op}}$. The limit of this summand thus vanishes, as desired.

\subsection{The proofs of \texorpdfstring{\Cref{homthm,homthm2}}{Theorems 6.4 and 6.5}}\label{sec:Homproof}

\begin{proof}[Proof of \Cref{homthm}]
We distinguish the following two cases.

{\bf Case 1: $\gamma$ is open.}

In \Cref{sec:globalHom} we have associated to each intersection (or relevant non-intersection) of $\gamma,\gamma'$ a direct summand of the diagram \eqref{glinteq1}, which passing to limits gave direct summands of $\on{Mor}_{\mathcal{C}}(M^L_{\gamma},M^{L'}_{\gamma'})$, matching exactly the desired description of the morphism object in \Cref{homthm}. Note that if $\gamma'$ is closed of rank $a'$, then it is locally equivalent to an $a'$-fold direct sum. In this case, each direct summand of $\on{Mor}_{\mathcal{C}}(M^L_{\gamma},M^{L'}_{\gamma'})$ thus appears with multiplicity $a'$.

By \Cref{lilem1}, the diagram \eqref{glinteq1} fully arises from morphisms between the segments of $\gamma$ and $\gamma'$. These morphisms are all each accounted for in exactly one of the direct summands of the diagram \eqref{glinteq1} described above. These direct summands thus describe the entirety of the diagram \eqref{glinteq1} and we may conclude that \Cref{homthm} holds for $\gamma$ not closed.

\vspace{1em}

{\bf Case 2: $\gamma$ is closed.}

The global section $M^L_{\gamma}$ is given by the coequalizer of the diagram \eqref{s1obeq}, so that the $R$-module $\on{Mor}_{\mathcal{L}}(M^L_{\gamma},M^{L'}_{\gamma'})$ is equivalent to the equalizer of the following diagram in $\on{RMod}_R$.
\begin{equation}\label{6.3eq1}
\begin{tikzcd}
{\on{Mor}_{\mathcal{L}}((M^L_{\eta})^{\oplus {a}},M^{L'}_{\gamma'})} \arrow[r, shift left] \arrow[r, shift right] & {\on{Mor}_{\mathcal{L}}((Z^L_e)^{\oplus {a}},M^{L'}_{\gamma'})}
\end{tikzcd}
\end{equation}
The $R$-module $\on{Mor}_{\mathcal{L}}((M^L_{\eta})^{\oplus {a}},M^{L'}_{\gamma'})\simeq \on{Mor}_{\mathcal{L}}(M^L_{\eta},M^{L'}_{\gamma'})^{\oplus a}$ can be determined using its description as the limit of the $a$-fold direct sum of the diagram \eqref{glinteq1}. All direct summands of \eqref{glinteq1} associated to intersections between $\eta$ and $\gamma'$ yield direct summands of the diagram \eqref{6.3eq1} of the form 
\[
\begin{tikzcd}
N \arrow[r, shift right] \arrow[r, shift left] & {0}
\end{tikzcd}
\]
so that passing to limits yields the direct summands $N\subset \on{Mor}_{\mathcal{C}}(M^L_{\gamma},M^{L'}_{\gamma'})$. However, not all summands of $\on{Mor}_{\mathcal{C}}(M^L_{\gamma},M^{L'}_{\gamma'})$ have to be of this form, because there can be crossings between $\gamma$ and $\gamma'$ which do not restrict to a crossing between $\eta$ and $\gamma'$. The remainder of this proof consists of an account of these summands.  

Since $\eta$ is not a matching curve, the direct summands associated to the intersections of $\eta$ and $\gamma'$ in general do not describe the entirety of $\on{Mor}_{\mathcal{L}}((M^L_{\eta})^{\oplus a},M^{L'}_{\gamma'})$. We additionally have to include the direct summands described in \Cref{intbdryint} to obtain the entirety of the morphism object $\on{Mor}_{\mathcal{L}}((M^L_{\eta})^{\oplus a},M^{L'}_{\gamma'})$. We denote by $\eta$ the curve obtained by cutting the closed curve $\gamma$ open at $e$. We denote the first segment of $\eta$ by $\delta_2$ and the last segment of $\eta$ by $\delta_1$. We denote the compose of $\delta^{1}$ and $\delta^{2}$ at $e$ by $\mu$. We now show the following. 
\begin{enumerate}[a)]
\item Every crossing from $\gamma$ to $\gamma'$ (or from $\gamma'$ to $\gamma$), which does not give rise to a crossing of $\mu$ and $\gamma'$, leads to a direct summands $\on{Mor}(L,L')^{\oplus a}$ (or $\on{Mor}(L,L')^{\oplus a}[-1]$) in the equalizer of \eqref{6.3eq1}.
\item Direct summands of $\on{Mor}_{\mathcal{L}}((M^L_{\eta})^{\oplus a},M^{L'}_{\gamma'})$ as described in \Cref{intbdryint} and direct summands of $\on{Mor}_{\mathcal{L}}((Z^L_e)^{\oplus a},M^{L'}_{\gamma'})$ do not persist in the equalizer of \eqref{6.3eq1} if they cannot be accounted for by a crossing as above.
\end{enumerate}

If $\gamma'$ is closed, statement a) needs to be modified, as in the previous case where $\gamma$ is not closed, to include the $a'$-fold multiplicity. Together with the previous discussion, the statements a) and b) then imply that $\on{Mor}_{\mathcal{C}}(M^L_{\gamma},M^{L'}_{\gamma'})$ is the direct sum of the desired number of suspensions or deloopings of $\on{Mor}(L,L')$, concluding this proof.

We begin by showing part a). The curves $\gamma$ and $\gamma'$ can be chosen so that their crossing restricts to an intersection between $\gamma$ and the composite of two segments $\delta'_{y}$ and $\delta'_{y+1}$, with $y\in I'$, of $\gamma'$ which end, respectively, begin at $e$. We denote the two vertices incident to $e$ by $v_1$ and $v_2$ and, reorienting $\eta,\gamma'$ if necessary, can assume that both $\delta'_{y}$ and $\delta_{1}$ lie at $v_1$ and both $\delta'_{y+1}$ and $\delta_{2}$ lie at $v_2$. 

We distinguish the following two cases. 

\begin{enumerate}[1)]
\item $\delta_1(1)\in \partial \Sigma_{v_1}$ precedes $\delta'_{y}(1)\in \partial \Sigma_{v_1}$ in the clockwise orientation on the intersected boundary component of $\partial \Sigma_{v_1}$ and $\delta_2(0)\in \partial \Sigma_{v_2}$ precedes $\delta'_{y+1}(0)\in \partial \Sigma_{v_2}$ in the clockwise orientation on the intersected boundary component of $\partial \Sigma_{v_2}$. This means that the crossing goes from $\gamma$ to $\gamma'$
\item $\delta_1(1)\in \partial \Sigma_{v_1}$ follows $\delta_{y}(1)\in \partial \Sigma_{v_1}$ in the clockwise orientation on the intersected boundary component of $\partial \Sigma_{v_1}$ and $\delta_2(0)\in \partial \Sigma_{v_2}$ follows $\delta'_{y+1}(0)\in \partial \Sigma_{v_2}$ in the clockwise orientation on the intersected boundary component of $\partial \Sigma_{v_2}$. This means that the crossing goes from $\gamma'$ to $\gamma$.
\end{enumerate}

In the case 1), we find by \Cref{intbdryint} a direct summand $\on{Mor}(L,L')^{\oplus a}\oplus \on{Mor}(L,L')^{\oplus a}\subset \on{Mor}_{\mathcal{L}}((M^L_{\eta})^{\oplus a},M^{L'}_{\gamma'})$, where the first copy of $\on{Mor}(L,L')^{\oplus a}$ arises from the boundary intersection of $\delta_2$ and $\delta'_{y+1}$ and the second copy of $\on{Mor}(L,L')^{\oplus a}$ arises from the boundary intersection of $\delta_{1}$ and $\delta'_{y}$. In terms of the diagram \eqref{6.3eq1}, the crossing corresponds to a direct summand of \eqref{6.3eq1} of the form
\[
\begin{tikzcd}
{\on{Mor}(L,L')^{\oplus a}\oplus \on{Mor}(L,L')^{\oplus a}} \arrow[r, "{(0,\on{id})}"', shift right] \arrow[r, "{(\on{id},0)}", shift left] & {\on{Mor}(L,L')^{\oplus a}}
\end{tikzcd}
\]
whose equalizer gives a direct summand $\on{Mor}(L,L')^{\oplus a}\subset \on{Mor}_{\mathcal{L}}(M^L_{\gamma},M^{L'}_{\gamma'})$. 

In the case 2), there are no morphisms in $\on{Mor}_{\mathcal{L}}((M^L_{\eta})^{\oplus a},M^{L'}_{\gamma'})$ associated to the crossing. Using part (2) of \Cref{lilem1}, we thus find a direct summand of \eqref{6.3eq1} corresponding to the crossing of the following form.
\[
\begin{tikzcd}
0 \arrow[r, shift right] \arrow[r, shift left] & {\on{Mor}(L,L')^{\oplus a}}
\end{tikzcd}
\]
Passing to equalizers, we thus obtain the direct summand $\on{Mor}(L,L')^{\oplus a}[-1]$ of $\on{Mor}_{\mathcal{L}}(M^L_{\gamma},M^{L'}_{\gamma'})$. This concludes the proof of a).

For part b), we assume that part of $\gamma'$ passes along $e$ but does not partake in a crossing of $\gamma'$ and $\gamma$. We employ the same notation for the segments of $\gamma'$ at $e$ as above. In this case, either $\delta_{1}(1)\in \partial \Sigma_{v_1}$ and $\delta'_{y+1}(0)\in \partial \Sigma_{v_2}$ both follow or both precede $\delta_{y}'(1)\in \partial \Sigma_{v_1}$ and $\delta_{2}(0)\in \partial \Sigma_{v_2}$, respectively. The corresponding direct summand of \eqref{6.3eq1} is thus of the following form.
\begin{equation}\label{eqdiag2}
\begin{tikzcd}
{\on{Mor}(L,L')^{\oplus a}} \arrow[r, "0"', shift right] \arrow[r, "\on{id}", shift left] & {\on{Mor}(L,L')^{\oplus a}}
\end{tikzcd}
\end{equation}
The equalizer of \eqref{eqdiag2} vanishes, showing b).
\end{proof}

\begin{proof}[Proof of \Cref{homthm2}]
The proof of \Cref{homthm2} goes along the same lines as the proof of \Cref{homthm}.

{\bf Case 1: $\gamma$ is open.}
 
As we showed in \Cref{sec:globalHom}, each self-crossing or directed boundary self-intersection gives rise to direct summands of $\on{Mor}_{\mathcal{L}}(M^L_{\gamma},M^{L'}_{\gamma})$ given by suspensions or deloopings of $\on{Mor}(L,L')$ or $\on{Mor}_{R}(Q,Q')$ of the desired form. 

Assume that all segments of $\gamma$ are of the second type. Given a segment $\delta$ of the second type, we have $\on{Mor}_{\mathcal{L}}(M_{\delta}^L,M_{\delta}^{L'})\simeq \on{Mor}(L,L')$, see \Cref{locmorfig3}. Similarly, we have $\on{Mor}_{\mathcal{L}}(Z^L_e,Z^{L'}_e)\simeq \on{Mor}(L,L')$. The constant diagram, as always up to equivalence, with value $\on{Mor}(L,L')$ thus defines a direct summand of \eqref{glinteq1}. Passing to limits, we obtain the direct summand $\on{Mor}(L,L')\subset \on{Mor}_{\mathcal{L}}(M^L_{\gamma},M^{L'}_{\gamma})$. 

Assume that exactly one segment of $\gamma$ is of the first type. The curve $\gamma$ is thus singular and exactly one end lies at a vertex of $\mathcal{T}$. Reorienting $\gamma$ if necessary, we can assume that $\gamma$ begins at the vertex. The endomorphisms of the segments of $\gamma$ thus yield a direct summand of \eqref{glinteq1}, which assigns to $(\Lambda^2_0)^{\on{op}}\times \{1\}$ the diagram \eqref{rdiag1} and is constant on the remainder of $E_{\gamma}^{\on{op}}$ with value $\on{Mor}(L,L')$. The limit of this direct summand is given by $\on{Mor}_R(Q,Q')\subset \on{Mor}_{\mathcal{L}}(M_{\gamma},M_{\gamma})$.

If exactly two segments of $\gamma$ are of the first type, then $\gamma$ is singular, and begins and ends at vertices of $\mathcal{T}$. Let $N$ be the number of segments of $\gamma$. The endomorphisms of the segments of $\gamma$ yield a direct summand of \eqref{glinteq1}, which assigns to $(\Lambda^2_0)^{\on{op}}\times \{1\}$ the diagram \eqref{rdiag1}, to $(\Lambda^2_0)^{\on{op}}\times \{N-1\}$ the diagram 
\[
\begin{tikzcd}[row sep=tiny, column sep=tiny]
{\on{Mor}(L,L')} \arrow[rd, "{\simeq}"] &                  &  \on{Mor}_R(Q,Q') \arrow[ld] \\
                         & {\on{Mor}(L,L')} &                                  
\end{tikzcd}
\]
and to the remainder of $E_{\gamma}^{\on{op}}$ the constant diagram with value $\on{Mor}(L,L')$. To compute the limit of this diagram, one uses $\on{Mor}(L,L')\simeq \on{Mor}_R(Q,\phi_*\phi^*(Q'))$ and $\phi_*\phi^*(Q')\simeq Q'\oplus Q'[1-n]$. The resulting direct summand is given by $\on{Mor}_R(Q,Q')\oplus \on{Mor}_R(Q,Q')[-n]\subset \on{Mor}_\mathcal{L}(M_{\gamma},M_{\gamma})$. If the endpoints of $\gamma$ furthermore coincide, then the singular intersections of $\on{Mor}_\mathcal{L}(M_{\gamma},M_{\gamma})$ produce two further direct summands given by suspensions or deloopings of $\on{Mor}_R(Q,Q')$. 

The above identified direct summand of $\on{Mor}_{\mathcal{L}}(M^L_{\gamma},M^{L'}_{\gamma})$ account for the entire morphism object and match the count given in \Cref{homthm2}. This thus concludes the proof in the case that $\gamma$ is not closed.

\vspace{1em}

{\bf Case 2: $\gamma$ is closed.}

We need to compute the equalizer of \eqref{6.3eq1} with $R=k$ a field. Showing that each self-crossing of $\gamma$ contributes a direct summand equivalent to $\left(\on{Mor}(L,L)\oplus \on{Mor}(L,L)[-1]\right)^{\oplus a^2}$ to the equalizer of \eqref{6.3eq1} is analogous to the discussion in the proof of \Cref{homthm} in the case that $\gamma$ is closed. A novel argument is required to determine the endomorphisms not corresponding to self-crossings. The morphisms from $M_{\eta}^L$ to $M_{\gamma}^{L}$ arising from the morphisms between the sections associated to the common segments of $\eta,\gamma$ (all of the second type) contribute a direct summand $\on{Mor}(L,L)^{\oplus a^2}\subset \on{Mor}(M^L_{\eta},M^{L}_{\gamma})$. Each of the two composites with the pointwise inclusion $Z^L_{e}\rightarrow M^{L}_{\eta}$ in $\mathcal{L}$ arising from an end of $\eta$ at $e$ yields an equivalence between the direct summand $\on{Mor}(L,L)^{\oplus a^2}$ of both $\on{Mor}_{\mathcal{L}}(M^L_{\eta},M^{L}_{\gamma})$ and $\on{Mor}_{\mathcal{L}}(Z^L_e,M^{L}_{\gamma})$. As we explain below, these equivalences give rise to the following direct summand of \eqref{6.3eq1}.
\begin{equation}\label{gcleqeq}
\begin{tikzcd}[column sep=huge]
{\on{Mor}(L,L)^{\oplus a^2}} \arrow[r, "{\on{id}}", shift left] \arrow[r, "{\mathscr{J}\circ (\mhyphen)\circ \mathscr{J}^{-1}}"', shift right] & {\on{Mor}(L,L)^{\oplus a^2}}
\end{tikzcd}
\end{equation}
Above $\mathscr{J}$ denotes the monodromy equivalence, which was assumed to be a single Jordan block with eigenvalue $\lambda \in k\backslash \{0\}$. The matrix $\mathscr{J}^{-1}$ is the inverse matrix. Using the equivalence $\on{Mor}(L,L)^{\oplus a^2}\simeq \on{Mor}(L^{\oplus a},L^{\oplus a})$, the morphism $\mathscr{J}\circ (\mhyphen)\circ \mathscr{J}^{-1}$ takes a map $L^{\oplus a}\rightarrow L^{\oplus a}$, precomposes it with the endomorphism of $L^{\oplus a}$ given by $\mathscr{J}^{-1}$ and postcomposes it with the endomorphism of $L^{\oplus a}$ given by $\mathscr{J}$. The equalizer of \eqref{gcleqeq} is equivalent to the fiber of the morphism
\begin{equation}\label{eq:id-jj} {\on{Mor}(L,L)^{\oplus a^2}} \xlongrightarrow{\on{id}-\mathscr{J}\circ (\mhyphen)\circ \mathscr{J}^{-1}} {\on{Mor}(L,L)^{\oplus a^2}}\,.
\end{equation}
A direct computation shows, that the above morphism maps an $a\times a$-matrix $(m_{i,j})_{1\leq i,j\leq a}$ with entries in $\on{Mor}(L,L)$ to the $a\times a$-matrix $(m_{i,j}')_{1\leq i,j\leq a}$ with 
\[ -m_{i,j}'=\frac{m_{i+1,j}}{\lambda}+\sum_{l>0}(-1)^l\frac{\lambda m_{i,j-l}+m_{i+1,j-l}}{\lambda^{l+1}}
\]
where we set $m_{i,j}=0$ for $j\leq 0$ or $i>a$. The kernel of \eqref{eq:id-jj} thus consists of upper triangular matrices $(m_{i,j})_{1\leq i,j\leq a}$, satisfying that $m_{i,j}=m_{i+1,j+1}$ for all $1\leq i,j\leq a-1$.  The fiber of \eqref{eq:id-jj} splits as its kernel, which is equivalent to $\on{Mor}(L,L)^{\oplus a}$ and the delooping of its cokernel, which is given by $\on{Mor}(L,L)^{\oplus a}[-1]$. This shows that we obtain the desired direct summand of $\on{Mor}_{\mathcal{C}}(M^L_{\gamma},M^{L}_{\gamma})$. We have again determined the entire morphism object $\on{Mor}_{\mathcal{L}}(M_{\gamma}^L,M_{\gamma}^{L})$, showing \Cref{homthm2}.

To arrive at the direct summand \eqref{gcleqeq}, we need to describe the equivalence of the direct summand $\on{Mor}(L,L)^{\oplus a^2}$ obtained from composing the two equivalences contained in the following diagram in $\on{RMod}_k$, arising from restricting morphisms to the two endpoints of $\eta$. 
\begin{equation}\label{morobjdiag}
\begin{tikzcd}[column sep=tiny]
                                                 & \on{Mor}(L,L)^{\oplus a^2} \arrow[d, hook] \arrow[rd, "\simeq"]                        &                                                \\
\on{Mor}(L,L)^{\oplus a^2} \arrow[d, hook] \arrow[ru, "\simeq"] & {\on{Mor}_{\mathcal{L}}\left(\left(M^L_{\eta}\right)^{\oplus a},M^{L}_{\gamma}\right)} \arrow[rd] \arrow[ld] & \on{Mor}(L,L)^{\oplus a^2} \arrow[d, hook]                    \\
{\on{Mor}_{\mathcal{L}}\left(\left(Z^L_{e}\right)^{\oplus a},M^{L}_{\gamma}\right)}   &                                                                         & {\on{Mor}_{\mathcal{L}}\left(\left(Z^L_{e}\right)^{\oplus a},M^{L}_{\gamma}\right)}
\end{tikzcd}
\end{equation}

This equivalence is affected by two kinds of monodromy. Firstly, the monodromy of the perverse schober $\mathcal{F}_\mathcal{T}(k)$ along $\gamma$, which is defined below and shown to be trivial. Secondly, the monodromy equivalence of $(\gamma,L)$, given by the Jordan-block $\mathscr{J}$. 

Given a pure segment $\delta$ of the second type in ${\bf S}$ lying at a vertex $v'$ going from the edge $e'$ to the edge $e''$, we define the transport (partial monodromy) of $\mathcal{F}_\mathcal{T}(k)$ along $\delta$ as the following autoequivalence of $\mathcal{F}_\mathcal{T}(k)(e')=\mathcal{F}_\mathcal{T}(k)(e'')=\on{RMod}_{R[t_{n-2}]}$.
\begin{itemize}
\item If $\delta$ wraps around $v'$ in the counterclockwise direction, we define the transport as the composite of the left adjoint of $\mathcal{F}_\mathcal{T}(k)(v\rightarrow e')$ with the functor $\mathcal{F}_\mathcal{T}(k)(v\rightarrow e'')$.
\item If $\delta$ wraps around $v'$ in the clockwise direction, we define the transport as the composite of the right adjoint of $\mathcal{F}_\mathcal{T}(k)(v\rightarrow e')$ with the functor $\mathcal{F}_\mathcal{T}(k)(v\rightarrow e'')$.
\end{itemize}

The monodromy of $\mathcal{F}_\mathcal{T}(k)$ along $\gamma$ is defined as the autoequivalence of $\on{RMod}_{R[t_{n-2}]}$ obtained from composing the transports of all segments of $\gamma$. It is computed as follows.

Consider a segment $\delta_i$ of $\gamma$ lying at $v^i$ and connecting $e^i$ and $e^{i+1}$. We assume that $\delta_i$ turns counterclockwise, the clockwise case is analogous. Up to the action of the paracyclic twist functor $T_{v^i}$, see \Cref{cyctwrem} and \Cref{paratwprop}, we can assume that $\mathcal{F}_\mathcal{T}(k)(v^i\rightarrow e^{i})$ and $\mathcal{F}_\mathcal{T}(k)(v^i\rightarrow e^{i+1})$ are given by $T_1\circ \varrho_1$, respectively, $T_2\circ \varrho_2$, with $T_1,T_2$ each given by one of the two autoequivalences $\on{id},T$, using the notation from \Cref{constr:Ginzburgschober}. The left adjoint $\varsigma_2$ of $\varrho_1$ is right adjoint to $\varrho_2$, see \Cref{adjsrem}. Since $\varsigma_2$ is a fully faithful functor, it thus follows that $\varrho_2\circ\varsigma_2\simeq \on{id}_{\on{RMod}_{R[t_{n-2}]}}$. The transport along $\delta_i$ is thus a power of the involution $T$. The total monodromy along $\gamma$ is hence given by an $i$-th power of the involution $T$ for some $i\in \mathbb{Z}$. The integer $i$ is even, as follows from inspecting the construction of $\mathcal{F}_\mathcal{T}(k)$ and from the observation that there are an equal number of halfedges being transversed by $\gamma$ which carry an even or odd labeling in the chosen total orders. It follows that the monodromy of $\mathcal{F}_\mathcal{T}(k)$ along $\gamma$ is trivial. 

We proceed by spelling out in detail how the equivalence \eqref{morobjdiag} arises from the composition of the two kinds of monodromy.

The morphism object $\on{Mor}(M_{\eta}^L,M_{\gamma}^{L})$ is given by the limit of \eqref{glinteq1}. Instead of directly computing the equivalence in \eqref{morobjdiag}, we can thus equivalently compute the equivalences obtained by tracing along the segments of $\eta$, i.e.~compose the endomorphisms of $\on{Mor}(L,L)^{\oplus a^2}$ contained in the commutative diagrams 

\begin{adjustwidth}{-1in}{-1in}
\[
\begin{tikzcd}
{\on{Mor}(L,L)^{\oplus a^2}} \arrow[r, "\simeq"] \arrow[d, "\simeq"]                                          & {\on{Mor}(L,L)^{\oplus a^2}} \arrow[r, "\simeq"] \arrow[d, "\simeq"]            & {\on{Mor}(L,L)^{\oplus a^2}} \arrow[d, "\simeq"]                                                                                           \\
{\on{Mor}\left({Z_{e^i}^L(e^i)},M^{L}_{\delta_i}(e^i)\right)^{\oplus a^2}}  \arrow[d, "\simeq"] &  {\on{Mor}\left(M^{L}_{\delta_i}(e^i),M^{L}_{\delta_i}(e^i)\right)^{\oplus a^2}} \arrow[l, "\simeq"'] & {\on{Mor}_{\mathcal{V}^n_{\phi^*}}\left(M_{\delta_i}^L(v^i),M^{L}_{\delta_i}(v^i)\right)^{\oplus a^2}} \arrow[d, hook] \arrow[l, "{\mathcal{F}_\mathcal{T}(v^i\rightarrow e^{i})}"] \\
{\on{Mor}_{\mathcal{L}}\left(\left(Z^L_{e^i}\right)^{\oplus a},M^{L}_{\gamma}\right)}    &                                                                                 & {\on{Mor}_{\mathcal{L}}\left(\left(M^L_{\delta_i}\right)^{\oplus a},M^{L}_{\gamma}\right)}   \arrow[ll]                                              
\end{tikzcd}
\]
\[
\begin{tikzcd}
{\on{Mor}(L,L)^{\oplus a^2}} \arrow[d, "\simeq"] \arrow[r, "\simeq"]                                                                       & {\on{Mor}(L,L)^{\oplus a^2}} \arrow[d, "\simeq"] \arrow[r, "\simeq"]                    & {\on{Mor}(L,L)^{\oplus a^2}} \arrow[d, "\simeq"]                                                                  \\
{\on{Mor}_{\mathcal{V}^n_{\phi^*}}\left(M_{\delta_i}^L(v^i),M^{L}_{\delta_i}(v^i)\right)^{\oplus a^2}} \arrow[d, hook] \arrow[r, "{\mathcal{F}_\mathcal{T}(v^i\rightarrow e^{i+1})}"'] & {\on{Mor}\left(M^{L}_{\delta_i}(e^{i+1}),M^{L}_{\delta_i}(e^{i+1})\right)^{\oplus a^2}} \arrow[r, "\simeq"]  & {\on{Mor}\left({Z_{e^{i+1}}^L(e^{i+1})},M^{L}_{\delta_i}(e^{i+1})\right)^{\oplus a^2}} \arrow[d, hook] \\
{\on{Mor}_{\mathcal{L}}\left(\left(M^L_{\delta_i}\right)^{\oplus a},M^{L}_{\gamma}\right)} \arrow[rr]                                      &                                                                                         & {\on{Mor}_{\mathcal{L}}\left(\left(Z^L_{e^{i+1}}\right)^{\oplus a},M^{L}_{\gamma}\right)}                        
\end{tikzcd}
\]
\end{adjustwidth}

\noindent with $1\leq i \leq N$, where $\gamma$ has $N$ segments $\delta_i$, $i\in I=\mathbb{Z}/N\mathbb{Z}$, lying at $v^i\in \mathcal{T}_0$ and beginning and ending at the edge $e^i$, respectively, $e^{i+1}$. For $i=N$, the morphism $\mathcal{F}_{\mathcal{T}}(v^n\rightarrow e^{1})$ in the above diagram needs to additionally be composed with the endomorphism $\mathscr{J}\circ (\mhyphen)$. To justify this, we make a choice of coequalizer $M_{\gamma}^{L}$ of \eqref{s1obeq}, which assigns to each edge $v\rightarrow e$ the morphism $M_{\eta}^{L}(v\rightarrow e)^{\oplus a}$ (strictly and not just up to equivalence), except for the morphism $v^n\rightarrow e^1$, where $M_{\eta}^{L}(v^n\rightarrow e^1)^{\oplus a}$ is composed with the equivalence $\mathscr{J}$. The precomposition with $\mathscr{J}^{-1}$ arises from the appearance of $\mathscr{J}^{-1}$ in \eqref{s1obeq}.

Before explaining why the equivalences in the above diagram also describe the monodromy of $\mathcal{F}_\mathcal{T}(k)$, we have to take care of some further contributions. These are the equivalences 
\begin{align*} 
E_i\colon \on{Mor}(M^L_{\delta_{i-1}}(e_i),M^L_{\delta_{i-1}}(e_i))^{\oplus a^2}&\simeq \on{Mor}({Z^L_{e_i}}(e_i),{Z^L_{e_i}}(e_i))^{\oplus a^2}\\
& \simeq \on{Mor}(M^L_{\delta_{i}}(e_i),M^L_{\delta_{i}}(e_i))^{\oplus a^2}\end{align*}
which arise from the fact that the inclusions $Z_{e_i}^L\rightarrow M^L_{\delta_{i}},M^L_{\delta_{i-1}}$ were only specified up to $k$-linear equivalence (since the local sections were defined as Kan extensions). Under the equivalences with $\on{Mor}(L,L)^{\oplus a^2}$, the equivalence $E_i$ corresponds to an endomorphism $\mathscr{D}^{-1}(\mhyphen)\mathscr{D}$ of $\on{Mor}(L,L)^{\oplus a^2}$, where $\mathscr{D}$ is some invertible diagonal $a\times a$-matrix with entries in $\pi_0\on{Map}(L,L)$, all of whose diagonal entries are identical. It follows that $\mathscr{D}^{-1}(\mhyphen)\mathscr{D}$ acts as the identity on $\on{Mor}(L,L)^{\oplus a^2}$. Thus, the equivalences $E_i$ do not contribute to the diagram \eqref{gcleqeq}.

The left and right adjoints of the functors $\mathcal{F}_\mathcal{T}(v^i\shortrightarrow e^{i})$ contained in the middle parts of the above diagrams are fully faithful and hence define right inverses of $\mathcal{F}_\mathcal{T}(v^i\shortrightarrow e^{i})$ on morphism objects. If $\delta_i$ wraps clockwise, then $M_{\delta_i}^L(v^i)\simeq \on{radj}(\mathcal{F}_\mathcal{T}(v^i\shortrightarrow e^{i}))(M_{\delta_i}^L(e^i))$. Similarly, if $\delta_i$ wraps counterclockwise, then $M_{\delta_i}^L(v^i)\simeq \on{ladj}(\mathcal{F}_\mathcal{T}(v^i\shortrightarrow e^{i}))(M_{\delta_i}^L(e^i))$. Tracing along the middle part of the above diagram thus yields a contribution of the monodromy of $\mathcal{F}_\mathcal{T}(k)$. This concludes the argument, why the direct summand \eqref{gcleqeq} of the diagram \eqref{6.3eq1} appears.
\end{proof}

We describe in \Cref{ex:counter} global sections arising from matching data, whose matching curves are non-pure, and for which an arising morphism object does not simply count intersections.

\begin{example}\label{ex:counter}
Consider the $4$-gon with an ideal triangulation with dual trivalent spanning $\mathcal{T}$ and two matching curves $\gamma,\gamma'$, which can be depicted as follows.
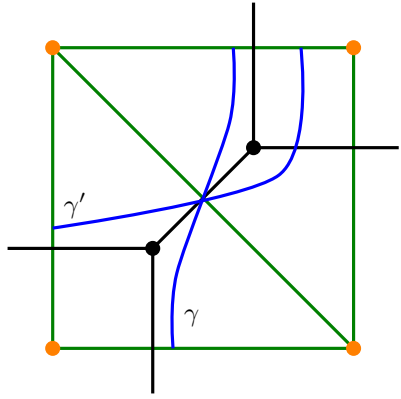
\begin{center}
\begin{tikzpicture}
  \draw[color=ao, very thick]
    (0, 4) -- (0, 0)
    (4, 0) -- (0, 0)
    (4, 4) -- (4, 0)
    (0, 4) -- (4, 4)
    (0, 4) -- (4, 0)
    ;
  \draw[very thick]
   (1.33,1.33) -- (2.67,2.67)
   (1.33,1.33) -- (1.33,-0.6)
   (1.33,1.33) -- (-0.6,1.33)
   (2.67,2.67) -- (4.6,2.67)
   (2.67,2.67) -- (2.67,4.6);
   \draw[color=blue][very thick] plot [smooth] coordinates {(0,1.6) (3, 2.3) (3.3, 4)};
   \draw[color=blue][very thick] plot [smooth] coordinates {(2.4,4) (2.35, 3) (1.65,1) (1.6, 0)};
  \node (0) at (2.67, 2.67){};
  \node (1) at (1.33, 1.33){};
  \node (2) at (0,0){};
  \node (3) at (4,0){};
  \node (4) at (4,4){};
  \node (5) at (0,4){};
  
  \node () at (0.3,1.9){$\gamma'$};
  \node () at (1.85,0.4){$\gamma$};
  \fill (0) circle (0.1);
  \fill (1) circle (0.1);
  \fill[color=orange] (2) circle (0.1);
  \fill[color=orange] (3) circle (0.1);
  \fill[color=orange] (4) circle (0.1);
  \fill[color=orange] (5) circle (0.1);
\end{tikzpicture}
\end{center}
The matching curve $\gamma$ is pure, whereas $\gamma'$ is not pure. There is a directed boundary intersection from $\gamma$ to $\gamma'$ and a crossing from $\gamma'$ to $\gamma$. For any $L,L'\in \on{RMod}_{R[t_1]}$, there are apparent matching data $(\gamma,L)$ and $(\gamma',L')$.

The computation of $\on{Mor}_{\mathcal{H}(\mathcal{T},\mathcal{F}_\mathcal{T}(R)}(M_{\gamma}^{L},M_{\gamma'}^{L'})$ boils down to the computation of the limit  of the $E_{\gamma}^{\on{op}}$-indexed diagram:
\begin{equation}\label{eq:counter}
\begin{tikzcd}
  & 0 \arrow[rd] \arrow[ld] &                     & {\on{Mor}(L,L')[1]} \arrow[rd, "\simeq"] \arrow[ld, "\alpha"'] &                  \\
0 &                                                     & {\on{Mor}(L,L')} &                                                             & {\on{Mor}(L,L')[1]}
\end{tikzcd}
\end{equation}
Spelling out its construction, we see that the morphism $\alpha$ is given by precomposition with the morphism $s_L\colon L\rightarrow T_{\on{RMod}_{R[t_{1}]}}(L)[1]\simeq L[-1]$, arising from the fiber and cofiber sequence of endofunctors of $\on{RMod}_{R[t_1]}$
\[ \phi^*\phi_*\rightarrow \on{id}_ {\on{RMod}_{R[t_{1}]}}\xrightarrow{s} T_{\on{RMod}_{R[t_{1}]}}[1]\] describing the cotwist functor of the adjunction $\phi^*\dashv \phi_*$. The natural transformation $s$ is also called the section of the cotwist functor. It evaluates to a zero morphism at $L$ if $L\in \on{Im}(\phi^*)$, but evaluates non-trivially for other $L$ (such as $L=R[t_1]$). The limit of the diagram \eqref{eq:counter} is given by the fiber of the morphism $\alpha$. If $\alpha=0$, the limit thus consists of the sum of two copies of suspensions or deloopings of $\on{Mor}(L,L)$, matching the number of intersections of $\gamma$ and $\gamma'$. If $\alpha\neq 0$, this is not true.
\end{example}

\subsection{Indecomposability of objects}\label{sec:indecomposability}

\begin{corollary}\label{cor:indecomp}
Let $R=k$ be a field and $L\in \on{RMod}_{k[t_{n-2}]}$, such that $H_0\on{End}(L)\simeq k$.
\begin{enumerate}[(1)]
\item Let $(\gamma,L)$ be an open matching datum in ${\bf S}\backslash M$, such that $\gamma$ is finite and pure. The discrete endomorphism ring $H_0\on{End}(M_{\gamma}^L)$ is local and $M_{\gamma}^L$ thus indecomposable.
\item Let $(\gamma,L)$ be a closed matching datum, whose monodromy matrix is a single Jordan block and such that $\gamma$ is pure. The discrete endomorphism ring $H_0\on{End}(M_{\gamma}^L)$ is local and $M_{\gamma}^L$ thus indecomposable.
\end{enumerate}
\end{corollary}

\begin{remark}
For $L=\phi^*(k),k[t_{n-2}],k[t_{n-2}^\pm]\in \on{RMod}_{k[t_{n-2}]}$ and $n\geq 3$, we have $H_0(\on{End}(L))\simeq k$. We can thus apply \Cref{cor:indecomp} to matching data with local values of different sizes. 
\end{remark}

\begin{proof}[Proof of \Cref{cor:indecomp}.]
We begin with proving part (1). By \Cref{homthm2}, we have a complete description of the degree $0$ endomorphisms $H_0\on{End}(M_{\gamma}^L)$. The degree $0$ morphisms arising from crossings, directed boundary intersections and singular intersections of distinct endpoints generate an ideal $J$. Every endomorphism of $M_{\gamma}^L$ is the sum of an endomorphism in $J$ and a $k$-linear multiple of the identity.
We argue below, that an endomorphism $\alpha=\beta+\lambda \on{id}_{M_{\gamma}^L}$ of $M_{\gamma}^L$ with $\beta\in J$ is an equivalence if and only if $\lambda \in k$ is nonzero. It thus follows, that $J$ is the unique maximal ideal, showing that $H_0\on{End}(M_{\gamma}^L)$ is a local ring. This shows part (1).

The argument generalizes to the setting in part (2). We can again form a maximal ideal out of the contributions from crossings, out of the copy of $H_0\on{Mor}(L,L)^{\oplus a-1}$ consisting of strictly upper triangular matrices in the kernel of \eqref{eq:id-jj} and the copy of $H_0(\on{Mor}(L,L)^{\oplus a}[-1])$. 

Let $\beta \in J$ and $\lambda \in k$. We conclude this proof by showing that the endomorphism $\alpha=\beta+\lambda \on{id}_{M_{\gamma}^L}$ of $M_{\gamma}^L$ is invertible if and only if $\lambda\neq 0$. The morphism $\alpha$ is a natural transformation between sections and thus invertible if and only if the morphism $\alpha(x)\colon M_{\gamma}^L(x)\rightarrow M_{\gamma}^L(x)$ in $\mathcal{F}_\mathcal{T}(x)$ is an equivalence for all $x\in \on{Exit}(\mathcal{T})$. Locally, near a vertex $v$  of $\mathcal{T}$ with incident edges $e_1,\dots,e_n$, $\mathcal{F}_\mathcal{T}$ is described by the conservative functor 
\[ (\varrho_1,\dots,\varrho_n)\colon \mathcal{F}_\mathcal{T}(v)\simeq \mathcal{V}^n_{\phi^*}\longrightarrow \on{RMod}_{R[t_{n-2}]}^{\times n}\simeq \prod_{i=1}^n \mathcal{F}_\mathcal{T}(e_i)\,.\] 
It thus suffices to consider the case that $x=e\in \on{Exit}(\mathcal{T})$ is an edge of $\mathcal{T}$.

Before we proceed, let us note that the endomorphisms of $M_{\gamma}^L$ appearing in the direct summand $H_0\left(\on{Mor}(L,L)[-1]^{\oplus i^{\on{cr}}(\gamma,\gamma)}\right)$ evaluate to zero at all objects $x\in \on{Exit}(\mathcal{T})$. This is easily seen from tracing through their construction. We thus do not need to consider these endomorphisms in the following to determine whether $\alpha(x)$ is invertible.

Let now $x=e$ be an edge of $\mathcal{T}$ and choose a halfedge with incident vertex $v$. The object $M_{\gamma}^L(e)\in \on{RMod}_{R[t_{n-2}]}$ decomposes into the sum of the evaluations at $e$ of the local sections associated with the different segments of $\gamma$ at $v$. There are three pure segments at $v$ with an intersection with $e$, two of the second type and one of the first type. The corresponding local sections are pairwise semiorthogonal, since the morphisms between these sections arise from directed boundary intersections in $\Sigma_v$ at the boundary component intersecting $e$. From this directedness, it follows that $\alpha(e)$ is of a block upper triangular shape and hence an equivalence if and only if for any such segment $\delta$ at $v$, the restriction of $\alpha(e)$ to an endomorphism of $\bigoplus_{\delta_i=\delta} M_{\delta}^L(e)$ is an equivalence, where the sum runs over the segments $\delta_i$ with $i\in I$ of $\gamma$ at $v$ which are identical to $\delta$. 

If $\delta$ is of the first type at $v$ ending at $e$, it is obvious that the restriction of $\alpha(e)$ to $\bigoplus_{\delta_i=\delta} M_{\delta}^L(e)$ is an equivalence if and only if $\lambda\neq 0$.

Suppose that $\delta$ is a segment of the second type at $v$ ending at $e$. Let $B$ be the boundary component of $\Sigma_v$ corresponding to the chosen halfedge of $e$ (this essentially means $B$ intersects $e$, but also remains correct if $e$ is a loop). We consider each segment $\delta_i=\delta$ as oriented, so that it ends at $B$. An orientation of a segment $\delta_i$, with $i\in I$, determines an orientation of $\gamma$. We can choose the map $\gamma$ into ${\bf S}\backslash M$ such that all crossings of $\gamma$ involving two segments $\delta_i=\delta_j=\delta$ with $i\neq j$ are such that they appear after $\delta_i$ in the induced orientation of $\gamma$ by $\delta_i$ (and thus also after $\delta_j$ in the induced orientation of $\gamma$ by $\delta_j$). The curve $\gamma$ induces an embedding of the segments $\{\delta_i=\delta\}_{i\in I}$ into $\Sigma_v\subset {\bf S}\backslash M$, which endows these segments with a total order, given by the clockwise order of the intersections with $B$. This total order has the following good property: an endomorphism of $\gamma$ arising from a self-crossing which evaluates non-trivially at $e$ always arises via evaluating at $e$ a morphism from $M_{\delta_i}^L$ to $M_{\delta_j}^L$ with $\delta_i=\delta_j=\delta$ and $i<j$. 

If there is a directed boundary self-intersection of $\gamma$ so that it goes from the segment $\delta_i=\delta$ to the segment $\delta_j=\delta$, we further distinguish two cases. In the first case, the directed boundary intersection follows after $\delta_i$ in the induced oriented of $\gamma$ (by $\delta_i)$. It is immediate that $i<j$. In the second case, the directed boundary is before $\delta_i$ in the induced orientation of $\gamma$. In this case, we have $j=i+1$, as otherwise there would have to be crossings appearing before $\delta_i$ in the orientation of $\gamma$. We swap the order of $i$ and $j=i+1$. 

In the constructed order of the segments $\delta_i=\delta$, we find that if $\alpha(e)$ induces a nonzero morphism from $M_{\delta_i}^L$ to $M_{\delta_j}^L$ with $\delta^i=\delta^j=\delta$, then $i<j$. This means that $\alpha(e)$ restricts to an endomorphism of $\bigoplus_{\delta_i=\delta} M_{\delta}^L(e)$ given by an upper triangular matrix, with diagonal entries $\lambda$. It follows that for any edge $e$, $\alpha(e)$ is an equivalence if and only if $\lambda\neq 0$, concluding the proof.
\end{proof}

\subsection{The Jacobian gentle algebra}\label{subsec:Jacobian}

\begin{definition}
Let $k$ be the commutative ground ring. The Jacobian algebra $\mathscr{J}_\mathcal{T}$ is defined as the $0$-th homology $k$-algebra $\mathscr{J}_\mathcal{T}=H_0(\mathscr{G}_\mathcal{T})$ of the relative Ginzburg algebra $\mathscr{G}_\mathcal{T}$.
\end{definition}

Consider the sub-quiver $P_\mathcal{T}$ of the quiver $\tilde{Q}_\mathcal{T}$ of \Cref{gqdef} consisting of all vertices and the subset of arrows $\{a_{v,i,i+1}\}_{v,i}$, where $i+1$ denotes the counterclockwise halfedge predecessor of the halfedge $i$, meaning that $P_\mathcal{T}$ consists of all arrows of $\tilde{Q}_\mathcal{T}$ lying in degree $0$. It is immediate from the definition of $\mathscr{G}_\mathcal{T}$ that $\mathscr{J}_\mathcal{T}\simeq kP_{\mathcal{T}}/I$, where the ideal $I=\{a_{v,i,i+1}a_{v,i-1,i}\}$ consists of certain paths of length two. With this, it is straightforward to see that the Jacobian algebra $\mathscr{J}_\mathcal{T}$ is a gentle algebra, in the sense recalled in \Cref{def:gtl}. While the gentle algebra $\mathscr{J}_\mathcal{T}$ is finite dimensional if ${\bf S}$ has no punctures, it is infinite dimensional if there are punctures, as the cycles wrapping around the puncture do not lie in the ideal $I$.

\begin{definition}\label{def:gtl}
A $k$-algebra is called a gentle algebra if it is isomorphic to the path algebra $kQ/I$ of a finite quiver $Q$ modulo an ideal $I$ generated by paths of length $2$, such that
\begin{itemize}
\item every vertex of $Q$ has at most two incoming and two outgoing arrows.
\item For each arrow $a$, there is at most one arrow $b$ such that $ab$ lies in $I$ and there is at most one arrow $b$ such that $ba\in I$.
\item For each arrow $a$, there is at most one arrow $b$ such that $ab\notin I$ and there is at most one arrow $b$ such that $ba\notin I$.
\end{itemize}
\end{definition}

The main result of this section is the following description of the homology algebra $H_*(\mathscr{G}_\mathcal{T})$. 

\begin{proposition}\label{jacprop4}
Let $R=k$ be a commutative ring and suppose that ${\bf S}$ has no punctures.
There exists an isomorphism of graded algebras between $H_*(\mathscr{G}_\mathcal{T})$ and the tensor product $\mathscr{J}_\mathcal{T}\otimes_k k[t_{n-2}]$.
\end{proposition}

Other classes of Ginzburg algebras whose homology has been computed include non-relative Ginzburg algebras of acyclic quivers, see \cite{Her16}, and some classes of relative Ginzburg algebras whose homology is concentrated in degree $0$, see \cite[Section 8.2]{Wu21}.

\begin{remark}
If ${\bf S}$ has punctures, the curves $c_{e}$ associated with the edges of $\mathcal{T}$ have common infinite ends. We expect that a generalization of \Cref{homthm} which allows common infinite ends would allow to extend \Cref{jacprop4} to arbitrary surfaces.
\end{remark}

\begin{proof}[Proof of \Cref{jacprop4}.]
Let $L=k[t_{n-2}]\in \on{RMod}_{k[t_{n-2}]}$. Using \Cref{geomprprop}, we find that there exists an isomorphism of dg algebras 
\begin{equation}\label{jacpfeq1}
\mathscr{G}_\mathcal{T}\simeq \on{End}\left(\bigoplus_e M_{c_e}^L\right)\,,
\end{equation}
so that it suffices to construct an isomorphism between the homology algebra $H_*\on{End}\left(\bigoplus_e M_{c_e}^L\right)$ and $\mathscr{J}_\mathcal{T}\otimes_k k[t_{n-2}]$. 

Given two edges $e_1,e_2$ of $\mathcal{T}$, the associated pure matching curves $c_{e_1},c_{e_2}$ do not intersect, except for directed boundary intersections. Applying \Cref{homthm}, we obtain for each directed boundary intersection a direct summand 
\[ k[t_{n-2}]\simeq \on{Mor}(k[t_{n-2}],k[t_{n-2}])\subset \on{End}\left(\bigoplus_e M_{c_e}^L\right)\,.\] 
This shows that there exists an equivalence in $\on{RMod}_k$
\begin{equation}\label{jacpfeq2}
H_*\on{End}\left(\bigoplus_e M_{c_e}^L\right)\simeq H_0\on{End}\left( \bigoplus_e M_{c_e}^L\right)\otimes_k k[t_{n-2}]\,.
\end{equation}
Since there exists an isomorphism of $k$-algebras $H_0\on{End}\left( \bigoplus_e M_{c_e}^L\right)\simeq \mathscr{J}_\mathcal{T}$ by \eqref{jacpfeq1}, to conclude this proof, it suffices to show that \eqref{jacpfeq2} is also an isomorphism of dg algebras (with vanishing differentials). For that, we need to compare the composition in the two dg algebras. 

We call two directed boundary intersections from $c_{e_1}$ to $c_{e_2}$ and $c_{e_2}$ to $c_{e_3}$ composable if they lie at the same boundary component $B$ of ${\bf S}\backslash M$. In this case, starting at $B$, the curves are composed of identical segments such that $c_{e_1}$ shares the same segments with both $c_{e_2}$ and $c_{e_3}$ and the two curves $c_{e_2}$ and $c_{e_3}$ share at least as many segments with each other as with $c_{e_1}$. Let $a,b\in \{1,2,3\}$ with $a\leq b$. Each generating morphisms given by $t_{n-2}^i\in k[t_{n-2}]\subset \on{End}(\bigoplus_e M_{c_e}^L)$ with $i\geq 0$ associated to the boundary intersections of $c_{e_a},c_{e_b}$ at $B$, or the endomorphisms of $M_{c_{e_a}}^L$ if $a=b$, corresponds to a morphism between the sections $M_{c_{e_a}}^L,M_{c_{e_b}}^L$ which restricts for each shared segment $\delta$ of $c_{e_a},c_{e_b}$ to the endomorphism $t_{n-2}^i\in k[t_{n-2}]\simeq \on{End}(M_{\delta}^L)$. We thus see, that the composite of $t_{n-2}^i\colon M_{c_{e_a}}^L\rightarrow M_{c_{e_b}}^L$ with $t_{n-2}^j\colon M_{c_{e_b}}^L\rightarrow M_{c_{e_c}}^L$ is given by $t_{n-2}^{i+j}\colon M_{c_{e_a}}^L\rightarrow M_{c_{e_c}}^L$ for all $a\leq b\leq c \in\{1,2,3\}$. 

We also note that if two boundary intersections are not composable, then the corresponding endomorphisms of $\bigoplus_e M_{c_e}^L$ compose to zero.

Comparing the two sides of \eqref{jacpfeq1} in degree $0$, one obtains that directed boundary intersection from $c_{e_1}$ to $c_{e_2}$ are in bijection with nonzero paths from $e_1$ to $e_2$ in $\mathscr{J}_\mathcal{T}$. Using that the equivalence \eqref{jacpfeq2} induces an isomorphism of $k$-algebras in degree $0$, we further obtain that two boundary intersections are composable if and only if the corresponding paths in $\mathscr{J}_\mathcal{T}$ are composable with nonzero composite. The description of the product of the generating morphisms of $\on{End}\left(M_{c_e}^L\right)$ given above implies that $\alpha$ commutes with the multiplications and is thus an isomorphism of dg algebras, concluding this proof.
\end{proof}

\subsection{Morphisms in the case of \texorpdfstring{$2$}{2}-angulations}\label{subsec:n=2}

For the following, we assume that the marked surface ${\bf S}$ is equipped with a $2$-angulation, instead of an $n$-angulation with $n\geq 3$. We describe some of the morphism objects in terms of intersections for global sections arising from matching data with local value in the image of $\phi^*$. For other local values such descriptions do not always hold, see also the proof below. 

\begin{theorem}\label{thm:int=hom_n=2}
Let $(\gamma,\phi^*(R)),(\gamma',\phi^*(R))$ be two finite pure matching data in ${\bf S}\backslash M$ of ranks $a,a'$. Then \Cref{homthm,homthm2} apply, in the sense that the following hold: 
\begin{enumerate}[i)] 
\item If $\gamma\not=\gamma'$, then $\on{Mor}_{\mathcal{C}}(M^{\phi^*(R)}_{\gamma},M^{\phi^*(R)}_{\gamma'})\in \on{RMod}_R$ is equivalent to 
\begin{align*} (R\oplus R[-1])^{\oplus i^{\on{bdry}}(\gamma,\gamma')+ aa'i^{\on{cr}}(\gamma,\gamma')} & \oplus (R[-1]\oplus R[-2]) ^{\oplus aa' i^{\on{cr}}(\gamma',\gamma)} \\
& \oplus \bigoplus_{x\in i^{\on{sg}}(\gamma,\gamma')}R[\on{deg}_x]\,.\end{align*}  
\item If $\gamma=\gamma'$ is open and regular, then $\on{Mor}_{\mathcal{C}}(M_{\gamma}^{\phi^*(R)},M_{\gamma}^{\phi^*(R)}) \in \on{RMod}_R$ is equivalent to 
\[ (R\oplus R[-1])^{\oplus 1+i^{\on{cr}}(\gamma,\gamma)}\oplus (R[-1]\oplus R[-2])^{\oplus i^{\on{cr}}(\gamma,\gamma)}\,.\]
\item If $\gamma=\gamma'$ is open and singular, then $\on{Mor}_{\mathcal{C}}(M_{\gamma}^{\phi^*(R)},M_{\gamma}^{\phi^*(R)})\in \on{RMod}_R$ is equivalent to 
\[ \left((R\oplus R^{\oplus 2}[-1]\oplus R[-2])\right)^{\oplus i^{\on{cr}}(\gamma,\gamma)}\oplus \bigoplus_{x\in i^{\on{sg}}(\gamma,\gamma')}R[\on{deg}_x]\,. 
\] 
\item If $\gamma =\gamma'$ is a closed matching curve, $R=k$ a field, and the monodromy equivalence of $(\gamma,{\phi^*(k)})$ is given by a single $a\times a$-Jordan block, then $\on{Mor}_{\mathcal{C}}(M^{\phi^*(k)}_{\gamma},M^{\phi^*(k)}_{\gamma})\in \on{RMod}_k$ is equivalent to  
\[
\left(k\oplus k[-1]\right)^{\oplus a+a^2i^{\on{cr}}(\gamma,\gamma)}\,.
\]
\end{enumerate}
\end{theorem}

\begin{proof}
The assertions follow from the same sequence of arguments as the proofs of \Cref{homthm,homthm2}. There is however one step in the argument for which additional care must be taken: Under the assumption that $n\geq 3$, two pure segments located at the same vertex cannot have two boundary intersections, meaning the corresponding case in the second half of \Cref{locmorfig3} does not need to be considered. However, in the case $n=2$, the following configuration of two pure segments leads to two boundary intersections:\\
\begin{center}
\begin{tikzpicture}
  \node () at (0,0.65){$\delta$};
  \node () at (0,-0.75){$\delta'$};
  \node () at (-1,0.2){$e$};
  \node () at (1,0.2){$e'$};
  \node () at (0,0.25){$v$};
  \fill (0,0) circle (0.1);
  \draw[very thick] (2,0)--(-2,0);
  \draw[very thick, color=blue](1.92,0.45)--(-1.93,0.45);
  \draw[very thick, color=blue](1.92,-0.45)--(-1.93,-0.45);
  \draw[color=ao][very thick] (0,0) circle(2);
  \fill[color=orange] (0,2) circle (0.1);
  \fill[color=orange] (0,-2) circle (0.1);  
\end{tikzpicture} 
\end{center}
We have 
\[ \on{Mor}_{\mathcal{L}}(M_{\delta}^{\phi^*(R)},M_{\delta'}^{\phi^*(R)})\simeq \on{Mor}_{\mathcal{L}}(M_{\delta'}^{\phi^*(R)},M_{\delta}^{\phi^*(R)})
\simeq R\,.\]
In the passage to global morphisms arising from the above local computation of morphism objects, the support of the corresponding local morphisms is relevant, or more precisely the exact form of the diagram \eqref{glinteq1} is relevant. 
The first morphism object appears as the limit of the diagram \eqref{glinteq1}, which is found to be given as follows: 
\begin{equation}\label{eq:n=2_local_morphisms}
\begin{tikzcd}[column sep=-30]
                                                                         & {\underbrace{\on{Mor}_{\mathcal{V}^2_{\phi^*}}(M_{\delta}^{\phi^*(R)}(v),M_{\delta'}^{\phi^*(R)}(v))}_{\simeq R}} \arrow[ld, "0"'] \arrow[rd, "\simeq"] &                                                                            \\
{\underbrace{\on{Mor}(M_{\delta}^{\phi^*(R)}(e),M_{\delta'}^{\phi^*(R)}(e))}_{\simeq R}} &                                                                                                                                         & {\underbrace{\on{Mor}(M_{\delta}^{\phi^*(R)}(e'),M_{\delta'}^{\phi^*(R)}(e'))}_{\simeq R}}
\end{tikzcd}
\end{equation}
To see the vanishing of the left morphism, we use local coordinates for the perverse schober and notation as in \Cref{objeq}, to write $M_{\delta}^{\phi^*(R)}=(0\to \phi^*(R))\in \mathcal{V}^2_{\phi^*}$ and $M_{\delta}^{\phi^*(R)}=(R\oplus R[-1] \simeq \phi_*\phi^*(R)\to \phi^*(R))\in \mathcal{V}^2_{\phi^*(R)}$. That the image of any morphism $M_{\delta}^{\phi^*(R)}\to M_{\delta'}^{\phi^*(R)}$ vanishes under $\varrho_2=\on{rcof}_{1,2}=\mathcal{F}_\mathcal{T}(v\to e)$ follows from inspecting the following diagram with horizontal fiber and cofiber sequences: 
\[
\begin{tikzcd}
0 \arrow[r] \arrow[d]                                & \phi^*(R) \arrow[d, "a"] \arrow[r, "\simeq "] & \phi^*(R) \arrow[d, "0"] \\
{\phi^*(R)\oplus \phi^*(R)[-1]} \arrow[r, two heads] & \phi^*(R) \arrow[r, "0"]                      & \phi^*(R)               
\end{tikzcd}
\]
The fact that the section of the cotwist of $\phi^*\dashv \phi_*$ vanishes is thus the cause of the vanishing of the morphism in \eqref{eq:n=2_local_morphisms}. A similar computation also appears in \Cref{ex:counter}. Stated in terms of the support of morphisms, this means that the support of the boundary intersection from $\delta$ to $\delta'$ is at $v$ and $e'$ but not at $e$.

The second morphism object is the limit of a similar diagram, swapping the roles of $e$ and $e'$. 

We thus find that the boundary intersections of the segments in the case $n=2$ behave, in terms of their support, in the same way as the boundary intersections of segments in the case $n\geq 3$ (which do not come in pairs). The remainder of the proofs of \Cref{homthm,homthm2} thus apply to show the theorem.
\end{proof}

\section{Categorification of the extended mutation matrix}\label{sec:mutation}

For the entirety of this section, we fix a marked surface ${\bf S}$ equipped with an ideal $3$-angulation with dual trivalent spanning graph $\mathcal{T}$. The goal of this section is to categorify the extended mutation matrix of the cluster algebra associated with ${\bf S}$ and a choice of multi-lamination. 

\subsection{The geometric definition of the extended mutation matrix}

In the following, we recall the definition of the extended mutation matrix of the cluster algebra with coefficients of a marked surface equipped with an ideal triangulation and a multi-lamination $\Lambda$, as given by Fomin-Thurston \cite{FT12}. We express the definition in terms of the dual trivalent spanning graph $\mathcal{T}$ of the ideal triangulation.

\begin{definition}[{$\!\!$\cite[Definition 4.1]{FST08}}]\label{def2.1}
We arbitrarily label the internal edges of $\mathcal{T}$ by $e_1,\dots,e_m$. We define the quiver $Q_\mathcal{T}$ as follows.
\begin{itemize}
\item The vertices are the internal edges of $\mathcal{T}$.
\item Let $e_i\neq e_j$ be two edges which are not loops. We add an arrow $a\colon e_i\rightarrow e_j$ for each vertex $v$ of $\mathcal{T}$ incident to halfedges of $e_i,e_j$ at which the halfedge of $e_j$ precedes the halfedge of $e_i$ in the cyclic (counterclockwise) order. The arrows of $Q_\mathcal{T}$ thus go in the clockwise direction. 
\item For each loop $e_i$, we add further arrows obtained as follows. Consider the unique edge $e_j$ such that $e_i$ and $e_j$ are incident to the same vertex of $\mathcal{T}$, meaning that $e_j$ is dual to the outer edge of the self-folded ideal triangle containing the dual of $e_i$. For $l\neq i,j$, we add an arrow $e_l\rightarrow e_i$ for each arrow $e_l\rightarrow e_j$ and an arrow $e_i\rightarrow e_l$ for each arrow $e_j\rightarrow e_l$.
\end{itemize}
The signed adjacency matrix of $\mathcal{T}$ is the skew-symmetric $m\times m$-matrix $B(\mathcal{T})=(b_{i,j})$ where $b_{i,j}$ is the number of arrows from $e_i$ to $e_j$ minus the number of arrows from $e_j$ to $e_i$ in $Q_\mathcal{T}$.
\end{definition}

\begin{definition}[$\!\!${\cite[Definition 12.1]{FT12}}]
A lamination curve is a curve $\gamma\colon U\rightarrow {\bf S}\backslash M$ with $U=S^1,[0,1],[0,\infty),(-\infty,\infty)$ such that 
\begin{itemize}
\item $\gamma$ does not self-intersect.
\item all endpoints of $\gamma$ lie in $\partial {\bf S}\backslash M$.
\item the curve does not bound any unpunctured disc, once-punctured disc or unpunctured $1$-gon in ${\bf S}$.
\item if $U$ is not compact, then at the infinite ends the curve spirals around a puncture.
\item if $U=(-\infty,\infty)$, then $\gamma$ is not homotopic to a curve both of whose ends spiral around the same puncture $p$ and which lies in a contractible neighborhood of $p$ containing no further punctures.
\end{itemize}
Lamination curves are considered as equivalence classes under homotopies fixing endpoints. A lamination $\lambda$ on ${\bf S}$ consists of a collection of pairwise non-intersecting lamination curves in ${\bf S}$. A multi-lamination $\Lambda=(\lambda_1,\dots,\lambda_l)$ on ${\bf S}$ is a collection of $l\geq 1$ laminations on ${\bf S}$.
\end{definition}

\begin{figure}[ht]
\begin{center}
\begin{tikzpicture}
\draw[color=ao][very thick] (0,0) circle(2);
\node (0) at (0.6,0.4){};
\node (6) at (-0.6,-0.5){};
\node (1) at (2,0){};
\node (2) at (0.618,1.902){};
\node (3) at (-1.618,1.176){};
\node (4) at (-1.618,-1.176){};
\node (5) at (0.618,-1.902){};

\fill[red] (0) circle(0.1);
\fill[orange] (1) circle(0.1);
\fill[orange] (2) circle(0.1);
\fill[orange] (3) circle(0.1);
\fill[orange] (4) circle(0.1);
\fill[orange] (5) circle(0.1);
\fill[red] (6) circle(0.1);

\draw[color=blue, very thick]  plot [smooth] coordinates {(-2,0) (-0.1,-0.2) (0,-2)};
\draw[color=blue, very thick]  plot [smooth] coordinates {(-0.846,1.81) (-0.2,0.2) (1.81,-0.846)};
\draw[color=blue, very thick]  plot [smooth] coordinates {(0.4,0.4) (0.6,0.17) (0.86,0.4) (0.6,0.69) (0.28,0.4) (0.6,0.05) (0.98,0.4) (0,2)};
\end{tikzpicture}
\caption{A lamination (in blue) with one spiraling curve in a surface with boundary (in green) with $7$ marked points. Boundary marked points are in orange, punctures in red.
}\label{fig:lam}
\end{center}
\end{figure}

Some more pretty examples and counterexamples of laminations can be found in \cite[Figures 32 and 33]{FT12}.

\begin{definition}\label{intsgndef} 
Denote the internal edges of $\mathcal{T}$ by $e_1,\dots,e_m$. 
\begin{itemize}
\item Let $\gamma_i$ be a lamination curve and $e_j$ not a loop. We call a crossing of $\gamma_i$ with $e_j$ positive (or negative), if their local arrangement is as depicted on the left (or right) in \Cref{intsgn}. We denote the signed count of such crossings of $\gamma_i$ and $e_j$ by $(e_j,\gamma_i)$.
\item Let $\gamma_i$ be a lamination curve and $e_j$ a loop. Let $e_l$ be unique other edge incident to the same vertex of $\mathcal{T}$ as $e_j$. We define $(e_j,\gamma_i)\coloneqq (e_l,\tilde{\gamma}_i)$, where $\tilde{\gamma}_i$ is the lamination curve obtained by replacing each infinite end of $\gamma_i$ spiraling around a puncture $p$ by the infinite end spiraling around $p$ in the opposite direction.
\item The shear coordinates of a lamination $\lambda$ with respect to $\mathcal{T}$ are given by the $m$-tuple $v_{\lambda,\mathcal{T}}\in \mathbb{Z}^m$ whose $j$-th entry is given by 
\[ (v_{\lambda,\mathcal{T}})_j=\sum_{\gamma\in \lambda} (e_j,\gamma)\,.\] 
\end{itemize}   
\end{definition}

\begin{figure}[h]
\begin{center}
\begin{tikzpicture}
  \draw[color=blue!255, very thick] plot [smooth] coordinates {(-0.8,-1) (0.1,-0.15) (1.5, 0.15) (2.6, 1) };
  \node (0) at (0,0){};  
  \node (6) at (1.8,0){};  
  \node (8) at (0.8,1){\bf +1};
  \node (2) at (0.8,0.2){$e_j$};
  \node (3) at (-0.3,-0.8){$\gamma_i$};
  \fill (0) circle (0.1);
  \fill (6) circle (0.1);
  \draw[very thick]
  (0,0)
  (2,0)
  (-0.5,-1)
  (-0.5,1)
  (1.8,0) -- (2.8,1)
  (1.8,0) -- (2.8,-1)
  (-1,-1) -- (0,0)
  (-1,1) -- (0,0)
  (0,0) -- (1.8,0); 
\end{tikzpicture}
\hspace{3em}
\begin{tikzpicture}
  \draw[color=blue!255, very thick] plot [smooth] coordinates {(-0.8,1) (0.1,0.15) (1.5, -0.15) (2.6, -1) };
  \node (0) at (0,0){}; 
  \node (6) at (1.8,0){};  
  \node (8) at (0.8,1){\bf -1};
  \node (2) at (0.8,0.2){$e_j$};
  \node (3) at (-0.3,0.8){$\gamma_i$};
  \fill (0) circle (0.1);
  \fill (6) circle (0.1);

  \draw[very thick]
  (0,0)
  (2,0)
  (-0.5,-1)
  (-0.5,1)
  (1.8,0) -- (2.8,1)
  (1.8,0) -- (2.8,-1)
  (-1,-1) -- (0,0)
  (-1,1) -- (0,0)
  (0,0) -- (1.8,0); 
\end{tikzpicture}
\caption{A crossing of a lamination curve $\gamma_i$ (in blue) with an edge $e_j$ of the triangulation contributes to the shear coordinates $+1$ if the crossing is as on the left and $-1$ if the crossing is as on the right.} \label{intsgn}
\end{center}
\end{figure}

\begin{definition}
Let $\Lambda=(\lambda_1,\dots,\lambda_l)$ be a multi-lamination on ${\bf S}$. The extended mutation matrix $B(\mathcal{T},\Lambda)$ is given by the $m\times (m+l)$-matrix with
\begin{itemize}
\item the upper $m\times m$-submatrix is given by the signed adjacency matrix $B(\mathcal{T})$, 
\item the $(m+l)$-th row of $B(\mathcal{T},\Lambda)$ is given by the shear coordinates $v_{\lambda_i,\mathcal{T}}$ of $\lambda_i$ with respect to $\mathcal{T}$, for $1\leq i \leq l$.  
\end{itemize}
The $c$-matrix $C(\mathcal{T},\Lambda)$ is the $m\times l$ submatrix of $B(\mathcal{T},\Lambda)$ consisting of the columns $m+1,\dots,m+l$.  
\end{definition}

\begin{theorem}[$\!\!${\cite[Theorem 13.5]{FT12}}]\label{FTthm}
If two trivalent spanning graphs $\mathcal{T},\mathcal{T}'$ of ${\bf S}$ are related by the flip of an edge, then $B(\mathcal{T},\Lambda)$ and $B(\mathcal{T}',\Lambda)$ are related by matrix mutation.
\end{theorem}

\subsection{The geometric model}

We fix a base commutative ring $k$. Let $\mathscr{G}_\mathcal{T}$ be the relative Ginzburg algebra of \Cref{gqdef}. We spell out part of the geometric model for $\mathcal{D}(\mathscr{G}_\mathcal{T})$ obtained from pure matching curves in ${\bf S}\backslash M$ with local value $\phi^*(k)\in \on{RMod}_{k[t_1]}$. Such a curve $\gamma$ in ${\bf S}\backslash M$ is composed of (possibly infinitely many) pure segments of the following two types such that $\gamma$ intersects $\partial {\bf S}\backslash M$ and the vertices of $\mathcal{T}$ only at the endpoints (possibly none).

\begin{figure}[h!]
\begin{center}
\begin{tikzpicture}
\node (1) at (-2.7,1.5){\large $1)$};
\node (0) at (0,0){};
\fill (0) circle (0.1);
\draw[color=ao][very thick] 
(-2.25,-0.75)--(2.25,-0.75)
(0,1.5) --(-2.25,-0.75)
(0,1.5) --(2.25,-0.75);
\draw[very thick]
(0,0)--(0,-1.35)
(0,0)--(1.35,1.2)
(0,0)--(-1.35,1.2);
\node (2) at (-2.25,-0.75){};
\node (3) at (2.25,-0.75){};
\node (4) at (0,1.5){};
\fill[color=orange] (2) circle (0.1);
\fill[color=orange] (3) circle (0.1);
\fill[color=orange] (4) circle (0.1);

\draw[color=blue][very thick] plot [smooth] coordinates {(0,0) (0.15,-0.075) (0.18,-0.75)};
\end{tikzpicture}
\quad\quad
\begin{tikzpicture}
\node (1) at (-2.7,1.5){\large $2)$};
\node (0) at (0,0){};
\fill (0) circle (0.1);
\draw[color=ao][very thick] 
(-2.25,-0.75)--(2.25,-0.75)
(0,1.5) --(-2.25,-0.75)
(0,1.5) --(2.25,-0.75);
\draw[very thick]
(0,0)--(0,-1.35)
(0,0)--(1.35,1.2)
(0,0)--(-1.35,1.2);
\node (2) at (-2.25,-0.75){};
\node (3) at (2.25,-0.75){};
\node (4) at (0,1.5){};
\fill[color=orange] (2) circle (0.1);
\fill[color=orange] (3) circle (0.1);
\fill[color=orange] (4) circle (0.1);

\draw[color=blue][very thick] plot [smooth] coordinates {(0.93,0.57) (0.255,0) (0.18,-0.75)};
\end{tikzpicture}
\caption{An ideal triangle of the ideal triangulation (in green), containing a pure matching curve of the first type on the left and a pure matching curve of the second type on the right (each in blue). The dual ribbon graph is in black. The orange points denote marked points (possibly punctures) and are removed in ${\bf S}\backslash M$.}\label{fig:curves3gon}
\end{center}
\end{figure}

The pure segments of the first type start at a vertex of the ribbon graph and end on an edge of the triangulation. The pure segments of the second type start and end on edges of the triangulation and only wrap around the vertex of the ribbon graph by one step in the clockwise or counterclockwise direction. Stated differently, we simply consider the segments given by those in \Cref{fig:curves3gon}, up to a rotation of the triangle. Note that in \Cref{fig:curves3gon}, two edges of the ideal triangle may coincide in the case that the triangle is self-folded (or equivalently an edge of $\mathcal{T}$ is a loop). 

Homotopy classes of pure matching curves, relative $\partial {\bf S}\backslash M$ and the vertices of $\mathcal{T}$, are in bijection with homotopy classes of curves in ${\bf S}$ relative $\partial {\bf S}\backslash M$ which do not cut out any discs in ${\bf S}\backslash M$, see \Cref{mcprop1}.

\begin{example}\label{geomexp}
Let $e$ be an internal edge of the ribbon graph $\mathcal{T}$. If $e$ is not a loop, we let $\gamma_e$ denote the pure matching curve which traces along $e$. If $e$ is a loop, we let $\gamma_e$ denote the pure matching curve which traces along $e$ (in any direction, e.g.~clockwise) and then traces along the other edge incident to the same vertex as $e$. We depict $\gamma_e$ in the these two cases as follows.
\begin{center}
\begin{tikzpicture}
  \draw[color=ao, very thick]
    (0, 3) -- (0, 0)
    (3, 0) -- (0, 0)
    (3, 3) -- (3, 0)
    (0, 3) -- (3, 3)
    (0, 3) -- (3, 0);
  \draw[very thick]
   (1,1) -- (2, 2)
   (1,1) -- (1,-0.45)
   (1,1) -- (-0.45,1)
   (2, 2) -- (3.45, 2)
   (2, 2) -- (2, 3.45);
  \draw[color=blue][very thick] plot [smooth] coordinates {(1,1) (1.3,1) (2,1.7) (2, 2)};
  \node (0) at (2, 2){};
  \node (1) at (1, 1){};
  \node (2) at (1.5,1.8){$e$};
  \node () at (2.1,1.35){$\gamma_e$};
  \fill (0) circle (0.1);
  \fill (1) circle (0.1);

\node (2) at (0,0){};
\node (3) at (3,0){};
\node (4) at (0,3){};
\node (5) at (3,3){};
\fill[color=orange] (2) circle (0.1);
\fill[color=orange] (3) circle (0.1);
\fill[color=orange] (4) circle (0.1);
\fill[color=orange] (5) circle (0.1);

\draw[color=ao, very thick]
(5,1.5)--(9.5,1.5)--(7.25,0)--(5,1.5)
(7.25,3)--(7.25,2.175);
\draw[color=ao, very thick] (9.5,1.5) arc [start angle=0, end angle=180, x radius=2.25, y radius=1.5];
\draw[very thick]
(5.45,0.45)--(7.25,0.75)--(9,0.45)
(7.25,0.75)--(7.25,1.8);
\draw[very thick] (8,2.2125) arc [start angle=0, end angle=360, x radius=0.75, y radius=0.4125];
\draw[color=blue][very thick] (7.25,1.8) arc [start angle=270, end angle=0, x radius=1.05, y radius=0.525];
\draw[color=blue][very thick] plot [smooth] coordinates {(8.3,2.325) (8.15, 1.95)  (7.25,0.75)};
\node () at (7.85,2.175){$e$};
\node () at (8.45,1.8){$\gamma_e$};
\node (2) at (5,1.5){};
\node (3) at (9.5,1.5){};
\node (4) at (7.25,0){};
\node (5) at (7.25,2.175){};
\fill (7.25,0.75) circle (0.1);
\fill (7.25,1.8) circle (0.1);
\fill[color=orange] (2) circle (0.1);
\fill[color=orange] (3) circle (0.1);
\fill[color=orange] (4) circle (0.1);
\fill[color=red] (5) circle (0.1);
\end{tikzpicture}
\end{center}
\end{example}

\begin{proposition}\label{geomprop}
There exists a full subcategory of $\mathcal{D}(\mathscr{G}_\mathcal{T})$ whose objects are labeled $M_{\gamma}$, where $\gamma$ is a pure matching curve considered modulo orientation. Given two such curves $\gamma_1\neq \gamma_2$, the $k$-module $\on{Ext}^\ast_{\mathcal{D}(\mathscr{G}_\mathcal{T})}(M_{\gamma_1},M_{\gamma_2})$ is given by the direct sum of the free $k$-modules obtained by counting intersections described as below.
\end{proposition}

\begin{proof}
Using the equivalence of $\infty$-categories $\mathcal{D}(\mathscr{G}_\mathcal{T})\simeq \mathcal{H}(\mathcal{T},\mathcal{F}_\mathcal{T}(k))$ of \Cref{gluethm}, the proposition is shown in more generality in \Cref{sec:objsfromcurves,sec:homs}. To match the descriptions of the $\on{Ext}$-groups, we also use \Cref{endomorlem}. If $\gamma$ is closed, we further need to specify a rank $a\geq 1$ and a monodromy equivalence, to have an associated object $M_{\gamma}$. We set $a=1$ and the monodromy equivalence to be arbitrary.
\end{proof}

\noindent {\bf Singular intersections}

For each singular intersection between $\gamma_1$ and $\gamma_2$ as below, where the orange arrow goes in the counterclockwise direction, we have $k\subset \on{Ext}^1(\gamma_1,\gamma_2)$ and $k\subset \on{Ext}^2(\gamma_2,\gamma_1)$.
\begin{center}
\begin{tikzpicture}
\draw[color=ao][very thick] 
(-2.25,-0.75)--(2.25,-0.75)
(0,1.5) --(-2.25,-0.75)
(0,1.5) --(2.25,-0.75);
\draw[very thick]
(0,0)--(0,-1.35)
(0,0)--(1.35,1.2)
(0,0)--(-1.35,1.2);
\node (2) at (-2.25,-0.75){};
\node (3) at (2.25,-0.75){};
\node (4) at (0,1.5){};
\fill[color=orange] (2) circle (0.1);
\fill[color=orange] (3) circle (0.1);
\fill[color=orange] (4) circle (0.1);

\draw[color=blue][very thick] plot [smooth] coordinates {(0,0) (0.15,-0.075) (0.18,-0.75)};
\draw[color=blue][very thick] plot [smooth] coordinates {(0,0) (-0.15,-0.075) (-0.98,0.53)};
\node () at (0.375,-0.3){$\gamma_2$};
\node () at (-0.9,0.225){$\gamma_1$};
\node (1) at (-0.75,0.3){};
\node (2) at (0,-0.45){} edge [<-, bend left, color =orange, very thick] (1);
\node (0) at (0,0){};
\fill (0) circle (0.1);
\end{tikzpicture}
\end{center}
If $\gamma_1$ and $\gamma_2$ have a singular intersection such that both $\gamma_1$ and $\gamma_2$ exit the ideal triangle at the same edge $f$ the above description of the $\on{Ext}$-groups has to be adapted as follows. Choosing $\gamma_1$ and $\gamma_2$ with the minimal number of intersections, we instead have $k\subset \on{Ext}^0(\gamma_1,\gamma_2)$ and $k\subset \on{Ext}^3(\gamma_2,\gamma_1)$, if the intersection of $\gamma_2$ with $f$ follows the intersection of $\gamma_1$ with $f$ in the counterclockwise direction.

\vspace{1em}

\noindent {\bf Crossings}

For each crossing of $\gamma_1$ and $\gamma_2$ as below, where the orange arrows go in the counterclockwise direction, we have $k\subset \on{Ext}^0(\gamma_1,\gamma_2),\on{Ext}^2(\gamma_1,\gamma_2)$ and $k\subset \on{Ext}^1(\gamma_2,\gamma_1),\on{Ext}^3(\gamma_2,\gamma_1)$.
\begin{center}
\begin{tikzpicture}
  \draw[color=ao, very thick]
    (0, 3) -- (0, 0)
    (3, 0) -- (0, 0)
    (3, 3) -- (3, 0)
    (0, 3) -- (3, 3)
    (0, 3) -- (3, 0);
  \draw[very thick]
   (1,1) -- (2, 2)
   (1,1) -- (1,-0.45)
   (1,1) -- (-0.45,1)
   (2, 2) -- (3.45, 2)
   (2, 2) -- (2, 3.45);
   \node (0) at (2, 2){};
  \node (1) at (1, 1){};
  \fill (0) circle (0.1);
  \fill (1) circle (0.1);

\node (2) at (0,0){};
\node (3) at (3,0){};
\node (4) at (0,3){};
\node (5) at (3,3){};
\fill[color=orange] (2) circle (0.1);
\fill[color=orange] (3) circle (0.1);
\fill[color=orange] (4) circle (0.1);
\fill[color=orange] (5) circle (0.1);

  \draw[color=blue][very thick] plot [smooth] coordinates {(1,1) (1.2,1) (2,1.8) (2, 2)};
   \draw[color=blue][very thick] plot [smooth] coordinates {(0,1.35) (3, 1.65)};
  \node (0) at (2, 2){};
  \node (1) at (1, 1){};
  \node () at (1.5,0.75){$\gamma_2$};
  \node () at (0.75,1.65){$\gamma_1$};
  \fill (0) circle (0.075);
  \fill (1) circle (0.075);
  \node (3) at (0.7,1.4){};
  \node (4) at (1.2,0.6){} edge [<-, bend left, color=orange, very thick] (3);
  \node (5) at (2.4,1.553){};
  \node (6) at (1.75,2.4){} edge [<-, bend left, color=orange, very thick] (5);
  \node (2) at (0,0){};
\end{tikzpicture}
\end{center}

\noindent {\bf Directed boundary intersections}

For each intersection of $\gamma_1,\gamma_2$ with the same boundary component of ${\bf S}\backslash M$ as below, where the orange arrow goes in the clockwise direction, we have $k\subset \on{Ext}^0(\gamma_1,\gamma_2),\on{Ext}^2(\gamma_1,\gamma_2)$.
\begin{center}
\begin{tikzpicture}
\node (0) at (0,0){};
\fill (0) circle (0.1);
\draw[color=ao][very thick] 
(-2.25,-0.75)--(2.25,-0.75)
(0,1.5) --(-2.25,-0.75)
(0,1.5) --(2.25,-0.75);
\draw[very thick]
(0,0)--(0,-1.35)
(0,0)--(1.35,1.2)
(0,0)--(-1.35,1.2);
\node (2) at (-2.25,-0.75){};
\node (3) at (2.25,-0.75){};
\node (4) at (0,1.5){};
\fill[color=orange] (2) circle (0.1);
\fill[color=orange] (3) circle (0.1);
\fill[color=orange] (4) circle (0.1);

\draw[color=blue][very thick] plot [smooth] coordinates {(0.93,0.57) (0.255,0) (0.18,-0.75)};
\draw[color=blue][very thick] plot [smooth] coordinates {(0.645,0.855) (0,0.405) (-0.645,0.855)};
\node ()  at (-0.15,0.825){$\gamma_1$};
\node () at (0.45,-0.3){$\gamma_2$};
\node (1) at (0.6,1.2){};
\node (2) at (1.35,0.45){} edge [<-, bend right, color=orange, very thick] (1);
\end{tikzpicture}
\end{center}

\begin{remark}
The above formulas for the morphisms from intersections reflect the \textit{relative} $3$-Calabi--Yau structure of $\mathcal{D}(\mathscr{G}_\mathcal{T})$: morphisms arising from singular intersections or crossings satisfy the usual $3$-Calabi--Yau duality. Boundary intersections are directed and the $3$-Calabi--Yau duality fails for these. See also \cite[Prop.~5.19]{Chr22b} for a systematic explanation of the version of Calabi--Yau duality arising from a relative Calabi--Yau structure.
\end{remark}

\subsection{The categorical description}\label{subsec:categorical_description_extended_mut_matrix}

As above, we fix a base commutative ring $k$. 

\begin{notation}
Consider a lamination $\lambda$ of ${\bf S}$, i.e.~a collection of lamination curves. Using \Cref{mcprop1}, we can consider each lamination curve as a pure matching curve. We denote by $M_\lambda=\bigoplus_{\gamma\in \lambda} M_{\gamma}\in \mathcal{D}(\mathscr{G}_\mathcal{T})$ the direct sum of the objects associated to the lamination curves, see \Cref{geomprop}. For $e$ an internal edge of $\mathcal{T}$, we denote by $M_{\gamma_e}\in \mathcal{D}(\mathscr{G}_\mathcal{T})$ the object associated to the pure matching curve $\gamma_e$, see \Cref{geomexp}.
\end{notation}

\begin{definition}\label{catdef1}
Label the internal edges of $\mathcal{T}$ by $e_1,\dots,e_m$. Let $\Lambda=(\lambda_1,\dots,\lambda_l)$ be a multi-lamination on ${\bf S}$. The categorical extended mutation matrix $\hat{B}(\mathcal{T},\Lambda)=(\hat{b}_{i,j})$ is given by the $m\times (m+l)$-matrix with 
\begin{itemize}
\item $\hat{b}_{i,j}=\chi \on{Ext}^\ast(M_{\gamma_{e_i}},M_{\gamma_{e_j}})$ for $1\leq i,j\leq m$ and
\item $\hat{b}_{i,m+j}=\frac{1}{2}\chi \on{Ext}^\ast(M_{\gamma_{e_i}},M_{\lambda_j})$ for $1\leq i\leq m$ and $1\leq j\leq l$,
\end{itemize}
where $\chi$ denotes the Euler-characteristic, see \Cref{sec:lincats}.
The categorical $c$-matrix $\hat{C}(\mathcal{T},\Lambda)$ is the $m\times l$ submatrix of $\hat{B}(\mathcal{T},\Lambda)$ consisting of the columns $m+1,\dots,m+l$.  
\end{definition}

\begin{remark}
We will see below that in the setting of \Cref{catdef1} 
\[ \chi \on{Ext}^\ast(M_{\gamma_{e_i}},M_{\gamma_{e_j}})=\on{dim}_k\on{Ext}^2(M_{\gamma_{e_i}},M_{\gamma_{e_j}})-\on{dim}_k\on{Ext}^1(M_{\gamma_{e_i}},M_{\gamma_{e_j}})\] 
of $e_i.e_j$ are not loops and 
\[ \frac{1}{2}\chi \on{Ext}^\ast(M_{\gamma_{e_i}},M_{\lambda_j})=\on{dim}_k\on{Ext}^2(M_{\gamma_{e_i}},M_{\lambda_j})-\on{dim}_k\on{Ext}^1(M_{\gamma_{e_i}},M_{\lambda_j})\,.\]
\end{remark}

\begin{theorem}\label{bmatthm}
Let ${\bf S}$ be a marked surface with a multi-lamination $\Lambda$ and let $\mathcal{T}$ be a trivalent spanning graph of ${\bf S}$. The extended mutation matrices $\hat{B}(\mathcal{T},\Lambda)$ and $B(\mathcal{T},\Lambda)$ are identical. 
\end{theorem}

\begin{corollary}\label{bmatcor}
Let ${\bf S}$ be a marked surface with a multi-lamination $\Lambda$. 
Suppose that two trivalent spanning graphs $\mathcal{T},\mathcal{T}'$ of ${\bf S}$ are related by the flip of an edge. The categorical extended mutation matrices $\hat{B}(\mathcal{T},\Lambda)$ and $\hat{B}(\mathcal{T}',\Lambda)$ are related by matrix mutation.
\end{corollary}
\begin{proof}
Combine \Cref{FTthm,bmatthm}.
\end{proof}

\begin{proof}[Proof of \Cref{bmatthm}.]
We first show that the upper $m\times m$-submatrix of $\hat{B}(\mathcal{T},\Lambda)$ agrees with the signed adjacency matrix $B(\mathcal{T})$. Let $e_i,e_j$ be two edges of $\mathcal{T}$. If $e_i=e_j$, it is obvious that $\hat{b}_{i,i}=0=b_{i,i}$. We can thus assume that $e_i\neq e_j$. Assume that $e_i,e_j$ are not a loop and a non-loop incident to the same vertex (i.e.~dual to the two edges of a self-folded ideal triangle). By \Cref{geomprop}, we know that $\on{dim}_k\on{Ext}^2(M_{\gamma_{e_i}},M_{\gamma_{e_j}})$ counts the number of singular intersections where $\gamma_{e_j}$ follows $\gamma_{e_i}$ in the clockwise order. This number is equal to the number of arrows from $e_i$ to $e_j$ in $Q_\mathcal{T}$. Similarly, $\on{dim}_k\on{Ext}^1(M_{\gamma_{e_i}},M_{\gamma_{e_j}})$ is equal to the number of arrows from $e_j$ to $e_i$ in $Q_\mathcal{T}$. All other $\on{Ext}$-groups vanish and it follows that $\hat{b}_{i,j}=\chi \on{Ext}^\ast(M_{\gamma_{e_i}},M_{\gamma_{e_j}})=b_{i,j}$. In the case that $e_i,e_j$ are a loop and a non-loop incident to the same vertex, with $e_i$ the loop, we have $b_{i,j}=b_{j,i}=0$ and $\on{Ext}^\ast(\gamma_{e_i},\gamma_{e_j})=k\oplus k[-1]$ and $\on{Ext}^\ast(\gamma_{e_j},\gamma_{e_i})=k[-2]\oplus k[-3]$ so that also $\hat{b}_{i,j}=\hat{b}_{j,i}=0$.

We continue by showing that the $c$-matrices are identical.
Using the additivity of $\on{Ext}^*$, it suffices to verify that for each lamination curve $\gamma$ and each edge $e_i$ there exists an equality  
\begin{equation}\label{inteq} 
\frac{1}{2}\chi \on{Ext}^\ast(M_{e_i},M_{\gamma})=(e_i,\gamma)\,.
\end{equation} 
We begin with the case that $e_i$ is not a loop. \Cref{geomprop} shows that $\on{Ext}^\ast(M_{\gamma_{e_i}},M_{\gamma})$ is the direct sum of contributions arising from crossings of $e_i$ and $\gamma$. If a crossing of $\gamma$ and $e_i$ is as on the left in \Cref{intsgn}, then $\on{Ext}^\ast(M_{\gamma_{e_i}},M_{\gamma})\simeq k\oplus k[-2]$ and $\frac{1}{2}\chi\on{Ext}^\ast(M_{\gamma_{e_i}},M_{\gamma})=1$ and the intersection thus contributes the same amount to both sides of \eqref{inteq}. Similarly, if the crossing of $\gamma$ and $\gamma_{e_i}$ is as on the right in \Cref{intsgn}, then $\on{Ext}^\ast(M_{\gamma_{e_i}},M_{\gamma})=k[-1]\oplus k[-3]$ and the intersection also contributes with $-1$ to both sides of \eqref{inteq}. 

Consider now the case that $e_i$ is a loop and let $e_i'$ be the unique other edge of $\mathcal{T}$ incident to $e_i$. If $\gamma$ does not have an infinite end spiraling around the puncture at which $e_i$ lies, then both sides of \eqref{inteq} vanish. We thus assume that such a spiraling infinite end exists. Consider the vertex $v$ incident to $e_i'$ at which $e_i$ does not lie and consider the two edges $e_1\neq e_2$ incident to $v$ such their cyclic (counterclockwise) order is given by $e_i',e_1,e_2,e_i'$. There are four possible arrangements: the end of $\gamma_i$ either arrives at $e_i'$ first passing along $e_1$ or $e_2$ and the infinite end either spirals clockwise or counterclockwise. In the clockwise case, one finds $\on{Ext}^\ast(M_{\gamma_{e_i}},M_{\gamma})\simeq k[-1]\oplus k[-3]$ if $\gamma$ passes along $e_1$ and $\on{Ext}^\ast(M_{\gamma_{e_i}},M_{\gamma})\simeq k\oplus k[-1]\oplus k[-2]\oplus k[-3]$ if $\gamma$ passes along $e_2$ (in this case, there are two crossings). In the counterclockwise case, one finds $\on{Ext}^\ast(M_{\gamma_{e_i}},M_{\gamma})\simeq 0$ if $\gamma$ passes along $e_1$ and $\on{Ext}^\ast(M_{\gamma_{e_i}},M_{\gamma})\simeq k \oplus k[-2]$ if $\gamma$ passes along $e_2$. In each case, we thus find as desired
\[\frac{1}{2}\chi\on{Ext}^\ast(M_{\gamma_{e_i}},M_{\gamma})= \frac{1}{2}\chi\on{Ext}^\ast(M_{\gamma_{e_i'}},M_{\tilde{\gamma}})= (e_i',\tilde{\gamma})=(e_i,\gamma)\,,\]
where $\tilde{\gamma}$ is as in \Cref{intsgndef}. We have shown that the $c$-matrices are also identical, concluding the proof.
\end{proof}

\bibliographystyle{plain}
\bibliography{biblio} 
~\\
\textsc{Universität Bonn, Mathematisches Institut, Endenicher Allee 60, 53115 Bonn, Germany}\\
\textit{Email address:} \texttt{christ@math.uni-bonn.de}
\end{document}